\documentclass[reqno,11pt]{amsart}
\usepackage{epsfig,amscd,amssymb,amsmath,amsfonts}

\usepackage{amsmath}
\usepackage{graphicx}
\usepackage{amsthm,color}
\usepackage{tikz}
\usetikzlibrary{graphs}
\usetikzlibrary{graphs,quotes}
\usetikzlibrary{decorations.pathmorphing}

\tikzset{snake it/.style={decorate, decoration=snake}}
\tikzset{snake it/.style={decorate, decoration=snake}}

\usetikzlibrary{decorations.pathreplacing,decorations.markings,snakes}

\newtheorem{theorem}{Theorem}[section]
\newtheorem{lemma}[theorem]{Lemma}
\newtheorem{proposition}{Proposition}[section]
\theoremstyle{definition}
\newtheorem{definition}[theorem]{Definition}

\newtheorem{example}[theorem]{Example}

\theoremstyle{remark}
\newtheorem{remark}[theorem]{Remark}

\numberwithin{equation}{section}

\newcommand{\exter}[2]{\Lambda^{S^2}_{#1}(2#2 + 1)}

\usepackage[margin=1.05in]{geometry}
\usepackage[colorlinks]{hyperref}
\usepackage{siunitx}

\makeatletter
\let\@wraptoccontribs\wraptoccontribs
\makeatother

\begin{document}

\title[Edge partitions of the complete graph and a determinant like function]{Edge partitions of the complete graph \\and a determinant like function }

\author{Steven R. Lippold}
\address{Department of Mathematics and Statistics, Bowling Green State University, Bowling Green, OH 43403 }
\email{steverl@bgsu.edu}


\author{Mihai D. Staic}
\address{Department of Mathematics and Statistics, Bowling Green State University, Bowling Green, OH 43403 }
\address{Institute of Mathematics of the Romanian Academy, PO.BOX 1-764, RO-70700 Bu\-cha\-rest, Romania.}

\email{mstaic@bgsu.edu}


\author{Alin Stancu}
\address{Department of Mathematics, Columbus State  University, Columbus, GA 31907}
\email{stancu\_alin1@columbusstate.edu}



\begin{abstract} In this paper we prove the case $dim(V_3)=3$ of a conjecture from \cite{sta2} about the exterior operad ${\Lambda}^{S^2}_{V_d}$. For this we introduce a collection of natural involutions  on the set of homogeneous cycle-free $d$-partitions of the complete graph $K_{2d}$, and show that these involutions correspond to the relations in ${\Lambda}^{S^2}_{V_d}(2d+1)$. When $d=3$ this correspondence allows us to give an explicit description of a determinant-like map and to settle the above mentioned conjecture. 
\end{abstract}

\maketitle


%

\section*{Introduction}

The exterior Graded-Swiss-Cheese (GSC) operad  ${\Lambda}^{S^2}_V$ was introduced by the second author in \cite{sta2}. The construction  was inspired by the work of Pirashvili on Higher Hochschild homology \cite{p} and Voronov on Swiss-Cheese operad \cite{vo}, as well as by the explicit description of $H_\bullet^{S^2}(A, A)$ from \cite{cs} and \cite{la}. While some of the properties of  ${\Lambda}^{S^2}_V$ are similar to those of the exterior algebra, there are also many notable differences. For example, there is no algebra structure on ${\Lambda}^{S^2}_V$ (which prompted the definition of a GSC-operad as the closest next best thing). 

Using the setting of GSC-operads one can describe ${\Lambda}^{S^2}_V$ as the quotient of a tensor GSC operad by a certain operad ideal. It was shown  in \cite{sta2} that if $V$ is finite dimensional then ${\Lambda}^{S^2}_V$ is also a finite dimensional graded vector space, more precisely $dim_k({\Lambda}^{S^2}_V(n))=0$, if $n>2dim_k(V)+1$. Moreover, if the dimension of $V_2$ is $2$ then  $dim_k({\Lambda}^{S^2}_{V_2}(5))=1$. In particular, there is a determinant like function $det^{S^2}:V_2^6\to k$ that satisfies an appropriate universality property. It was proved in \cite{sv} that the map  $det^{S^2}$ has a nice geometrical interpretation, essentially detecting when six ordered vectors $(v_{i,j})_{1\leq i<j\leq 4}$ determine the directions of the six sides (edges and diagonals) of a quadrilateral. 

It was conjectured in \cite{sta2} that if $dim_k(V_d)=d$, then $dim_k({\Lambda}^{S^2}_{V_d}(2d+1)=1$. One of the main result of this paper is checking this conjecture for the case $d=3$. In particular, we will get an explicit  description of a map $det^{S^2}:V_3^{15}\to k$ that has an appropriate universality property. We also present a road map on how the general case could be approached.

In the first section we recall the GSC-operad ${\Lambda}^{S^2}_V$ and review some of its properties from \cite{sta2}. While the language of GSC-operads is convenient to define ${\Lambda}^{S^2}_V$, it is not very helpful for explicit computations. For this reason, in this paper we will only look to the  graded vector space structure on  ${\Lambda}^{S^2}_V$. In order to make our presentation self contained, we give an explicit definition for the graded vector space  ${\Lambda}^{S^2}_V$,  striping down the entire algebraic formalism associated with the notion of GSC-operad. We end the first section by recalling a few general definitions and properties for edge partitions of simple undirected graphs. 
 
In the second section we define two actions on ${\Lambda}^{S^2}_{V_d}(2d+1)$ by the groups of permutations $S_d$ and $S_{2d}$. We introduce the distinguished element $E_d$ and study its invariance under these actions. One of the main results of this paper is to show that $\{E_3\}$ is a nontrivial basis for ${\Lambda}^{S^2}_{V_3}(7)$ (this is proven in the last section). The rest of the paper is dedicated to showing how  these purely algebraic results can be understood and proven by analyzing certain edge-partitions of the complete graph $K_{2d}$.

In the third section we discuss a correspondence between $d$-partitions of $K_n$ and a system of generators of  the tensor operad ${\mathcal T}^{S^2}_{V_d}(n+1)$. We develop a graphic calculus for these partitions in relation to ${\Lambda}^{S^2}_{V_d}(n+1)$ and show that in the case of ${\Lambda}^{S^2}_{V_d}(2d+1)$ only the homogeneous cycle-free $d$-partitions give non-trivial elements. 

The linchpin of this paper is Lemma \ref{keylemma}, a combinatorial result  which essentially states that on the set of homogeneous cycle-free $d$-partitions of the graph $K_{2d}$ there is a collection of intrinsic involutions. More precisely, for any cycle-free homogeneous $d$-partition $(\Gamma_1,\dots,\Gamma_d)$ of $K_{2d}$, and any three vertices $x$, $y$ and $z$, there exists a unique cycle-free homogeneous $d$-partition $(\Gamma_1,\dots,\Gamma_d)^{(x,y,z)}$ with identical edges, except on the face $(x,y,z)$, where they differ on at least two edges. After factoring to ${\Lambda}^{S^2}_{V_d}(2d+1)$ these involutions are compatible with the correspondence between generators for ${\mathcal T}^{S^2}_{V_d}(2d+1)$ and $d$-partitions of $K_{2d}$.  

In section four we use Lemma \ref{keylemma} to prove that the dimension of ${\Lambda}^{S^2}_{V_3}(7)$ is at most $1$.   Note that some of the results of section four could be used to prove that the dimension of ${\Lambda}^{S^2}_{V_d}(2d+1)$ is at most $1$ for all $d$. The missing piece is showing that every homogeneous cycle-free $d$-partition of $K_{2d}$ is equivalent with one that has a twin-star component. 

The main result in section five is Theorem \ref{th1A}, which states that there is a signature-like map defined on the set of homogeneous cycle-free $3$-partitions of $K_{6}$ with range in $\{\pm 1\}$, which is compatible with the involutions described above. As a direct application we can define a determinant-like map $det^{S_2}:V^{15}\rightarrow k$, which is $k$-multilinear and has the property that $det^{S^2}(v_{i,j})=0$, if there exist $1\leq x<y<z\leq 6$ such that $v_{x,y}=v_{x,z}=v_{y,z}$. This determinant-like map is essential in proving that $E_3$ is not trivial, and so the dimension of ${\Lambda}^{S^2}_{V_3}(7)$ is 1.

Another interesting consequence of  Theorem \ref{th1A} is that there is a natural orientation-like splitting of the set of homogeneous cycle-free 3-partitions of $K_6$. This is similar with the fact that the set of permutations $S_d$ can be split  into even and odd permutations. 

We also  reformulate the Conjecture from \cite{sta2} in terms of the action of a group $G_d^{cf,h}$ on the set of homogeneous cycle-free $d$-partitions of $K_{2d}$. If this group acts transitively on the above set of partitions, then the dimension of ${\Lambda}^{S^2}_{V_d}(2d+1)$ is at most 1. If there is a signature-like map defined on the set of homogeneous cycle-free $d$-partitions of $K_{2d}$,  which is compatible with the natural involutions, then the dimension of ${\Lambda}^{S^2}_{V_d}(2d+1)$ is at least 1. Such a signature map would imply the existence of a determinant-like map which generalizes those defined for $d=2$ and $3$.

Theorem \ref{th1A} was established by direct computations using MATLAB. There are $\num{756756}$ homogeneous 3-partitions of $K_6$, out of which $\num{66240}$ are cycle-free. 
In the Appendix we consider the natural action of $S_6\times S_3$ on the set of homogeneous cycle-free 3-partitions of $K_6$ and describe the 19 equivalence classes. One can then use this explicit description to give an alternative direct proof for Theorem \ref{th1A} that is not based on MATLAB computations.

\section{Preliminaries}

In this paper $k$ is a field such that $char(k)\neq 2$ and $char(k)\neq 3$. The tensor product $\otimes$ is over the field $k$. For every $d\in \mathbb{N}$ we consider a vector space  $V_d$  of dimension $d$, and  $\{e_1,e_2,...,e_d\}$ is a basis for $V_d$. Moreover we assume that if $d_1\leq d$ we have $V_{d_1}\subseteq V_{d}$ (that is $V_{d_1}$ is generated by $e_1,e_2, ...,e_{d_1}$).

\subsection{The operad ${\Lambda}^{S^2}_V$}
The GSC-operad ${\Lambda}^{S^2}_V$ was introduced in \cite{sta2} as the quotient of the tensor GSC-operad ${\mathcal T}^{S^2}_{V}$ by the ideal ${\mathcal E}^{S^2}_{V}$. While the operad language is very convenient to give a formal general definition, for computational purposes it is much easier to consider an intrinsic description of ${\Lambda}^{S^2}_V$, one that is using only linear algebra. 

Let $V$ be a finite dimensional vector space. For every $n\geq 0$ we denote
$${\mathcal T}^{S^2}_V(n+1)=V^{\otimes \frac{n(n-1)}{2}}=\begin{pmatrix}
k & V & \cdots & V&V\\
 & k & \cdots & V&V\\
 & &\ddots & \vdots & \vdots\\
 & & & k & V\\
 \otimes& & & & k\end{pmatrix}.$$
Since the tensor product is over $k$, the $k$'s on the diagonal do not play an important role, they are mostly for symmetry purpose. A simple tensor in ${\mathcal T}^{S^2}_V(n+1)$ will be denoted by
$$\begin{pmatrix}
1 & u_{1,2} &u_{1,3}& \cdots & u_{1,n-1}&u_{1,n}\\
 & 1 & u_{2,3}&\cdots & u_{2,n-1}&u_{2,n}\\
 &  & 1&\cdots & u_{3,n-1}&u_{3,n}\\
 & &&\ddots & \vdots & \vdots\\
 & & && 1 & u_{n-1,n}\\
 \otimes& & && & 1\end{pmatrix} \in V^{\otimes \frac{n(n-1)}{2}},$$
where $u_{i,j}\in V$.    A general element in ${\mathcal T}^{S^2}_V(n+1)$ is a sum of such simple tensors.

Formally, ${\mathcal E}^{S^2}_V$ is the GSC operad ideal generated by elements of the form $ \begin{pmatrix}
1& v&v\\
&1&v\\
\otimes& &1
\end{pmatrix}$ where $v\in V$. Since we want to avoid the operad structure in this paper,  we will give an  explicit description for the elements of ${\mathcal E}^{S^2}_V$.

For every $n\geq 0$ we take
${\mathcal E}^{S^2}_V(n+1)$ to be the subspace of $V^{\otimes \frac{n(n-1)}{2}}$ that  is linearly generated by simple tensors
$$\begin{pmatrix}
1 & u_{1,2} &u_{1,3}& \cdots & u_{1,n-1}&u_{1,n}\\
 & 1 & u_{2,3}&\cdots & u_{2,n-1}&u_{2,n}\\
 & &&\ddots & \vdots & \vdots\\
 & & && 1 & u_{n-1,n}\\
 \otimes& & && & 1\end{pmatrix}\in {\mathcal T}^{S^2}_V(n+1),$$ with the property that there exist $1\leq i<j<k\leq n$ such that $u_{i,j}=u_{i,k}=u_{j,k}$. We are now ready to reformulate the definition of  ${\Lambda}^{S^2}_V$.
\begin{definition} Let $V$ be a $k$ vector space. We define ${\Lambda}^{S^2}_V$ as a graded vector space with the component in degree $n+1$ defined  as the quotient vector space
$${\Lambda}^{S^2}_V(n+1)=\frac{{\mathcal T}^{S^2}_V(n+1)}{{\mathcal E}^{S^2}_V(n+1)},$$
 for every $n\geq 0$. \label{defla}
\end{definition}

\begin{remark}
As pointed out above, ${\Lambda}^{S^2}_V$  is not just a graded vector space, it has a GSC operad structure, but for the purpose of this paper this definition will suffice. The choice of grading from Definition \ref{defla} is consistent with the GSC operad structure. There is an argument for shifting it down by one (i.e. making the degree of ${\Lambda}^{S^2}_V(n+1)$ be $n$ not $n+1$),  that would make it more consistent with the usual grading of the exterior algebra.
\end{remark}

\begin{remark}
It was proved in \cite{sta2} that if $dim(V)=d$, then $dim({\Lambda}^{S^2}_V(n+1))=0$ for all $n>2d$.
In the same paper it was conjectured that  $dim({\Lambda}^{S^2}_V(2d+1))=1$, which was shown to be true in the case $dim(V)=2$. One of the main results this paper is proving this conjecture for  the case $dim(V)=3$.
\end{remark}

\begin{remark}
Let $V_2$ be a vector space with $dim(V_2)=2$. As a byproduct of the fact that $dim({\Lambda}^{S^2}_{V_2}(5))=1$, we get the existence of a determinant like map $det^{S^2}:V_2^6\to k$. More precisely, for $v_{i,j}=\alpha_{i,j}e_1+\beta_{i,j}e_2\in V_2$ we have
\begin{eqnarray} &det^{S^2}((v_{i,j})_{1\leq i<j\leq 4})=&\nonumber\\
&\alpha_{1,2}\alpha_{2,3}\alpha_{3,4}\beta_{1,3}\beta_{2,4}\beta_{1,4}+
\alpha_{1,2}\beta_{2,3}\alpha_{3,4}\beta_{1,3}\beta_{2,4}\alpha_{1,4}+
\alpha_{1,2}\beta_{2,3}\beta_{3,4}\alpha_{1,3}\alpha_{2,4}\beta_{1,4}&\nonumber\\
&+\beta_{1,2}\beta_{2,3}\alpha_{3,4}\alpha_{1,3}\alpha_{2,4}\beta_{1,4}+
\beta_{1,2}\alpha_{2,3}\beta_{3,4}\beta_{1,3}\alpha_{2,4}\alpha_{1,4}+
\beta_{1,2}\alpha_{2,3}\beta_{3,4}\alpha_{1,3}\beta_{2,4}\alpha_{1,4}&\\
&-\beta_{1,2}\beta_{2,3}\beta_{3,4}\alpha_{1,3}\alpha_{2,4}\alpha_{1,4}-
\beta_{1,2}\alpha_{2,3}\beta_{3,4}\alpha_{1,3}\alpha_{2,4}\beta_{1,4}-
\beta_{1,2}\alpha_{2,3}\alpha_{3,4}\beta_{1,3}\beta_{2,4}\alpha_{1,4}&\nonumber\\
&-\alpha_{1,2}\alpha_{2,3}\beta_{3,4}\beta_{1,3}\beta_{2,4}\alpha_{1,4}-
\alpha_{1,2}\beta_{2,3}\alpha_{3,4}\alpha_{1,3}\beta_{2,4}\beta_{1,4}-
\alpha_{1,2}\beta_{2,3}\alpha_{3,4}\beta_{1,3}\alpha_{2,4}\beta_{1,4}.& \nonumber
\end{eqnarray}
Essentially, $det^{S^2}$ is a the unique nontrivial multilinear map defined on $V_2^6$ with the property that $det^{S^2}((v_{i,j})_{1\leq i<j\leq 4})=0$,  if there exist $1\leq x<y<z\leq 4$ such that $v_{x,y}=v_{x,z}=v_{y,z}$.

The $det^{S^2}$ map has an interesting geometrical interpretation similar with the usual determinant map.
More precisely, it was proved in \cite{sv}  that $det^{S^2}((v_{i,j})_{1\leq i<j\leq 4})=0$ if and only if
there exist $\lambda_{i,j}\in k$ for  $1\leq i<j\leq 4$  not all zero, such that for all $1\leq x< y< z\leq 4$ we have
$$\lambda_{x,y}v_{x,y}+\lambda_{y,z}v_{y,z}+\lambda_{z,x}v_{z,x}=0,$$
with the convention that $\lambda_{i,j}=\lambda_{j,i}$, and $v_{i,j}=-v_{j,i}$ for all $1\leq j<i\leq 4$. This result is the analog of the fact that for two vectors, $v_1=\alpha_1e_1+\beta_1e_2,\;v_2=\alpha_2e_1+\beta_2e_2\in V_2$, we have $det\begin{pmatrix}
\alpha_1& \alpha_2\\
\beta_1&\beta_2
\end{pmatrix}=0$ if and only if  there exist $\lambda_1, \lambda_2\in k$ not both zero such that $\lambda_1v_1+\lambda_2v_2=0$ (i.e. $v_1$ and $v_2$ are collinear).
\label{rem14}
\end{remark}
\begin{remark} The map $det^{S^2}$ defined above does not seem to fit in the language of hyper-determinants from 
\cite{gkz}. One can still hope to use some of the techniques developed there in order to get a better understanding of our construction. 
\end{remark}

Next we recall a few technical results from \cite{sta2} that will be useful in this paper.
\begin{lemma} Let $V$ be a vector space, $\omega \in V$, and $n\geq 3$. Take  $$f=\begin{pmatrix}
1& v_{1,2}&...&v_{1,n-1}&v_{1,n}\\
& 1&...&v_{2,n-1}&v_{2,n}\\
& &...&.&.\\
& & &1&v_{n-1,n}\\
\otimes& & & &1
\end{pmatrix}\in {\mathcal T}^{S^2}_V(n+1),$$ such that $\omega$ appears at least $n$ times amongst  the elements in the set $\{v_{i,j}\;\vert \; 1\leq i<j\leq n\}$, then $\hat{f}=0\in {\Lambda}^{S^2}_V(n+1)$. \label{lemma2}
\end{lemma}
We have the following refinement of Lemma 3.4 from \cite{sta2}.
\begin{lemma} Let $\{e_1,...,e_d\}$ be a basis for $V_d$. If the $char(k)$ is not $2$ or $3$, then the ideal ${\mathcal E}^{S^2}_{V_d}$ is generated (as an operad ideal) by any of the following sets: \\ \label{relemma1}
1) $\{ \begin{pmatrix}
1& v&v\\
&1&v\\
\otimes & &1
\end{pmatrix}\vert {\rm ~ for ~ all} ~v\in V\}$.\\
2) $\{ \begin{pmatrix}
1& e_i&e_j\\
&1&e_k\\
\otimes& &1
\end{pmatrix}+\begin{pmatrix}
1& e_i&e_k\\
&1&e_j\\
\otimes& &1
\end{pmatrix}+\begin{pmatrix}
1& e_j&e_i\\
&1&e_k\\
\otimes& &1
\end{pmatrix}+\begin{pmatrix}
1& e_k&e_i\\
&1&e_j\\
 \otimes& &1
\end{pmatrix}+\begin{pmatrix}
1& e_j&e_k\\
&1&e_i\\
\otimes& &1
\end{pmatrix}+\begin{pmatrix}
1& e_k&e_j\\
&1&e_i\\
\otimes& &1
\end{pmatrix} \vert \\{\rm ~ for ~ all} ~1\leq i\leq j\leq k\leq d\}$.\\
\end{lemma}
\begin{proof} It follows from the linearity of ${\Lambda}^{S^2}_{V_d}$.
\end{proof}
A direct consequence of Lemma \ref{relemma1} is the following explicit description for ${\mathcal E}^{S^2}_{V_d}$. Note that in order to avoid the language of GSC operads one can  use Remark \ref{remba} as the definition of ${\mathcal E}^{S^2}_{V_d} (n+1)$
\begin{remark} Let $\{e_1,...,e_d\}$ be a basis for $V_d$, then ${\mathcal E}^{S^2}_{V_d}(n+1)$ is linearly generated (as a vector space) by elements described in equations (\ref{equ0}), (\ref{equ1}) and (\ref{equ2}),
where each tensor matrix is of type $n\times n$, the $*$ can be any element in $V_d$ (we still have $1$'s on the diagonal of course), and the boxed vectors $\boxed{e_l}$ are in the position $(x,y)$, $(x,z)$ and $(y,z)$  where $1\leq x<y<z\leq n$. Finally, one should notice that because of the linearity of the tensor product we may assume that the $\ast$ entries in equations (\ref{equ0}), (\ref{equ1}), and (\ref{equ2}) are actually elements of the basis  $\{e_1,...,e_d\}$.
\label{remba}
\begin{eqnarray}\begin{pmatrix}
1&*&*&*&*&*\\
&*&\boxed{e_i}&*&\boxed{e_i}&*\\
&&*&*&*&*\\
&&&*&\boxed{e_i}&*\\
&&&&*&*\\
\otimes&&&&&1
\end{pmatrix}\label{equ0}
\end{eqnarray}
for all $1\leq i\leq d$,
\begin{eqnarray}\begin{pmatrix}
1&*&*&*&*&*\\
&*&\boxed{e_i}&*&\boxed{e_i}&*\\
&&*&*&*&*\\
&&&*&\boxed{e_j}&*\\
&&&&*&*\\
\otimes&&&&&1
\end{pmatrix}+
\begin{pmatrix}
1&*&*&*&*&*\\
&*&\boxed{e_i}&*&\boxed{e_j}&*\\
&&*&*&*&*\\
&&&*&\boxed{e_i}&*\\
&&&&*&*\\
\otimes&&&&&1
\end{pmatrix}+\begin{pmatrix}
1&*&*&*&*&*\\
&*&\boxed{e_j}&*&\boxed{e_i}&*\\
&&*&*&*&*\\
&&&*&\boxed{e_i}&*\\
&&&&*&*\\
\otimes&&&&&1
\end{pmatrix}\label{equ1}
\end{eqnarray}
for all $1\leq i<j\leq d$, and
\begin{eqnarray}\begin{pmatrix}
1&*&*&*&*&*\\
&*&\boxed{e_i}&*&\boxed{e_j}&*\\
&&*&*&*&*\\
&&&*&\boxed{e_k}&*\\
&&&&*&*\\
\otimes&&&&&1
\end{pmatrix}+
\begin{pmatrix}
1&*&*&*&*&*\\
&*&\boxed{e_i}&*&\boxed{e_k}&*\\
&&*&*&*&*\\
&&&*&\boxed{e_j}&*\\
&&&&*&*\\
\otimes&&&&&1
\end{pmatrix}+\begin{pmatrix}
1&*&*&*&*&*\\
&*&\boxed{e_j}&*&\boxed{e_i}&*\\
&&*&*&*&*\\
&&&*&\boxed{e_k}&*\\
&&&&*&*\\
\otimes&&&&&1
\end{pmatrix}\nonumber\\
\label{equ2} \\
+\begin{pmatrix}
1&*&*&*&*&*\\
&*&\boxed{e_k}&*&\boxed{e_i}&*\\
&&*&*&*&*\\
&&&*&\boxed{e_j}&*\\
&&&&*&*\\
\otimes&&&&&1
\end{pmatrix}+
\begin{pmatrix}
1&*&*&*&*&*\\
&*&\boxed{e_j}&*&\boxed{e_k}&*\\
&&*&*&*&*\\
&&&*&\boxed{e_i}&*\\
&&&&*&*\\
\otimes&&&&&1
\end{pmatrix}+\begin{pmatrix}
1&*&*&*&*&*\\
&*&\boxed{e_k}&*&\boxed{e_j}&*\\
&&*&*&*&*\\
&&&*&\boxed{e_i}&*\\
&&&&*&*\\
\otimes&&&&&1
\end{pmatrix}\nonumber
\end{eqnarray}
for all $1\leq i<j<k\leq d$.
\end{remark}

\subsection{Edge-partition of graphs}
In this section we discuss a few results about edge partitions of the complete graph $K_m$. 

We start with a few conventions  and general results. In this paper $\Gamma$ is  an undirected, simple graph (i.e. no loops and no edges with multiplicity), $V(\Gamma)$ is the set of vertices, and $E(\Gamma)$ is the set of edges of $\Gamma$. Since our graphs are undirected we will make no distinction between the edge $(u,v)$ and the edge $(v,u)$. We say that $\Gamma$ has an
$s$-cycle if we can find a collection of distinct vertices  $v_1$, $v_2$, ..., $v_s\in V(\Gamma)$ such that $(v_1,v_2)$, $(v_2,v_3)$, ..., $(v_{s-1},v_s)$, $(v_s,v_1)\in E(\Gamma)$.

\begin{definition} Let $\Gamma$ be a graph and $k\geq 2$ be a natural number. A $k$-partition of $\Gamma$ is an ordered collection $(\Gamma_1,\Gamma_2,...,\Gamma_k)$ of sub-graphs $\Gamma_i$ of  $\Gamma$ such that:\\
1) $V(\Gamma_i)=V(\Gamma)$ for all $1\leq i \leq k$,  \\
2) $E(\Gamma_i)\cap E(\Gamma_j)=\emptyset$ for all $i\neq j$, \\
3) $\cup_{i=1}^nE(\Gamma_i)=E(\Gamma)$. \\
We say that the partition $(\Gamma_1,\Gamma_2,...,\Gamma_k)$ in homogeneous if $|vert(\Gamma_i)\vert=|vert(\Gamma_j)\vert$ for all $1\leq i<j\leq k$. We say that the partition $(\Gamma_1,\Gamma_2,...,\Gamma_k)$  is cycle-free if each $\Gamma_i$ is cycle-free.
\end{definition}

\begin{example} 1) Consider the complete graph $K_4$. Take $V(\Gamma_1)=V(\Gamma_2)=\{1,2,3,4\}$, $E(\Gamma_1)=\{(1,2),(1,4),(2,3)\}$, and $E(\Gamma_2)=\{(1,3),(2,4),(3,4)\}$, then $(\Gamma_1, \Gamma_2)$ is a homogeneous,
cycle-free $2$-partition for $K_4$ (see  Figure \ref{fig1}). \\
2) Take $V(\Gamma_1')=V(\Gamma_2')=\{1,2,3,4\}$, $E(\Gamma_1')=\{(1,2),(1,3),(1,4)\}$, and $E(\Gamma_2')=\{(2,3),(2,4),(3,4)\}$, then $(\Gamma_1', \Gamma_2')$ is a homogeneous $2$-partition for $K_4$. But it is not cycle-free because $(2,3,4)$ is a cycle in $\Gamma_2'$ (see Figure \ref{fig2}). \\
3) Take $V(\Gamma_1'')=V(\Gamma_2'')=\{1,2,3,4\}$, $E(\Gamma_1'')=\{(1,2),(1,3)\}$, and $E(\Gamma_2')=\{(2,3),(1,4),(2,4),(3,4)\}$, then $(\Gamma_1'', \Gamma_2'')$ is a  $2$-partition for $K_4$ but is not homogeneous, nor  cycle-free (see Figure \ref{fig3}).  \label{expar}
\end{example}

\begin{figure}[h]
\centering
\begin{tikzpicture}
  [scale=1.5,auto=left,every node/.style={shape = circle, draw, fill = white,minimum size = 1pt, inner sep=0.3pt}]
  \node (n1) at (0,0) {1};
  \node (n2) at (0,1)  {2};
  \node (n3) at (1,1)  {3};
  \node (n4) at (1,0)  {4};
  \foreach \from/\to in {n1/n2,n2/n3,n1/n4}
    \draw[line width=0.5mm,red]  (\from) -- (\to);	
\end{tikzpicture} \hspace{20mm}
\begin{tikzpicture}
  [scale=1.5,auto=left,every node/.style={shape = circle, draw, fill = white,minimum size = 1pt, inner sep=0.3pt}]
  \node (n1) at (0,0) {1};
  \node (n2) at (0,1)  {2};
  \node (n3) at (1,1)  {3};
  \node (n4) at (1,0)  {4};
  \foreach \from/\to in {n1/n3,n2/n4,n3/n4}
    \draw[line width=0.5mm,blue]  (\from) -- (\to);	
\end{tikzpicture}
\caption{$\Gamma(f)=(\Gamma_1,\Gamma_2)$ a cycle-free, homogeneous $2$-partition for $K_4$} \label{fig1}
\end{figure}
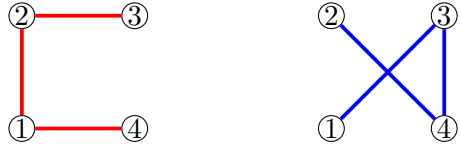

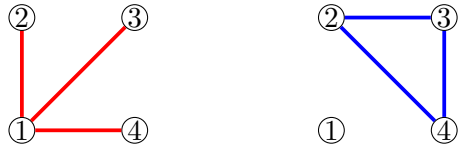
\begin{figure}[h]
\centering
\begin{tikzpicture}
  [scale=1.5,auto=left,every node/.style={shape = circle, draw, fill = white,minimum size = 1pt, inner sep=0.3pt}]
  \node (n1) at (0,0) {1};
  \node (n2) at (0,1)  {2};
  \node (n3) at (1,1)  {3};
  \node (n4) at (1,0)  {4};
  \foreach \from/\to in {n1/n2,n1/n3,n1/n4}
    \draw[line width=0.5mm,red]  (\from) -- (\to);	
\end{tikzpicture} \hspace{20mm}
\begin{tikzpicture}
  [scale=1.5,auto=left,every node/.style={shape = circle, draw, fill = white,minimum size = 1pt, inner sep=0.3pt}]
  \node (n1) at (0,0) {1};
  \node (n2) at (0,1)  {2};
  \node (n3) at (1,1)  {3};
  \node (n4) at (1,0)  {4};
  \foreach \from/\to in {n2/n3,n2/n4,n3/n4}
    \draw[line width=0.5mm,blue]  (\from) -- (\to);	
\end{tikzpicture}
\caption{$\Gamma(g)=(\Gamma_1',\Gamma_2')$ a homogeneous $2$-partition for $K_4$ that is not cycle free} \label{fig2}
\end{figure}

\begin{figure}[h]
\centering
\begin{tikzpicture}
  [scale=1.5,auto=left,every node/.style={shape = circle, draw, fill = white,minimum size = 1pt, inner sep=0.3pt}]
  \node (n1) at (0,0) {1};
  \node (n2) at (0,1)  {2};
  \node (n3) at (1,1)  {3};
  \node (n4) at (1,0)  {4};
  \foreach \from/\to in {n1/n2,n1/n3}
    \draw[line width=0.5mm,red]  (\from) -- (\to);	
\end{tikzpicture} \hspace{20mm}
\begin{tikzpicture}
  [scale=1.5,auto=left,every node/.style={shape = circle, draw, fill = white,minimum size = 1pt, inner sep=0.3pt}]
  \node (n1) at (0,0) {1};
  \node (n2) at (0,1)  {2};
  \node (n3) at (1,1)  {3};
  \node (n4) at (1,0)  {4};
  \foreach \from/\to in {n1/n4, n2/n3,n2/n4,n3/n4}
    \draw[line width=0.5mm,blue]  (\from) -- (\to);	
\end{tikzpicture}
\caption{$\Gamma(h)=(\Gamma_1'',\Gamma_2'')$ a $2$-partition for $K_4$ that is not homogeneous, nor cycle free} \label{fig3}
\end{figure}
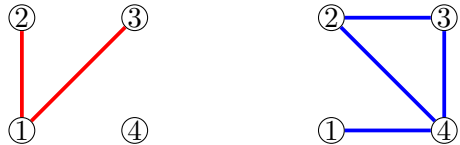

The following simple results will be used in this paper. They can be found in standard books of graph theory like \cite{h}, or be easily proved.  
\begin{lemma} 1) If $\Gamma$ is a simple graph that has no cycle, $\vert E(\Gamma)\vert =n-1$, and $\vert V(\Gamma)\vert =n$, then $\Gamma$ is connected. \\
2) If $\Gamma$ is a simple graph such that $\vert E(\Gamma)\vert=n$ and $\vert V(\Gamma)\vert =n$, then $\Gamma$ has a cycle.
\label{lemmagraph}
\end{lemma}

\begin{lemma} The only cycle free graphs with $6$ vertices and $5$ edges are the graphs $I_6$, $Y_6$, $E_6$, $H_6$, $C_6$ and $S_6$ from  Figure \ref{fig0}.
\label{lemma5E}
\end{lemma}

\begin{figure}[h]
\centering

\begin{tikzpicture}
  [scale=0.8,auto=left]
  \node[shape=circle,draw=black,minimum size = 5pt,inner sep=0.3pt] (n1) at (0,0) {};
  \node[shape=circle,draw=black,minimum size = 5pt,inner sep=0.3pt] (n2) at (1,0)  {};
  \node[shape=circle,draw=black,minimum size = 5pt,inner sep=0.3pt] (n3) at (2,0)  {};
  \node[shape=circle,draw=black,minimum size = 5pt,inner sep=0.3pt] (n4) at (3,0)  {};
	\node[shape=circle,draw=black,minimum size = 5pt,inner sep=0.3pt] (n5) at (4,0)   {};
	\node[shape=circle,draw=black,minimum size = 5pt,inner sep=0.3pt] (n6) at (5,0)   {};
	\node[shape=circle,minimum size = 14pt,inner sep=0.3pt] (n24) at (0,1) {$I_6$};
  \foreach \from/\to in {n1/n2,n2/n3,n3/n4,n4/n5,n5/n6}
    \draw[line width=0.5mm]  (\from) -- (\to);	
		
  \node[shape=circle,draw=black,minimum size = 5pt,inner sep=0.3pt]  (n11) at (6,0) {};
  \node[shape=circle,draw=black,minimum size = 5pt,inner sep=0.3pt]  (n21) at (7,0)  {};
  \node[shape=circle,draw=black,minimum size = 5pt,inner sep=0.3pt]  (n31) at (8,0)  {};
  \node[shape=circle,draw=black,minimum size = 5pt,inner sep=0.3pt]  (n41) at (9,0)  {};
	\node[shape=circle,draw=black,minimum size = 5pt,inner sep=0.3pt]  (n51) at (9.72,-0.72)   {};
	\node[shape=circle,draw=black,minimum size = 5pt,inner sep=0.3pt]  (n61) at (9.72,0.72)   {};
	\node[shape=circle,minimum size = 14pt,inner sep=0.3pt] (n24) at (6,1) {$Y_6$};
  \foreach \from/\to in {n11/n21,n21/n31,n31/n41,n41/n51,n41/n61}
    \draw[line width=0.5mm]  (\from) -- (\to);	
  \node[shape=circle,draw=black,minimum size = 5pt,inner sep=0.3pt]  (n12) at (11,0) {};
  \node[shape=circle,draw=black,minimum size = 5pt,inner sep=0.3pt]  (n22) at (12,0)  {};
  \node[shape=circle,draw=black,minimum size = 5pt,inner sep=0.3pt]  (n32) at (13,0)  {};
  \node[shape=circle,draw=black,minimum size = 5pt,inner sep=0.3pt]  (n42) at (14,0)  {};
	\node[shape=circle,draw=black,minimum size = 5pt,inner sep=0.3pt]  (n52) at (15,0)   {};
	\node[shape=circle,draw=black,minimum size = 5pt,inner sep=0.3pt]  (n62) at (13,1)   {};
	\node[shape=circle,minimum size = 14pt,inner sep=0.3pt] (n24) at (11,1) {$E_6$};
  \foreach \from/\to in {n12/n22,n22/n32,n32/n42,n42/n52,n32/n62}
    \draw[line width=0.5mm]  (\from) -- (\to);	

  \node[shape=circle,draw=black,minimum size = 5pt,inner sep=0.3pt] (n13) at (1.7,-3) {};
  \node[shape=circle,draw=black,minimum size = 5pt,inner sep=0.3pt] (n23) at (2.7,-3)  {};
  \node[shape=circle,draw=black,minimum size = 5pt,inner sep=0.3pt] (n33) at (0.97,-2.27)  {};
  \node[shape=circle,draw=black,minimum size = 5pt,inner sep=0.3pt] (n43) at (0.97,-3.72)  {};
	\node[shape=circle,draw=black,minimum size = 5pt,inner sep=0.3pt] (n53) at (3.42,-2.27)   {};
	\node[shape=circle,draw=black,minimum size = 5pt,inner sep=0.3pt] (n63) at (3.42,-3.72)   {};
	\node[shape=circle,minimum size = 14pt,inner sep=0.3pt] (n243) at (0,-2) {$H_6$};
  \foreach \from/\to in {n13/n23,n13/n33,n13/n43,n23/n53,n23/n63}
    \draw[line width=0.5mm]  (\from) -- (\to);	
  \node[shape=circle,draw=black,minimum size = 5pt,inner sep=0.3pt] (n14) at (6.3,-3) {};
  \node[shape=circle,draw=black,minimum size = 5pt,inner sep=0.3pt] (n24) at (7.3,-3)  {};
  \node[shape=circle,draw=black,minimum size = 5pt,inner sep=0.3pt] (n34) at (8.3,-3)  {};
  \node[shape=circle,draw=black,minimum size = 5pt,inner sep=0.3pt] (n44) at (9.3,-3)  {};
	\node[shape=circle,draw=black,minimum size = 5pt,inner sep=0.3pt] (n54) at (8.3,-4)   {};
	\node[shape=circle,draw=black,minimum size = 5pt,inner sep=0.3pt] (n64) at (8.3,-2)   {};
	\node[shape=circle,minimum size = 14pt,inner sep=0.3pt] (n244) at (6,-2) {$C_6$};
  \foreach \from/\to in {n14/n24,n24/n34,n34/n44,n34/n54,n34/n64}
    \draw[line width=0.5mm]  (\from) -- (\to);	
  \node[shape=circle,draw=black,minimum size = 5pt,inner sep=0.3pt]  (n16) at (13,-3) {};
  \node[shape=circle,draw=black,minimum size = 5pt,inner sep=0.3pt]  (n26) at (14,-3)  {};
  \node[shape=circle,draw=black,minimum size = 5pt,inner sep=0.3pt]  (n36) at (13.31,-2.05)  {};
  \node[shape=circle,draw=black,minimum size = 5pt,inner sep=0.3pt]  (n46) at (12.2,-2.41)  {};
	\node[shape=circle,draw=black,minimum size = 5pt,inner sep=0.3pt]  (n56) at (12.2,-3.59)   {};
	\node[shape=circle,draw=black,minimum size = 5pt,inner sep=0.3pt]  (n66) at (13.31,-3.95)   {};
	\node[shape=circle,minimum size = 14pt,inner sep=0.3pt] (n67) at (11,-2) {$S_6$};
  \foreach \from/\to in {n16/n26,n16/n36,n16/n46,n16/n56,n16/n66}
    \draw[line width=0.5mm]  (\from) -- (\to);	
\end{tikzpicture}

\caption{Cycle-free graphs with $6$ vertices $I_6$, $Y_6$, $E_6$, $H_6$, $C_6$ and $S_6$} \label{fig0}
\end{figure}
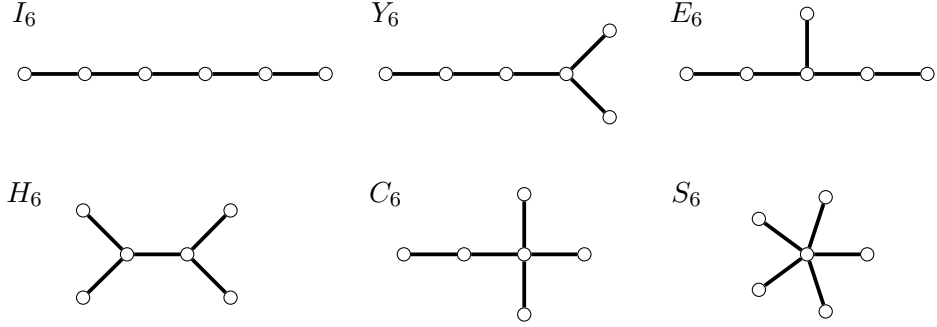


\section{Group Actions on $\Lambda^{S^2}_V(2d+1)$}

In this section, we  define two group actions of $S_d$ and $S_{2d}$ on $\Lambda^{S^2}_V(2d+1)$. We also introduce an element $E_d\in {\Lambda}^{S^2}_{V_d}(2d+1)$ that will play an important role in the rest of the paper.

The element $E_d\in {\Lambda}^{S^2}_{V_d}(2d-1)$ is defined inductively as follows.
$$E_1=\begin{pmatrix}
	1& e_1\\
	\wedge&1
\end{pmatrix}\in {\Lambda}^{S^2}_{V_1}(3),$$
$$E_2=\begin{pmatrix}
	1& e_1&e_2&e_1\\
	&1&e_1&e_2\\
	& &1&e_2\\
	\wedge&&&1
\end{pmatrix}=\begin{pmatrix}
	\begin{pmatrix}
		1& e_1\\
		\wedge&1
	\end{pmatrix}& \begin{pmatrix}
		e_2& e_1\\
		e_1&e_2
	\end{pmatrix}\\
	\wedge& \begin{pmatrix}
		1& e_2\\
		\wedge&1
	\end{pmatrix}
\end{pmatrix} \in {\Lambda}^{S^2}_{V_2}(5),$$
$$E_3=\begin{pmatrix}
	1& e_1&e_2&e_1&e_3&e_1\\
	&1&e_1&e_2&e_1&e_3\\
	& &1&e_2&e_3&e_2\\
	&&&1&e_2&e_3\\
	&&&&1&e_3\\
	\wedge&&&&&1
\end{pmatrix}
=\begin{pmatrix}
	\begin{pmatrix}
		1& e_1\\
		\wedge&1
	\end{pmatrix}& \begin{pmatrix}
		e_2& e_1\\
		e_1&e_2
	\end{pmatrix}&\begin{pmatrix}
		e_3& e_1\\
		e_1&e_3
	\end{pmatrix}\\
	& \begin{pmatrix}
		1& e_2\\
		\wedge&1
	\end{pmatrix}&\begin{pmatrix}
		e_3& e_2\\
		e_2&e_3
	\end{pmatrix}\\
\wedge	& &\begin{pmatrix}
		1& e_3\\
		 \wedge &1
	\end{pmatrix}
\end{pmatrix}
\in {\Lambda}^{S^2}_{V_3}(7),$$
and in general,
$$E_d=\begin{pmatrix}
	E_{d-1}& A_{d}\\
	\wedge& \begin{pmatrix}
		1& e_d\\
		\wedge&1
	\end{pmatrix}
\end{pmatrix}\in {\Lambda}^{S^2}_{V_d}(2d+1),$$
where
$$A_{d}=\otimes \begin{pmatrix}
	e_d&e_1\\
	e_1&e_d\\
	e_d&e_2\\
	e_2&e_d\\
	\vdots&\vdots\\
	e_d&e_{d-1}\\
	e_{d-1}&e_d
\end{pmatrix}\in {\mathcal B}_{V_d}(2d-1,3).$$

Recall that the construction of $\Lambda^{S^2}_V$ is functorial  so, in particular, we have an action $GL(V_d)\times\Lambda^{S^2}_{V_d}(2d+1)\to  \Lambda^{S^2}_{V_d}(2d+1)$ given by:
$$T*\begin{pmatrix}
	1& v_{1,2}&...&v_{1,2d-1}&v_{1,2d}\\
	& 1&...&v_{2,2d-1}&v_{2,2d}\\
	& &\ddots&.&.\\
	& & &1&v_{2d-1,2d}\\
	\wedge& & & &1
\end{pmatrix}=\begin{pmatrix}
	1& T(v_{1,2})&...&T(v_{1,2d-1})&T(v_{1,2d})\\
	& 1&...&T(v_{2,2d-1})&T(v_{2,2d})\\
	& &\ddots&.&.\\
	& & &1&T(v_{2d-1,2d})\\
	\wedge& & & &1
\end{pmatrix}.$$
Moreover, once we pick a basis in $V_d$, the symmetric group $S_d$ is a subgroup of $GL(V_d)$ with the action on $V_d$ determined by $\tau(e_i)=e_{\tau^{-1}(i)}$.
\begin{lemma}
	\label{GLLem}
Let $T\in End_k(V_d)$ and $dim(V_d)=d$, then $T*E_d=(det(T))^{2d-1}E_d$.
In particular, if $\tau\in S_d$, then $\tau*E_d=(sign(\tau))^{2d-1}E_d$.
\end{lemma}
\begin{proof}
Every element $T\in End_k(V_d)$ can be written as a product of elementary and diagonal matrices. So, it suffices to show that the statement holds for all elementary and diagonal matrices.

First, suppose $T$ is a diagonal matrix. So there exist an $1\leq i\leq d$ and $\lambda\in k$ such that
	\[T(e_s)=\begin{cases}
		\lambda e_i \ \ \ \ s=i\\
		e_s \ \ \ \ \ \ s\neq i.
	\end{cases}\]
	
Because of the linearity of $\exter{V_d}{d}$, we have
\begin{eqnarray*}
T*E_d&=&\begin{pmatrix}
	1& e_1&...&\lambda e_i&e_1&...&e_d&e_1\\
	&    1&...&e_1&\lambda e_i&...&e_1&e_d\\
	&      &...&.&.&...&.&.\\
	&      & &1&\lambda e_i&...&e_d&\lambda e_i\\
	&      & & &1&...&\lambda e_i&e_d\\
		&      & & & &...&.&.\\
				&      & & && &1&e_d\\
	\wedge &      & & && & &1
\end{pmatrix}\\
&=&\lambda^{2d-1}E_d \\
&=&det(T)^{2d-1}E_d.
\end{eqnarray*}

Next, assume that  $T$ is an elementary matrix, so there exist $i\neq j$ and $\lambda\in k$ such that
	\[T(e_s)=\begin{cases}
		e_s \ \ \ \ \ \ \ \ \ \ \ s\neq i \\
		e_i+\lambda e_j \ \ \ \ s=i.
	\end{cases}\]
We have \begin{eqnarray*}
T*E_d&=&\begin{pmatrix}
	1& e_1&...&e_i+\lambda e_j&e_1&...&e_d&e_1\\
	&    1&...&e_1&e_i+\lambda e_j&...&e_1&e_d\\
	&      &...&.&.&...&.&.\\
	&      & &1&e_i+\lambda e_j&...&e_d&e_i+\lambda e_j\\
	&      & & &1&...&e_i+\lambda e_j&e_d\\
		&      & & & &...&.&.\\
				&      & & && &1&e_d\\
	\wedge &      & & && & &1
\end{pmatrix}.
\end{eqnarray*}

By multi-linearity of the tensor product, it follows that $T*E_d$ is equal to a sum of $2^d$ elements of $\exter{V}{n}$ with entries in $\{e_1,e_2,...,e_d\}$. All these terms have at least $2d-1$ entries equal to $e_j$ (because the $e_j$ is not affected by $T$). Most of the terms will get another entry equal to  $\lambda e_j$ in some position $(k,l)$, hence by linearity and Lemma \ref{lemma2} they will be zero. The only term that will survive is the one that has the entry $e_i$ on positions $(2s-1,2i-1)$ and $(2s,2i)$ if $1\leq s\leq i-1$, $(2i-1,2i)$, and on the positions $(2i-1,2s)$ and $(2i,2s-1)$ if $i+1\leq s\leq d$. This element is equal to $E_d$, which means that
	\[T*E_d=E_d=(det(T))^{2d-1}E_d,\]  because $det(T)=1$.
\end{proof}

\begin{lemma}\label{actions2n} Let $V_d$ be a vector space of dimension $d$.  There is  group action $S_{2d}\times\exter{V}{d}\to\exter{V}{d}$ given by permuting the indices of columns. More precisely, if  $\sigma\in S_{2d}$  we define
\begin{eqnarray}
\sigma\rightharpoonup \begin{pmatrix}
	1& v_{1,2}&...&v_{1,2d-1}&v_{1,2d}\\
	& 1&...&v_{2,2d-1}&v_{2,2d}\\
	& &\ddots&.&.\\
	& & &1&v_{2d-1,2d}\\
	\wedge& & & &1
\end{pmatrix}=\begin{pmatrix}
	1& w_{1,2}&...&w_{1,2d-1}&w_{1,2d}\\
	& 1&...&w_{2,2d-1}&w_{2,2d}\\
	& &\ddots&.&.\\
	& & &1&w_{2d-1,2d}\\
	\wedge& & & &1
\end{pmatrix}
\label{defa}
\end{eqnarray}

where $w_{i,j}=v_{\sigma^{-1}(i),\sigma^{-1}(j)}$, up to the identification that $w_{k,l}=w_{l,k}$ for $k>l$. Moreover, we have
	$$\sigma\rightharpoonup E_d=(sign(\sigma))^{d-1}E_d.$$
\end{lemma}
\begin{proof}
The fact that the above formula gives an action is a straightforward computation. We will check the second equality for $d=2$ and $d=3$. The general case is similar and an alternative proof will be given  later in the paper.

Before we start, let's make a few conventions that will help with the flow of the proof. If we use a certain equality then we will refer it above the $\overset{(Eq.)}{=}$ sign. In order to specify on what positions we use that equality, we will box the corresponding vectors $\boxed{v_{i,j}}$. If a certain tensor matrix $A\in \exter{V}{d}$ is zero, then we will box it $\boxed{A}$, and we will highlight the vectors $\colorbox{orange}{$v_{i,j}$}$, $\colorbox{orange}{$v_{i,k}$}$ and $\colorbox{orange}{$v_{j,k}$}$ that make $\boxed{A}=0\in \exter{V}{d}$.

\begin{eqnarray}
(1,2)\rightharpoonup \begin{pmatrix}
	1& e_1&e_2&e_1\\
	&1&e_1&e_2\\
	& &1&e_2\\
	\wedge&&&1
\end{pmatrix}&\overset{(\ref{defa})}{=}&\begin{pmatrix}
	1& e_1&\boxed{e_1}&\boxed{e_2}\\
	&1&e_2&e_1\\
	& &1&\boxed{e_2}\\
	\wedge&&&1
\end{pmatrix} \nonumber \\
&\overset{(\ref{equ1})}{=}&-\begin{pmatrix}
	1& \boxed{e_1}&e_2&\boxed{e_2}\\
	&1&e_2&\boxed{e_1}\\
	& &1&e_1\\
	\wedge&&&1
\end{pmatrix}-\boxed{\begin{pmatrix}
	1&\colorbox{orange}{$e_1$}&e_2&\colorbox{orange}{$e_1$}\\
	&1&e_2&\colorbox{orange}{$e_1$}\\
	& &1&e_2\\
	\wedge&&&1
\end{pmatrix}}\nonumber \\
&\overset{(\ref{equ1})}{=}&\begin{pmatrix}
	1& e_1&e_2&e_1\\
	&1&\boxed{e_2}&\boxed{e_2}\\
	& &1&\boxed{e_1}\\
	\wedge&&&1
\end{pmatrix}
+\boxed{\begin{pmatrix}
	1& \colorbox{orange}{$e_2$}&\colorbox{orange}{$e_2$}&e_1\\
	&1&\colorbox{orange}{$e_2$}&e_1\\
	& &1&e_1\\
	\wedge&&&1
\end{pmatrix}}\nonumber \\
&\overset{(\ref{equ1})}{=}&-\begin{pmatrix}
	1& e_1&e_2&e_1\\
	&1&e_1&e_2\\
	& &1&e_2\\
	\wedge&&&1
\end{pmatrix}-
\boxed{\begin{pmatrix}
	1& \colorbox{orange}{$e_1$}&e_2&\colorbox{orange}{$e_1$}\\
	&1&e_2&\colorbox{orange}{$e_1$}\\
	& &1&e_2\\
	\wedge&&&1
\end{pmatrix}}\nonumber \\
&=&-E_2. \label{act1}
\end{eqnarray}
A similar computation shows that $(2,3)\rightharpoonup E_2=-E_2=(3,4)\rightharpoonup E_2$. Next, consider the case $d=3$. Using Equation \ref{act1} twice, we have
\begin{eqnarray*}
(1,2)\rightharpoonup \begin{pmatrix}
	1& e_1&e_2&e_1&e_3&e_1\\
	&1&e_1&e_2&e_1&e_3\\
	& &1&e_2&e_3&e_2\\
	&&&1&e_2&e_3\\
	&&&&1&e_3\\
	\wedge&&&&&1
\end{pmatrix}&\overset{(\ref{defa})}{=}&
\begin{pmatrix}
	1& \boxed{e_1}&\boxed{e_1}&\boxed{e_2}&e_1&e_3\\
	&1&\boxed{e_2}&\boxed{e_1}&e_3&e_1\\
	& &1&\boxed{e_2}&e_3&e_2\\
	&&&1&e_2&e_3\\
	&&&&1&e_3\\
	\wedge&&&&&1
\end{pmatrix}\\
\end{eqnarray*}
\begin{eqnarray*}
\overset{(\ref{act1})}{=}
-\begin{pmatrix}
	1& \boxed{e_1}&e_2&e_1&\boxed{e_1}&\boxed{e_3}\\
	&1&e_1&e_2&\boxed{e_3}&\boxed{e_1}\\
	& &1&e_2&e_3&e_2\\
	&&&1&e_2&e_3\\
	&&&&1&\boxed{e_3}\\
	\wedge&&&&&1
\end{pmatrix}
\overset{(\ref{act1})}{=}
\begin{pmatrix}
	1&e_1&e_2&e_1&e_3&e_1\\
	&1&e_1&e_2&e_1&e_3\\
	& &1&e_2&e_3&e_2\\
	&&&1&e_2&e_3\\
	&&&&1&e_3\\
	\wedge&&&&&1
\end{pmatrix}=E_3.
\end{eqnarray*}

Similarly, using Equation \ref{act1} for $(d-1)$ times, one can show that for every $d\geq 3$ and any $1\leq i\leq d$ we have
$(2i-1,2i)\rightharpoonup E_d=(-1)^{d-1}E_d$.

Next we have
\begin{eqnarray*}
&&(2,3)\rightharpoonup \begin{pmatrix}
	1& e_1&e_2&e_1&e_3&e_1\\
	&1&e_1&e_2&e_1&e_3\\
	& &1&e_2&e_3&e_2\\
	&&&1&e_2&e_3\\
	&&&&1&e_3\\
	\wedge&&&&&1
\end{pmatrix}\overset{(\ref{defa})}{=}
\begin{pmatrix}
	1& \boxed{e_2}&\boxed{e_1}&e_1&e_3&e_1\\
	&1&\boxed{e_1}&e_2&e_3&e_2\\
	& &1&e_2&e_1&e_3\\
	&&&1&e_2&e_3\\
	&&&&1&e_3\\
	\wedge&&&&&1
\end{pmatrix}\\
\end{eqnarray*}
\begin{eqnarray*}
&&\overset{(\ref{equ1})}{=}
-\begin{pmatrix}
	1& e_1&e_2&e_1&e_3&e_1\\
	&1&\boxed{e_1}&e_2&\boxed{e_3}&e_2\\
	& &1&e_2&\boxed{e_1}&e_3\\
	&&&1&e_2&e_3\\
	&&&&1&e_3\\
	\wedge&&&&&1
\end{pmatrix}-
\boxed{\begin{pmatrix}
	1& e_1&e_1&e_1&e_3&e_1\\
	&1&\colorbox{orange}{$e_2$}&\colorbox{orange}{$e_2$}&e_3&e_2\\
	& &1&\colorbox{orange}{$e_2$}&e_1&e_3\\
	&&&1&e_2&e_3\\
	&&&&1&e_3\\
	\wedge&&&&&1
\end{pmatrix}}\\
\end{eqnarray*}
\begin{eqnarray*}
&&\overset{(\ref{equ1})}{=}
\begin{pmatrix}
	1& e_1&e_2&e_1&e_3&e_1\\
	&1&\boxed{e_3}&e_2&e_1&\boxed{e_2}\\
	& &1&e_2&e_1&\boxed{e_3}\\
	&&&1&e_2&e_3\\
	&&&&1&e_3\\
	\wedge&&&&&1
\end{pmatrix}+
\boxed{\begin{pmatrix}
	1& e_1&e_2&e_1&e_3&e_1\\
	&1&e_1&e_2&e_1&e_2\\
	& &1&e_2&\colorbox{orange}{$e_3$}&\colorbox{orange}{$e_3$}\\
	&&&1&e_2&e_3\\
	&&&&1&\colorbox{orange}{$e_3$}\\
	\wedge&&&&&1
\end{pmatrix}}\\
\end{eqnarray*}
\begin{eqnarray*}
&&\overset{(\ref{equ1})}{=}
-\begin{pmatrix}
	1& e_1&e_2&e_1&e_3&e_1\\
	&1&\boxed{e_3}&e_2&\boxed{e_1}&e_3\\
	& &1&e_2&\boxed{e_1}&e_2\\
	&&&1&e_2&e_3\\
	&&&&1&e_3\\
	\wedge&&&&&1
\end{pmatrix}-
\boxed{\begin{pmatrix}
	1& e_1&e_2&e_1&e_3&e_1\\
	&1&\colorbox{orange}{$e_2$}&\colorbox{orange}{$e_2$}&e_1&e_3\\
	& &1&\colorbox{orange}{$e_2$}&e_1&e_3\\
	&&&1&e_2&e_3\\
	&&&&1&e_3\\
	\wedge&&&&&1
\end{pmatrix}}\\
\end{eqnarray*}
\begin{eqnarray*}
&&\overset{(\ref{equ1})}{=}
\begin{pmatrix}
	1& e_1&e_2&e_1&e_3&e_1\\
	&1&e_1&e_2&e_1&e_3\\
	& &1&e_2&e_3&e_2\\
	&&&1&e_2&e_3\\
	&&&&1&e_3\\
	\wedge&&&&&1
\end{pmatrix}+
\boxed{\begin{pmatrix}
	1& e_1&e_2&e_1&e_3&e_1\\
	&1&e_1&e_2&\colorbox{orange}{$e_3$}&\colorbox{orange}{$e_3$}\\
	& &1&e_2&e_1&e_2\\
	&&&1&e_2&e_3\\
	&&&&1&\colorbox{orange}{$e_3$}\\
	\wedge&&&&&1
\end{pmatrix}}\\
&&=E_3.
\end{eqnarray*}

A similar computation shows that for every $d\geq 3$ and any $1\leq i\leq d-1$, we have
$(2i,2i+1)\rightharpoonup E_d=(-1)^{d-1}E_d$.

\end{proof}

\begin{remark} It is obvious that the two actions of $S_{d}$ and $S_{2d}$ commute when we act on $E_d$. Later in the paper we will show that the same thing is true on $\exter{V_d}{d}$.
\end{remark}

\section{A dictionary between partitions of $K_{n}$ and elements in ${\Lambda}^{S^2}_{V_d}(n+1)$}
In this section we develop a graphic calculus for the partitions  of $K_{n}$ in relation with ${\Lambda}^{S^2}_{V_d}(n+1)$.

\subsection{Set of generators for ${\Lambda}^{S^2}_{V_d}(n+1)$}

Consider $\mathcal{B}_d=\{e_1,e_2,\dots,e_d\}$ an ordered basis for the vector space $V_d$. We have that $$\mathcal{G}_{\mathcal{B}_d}(n+1)=\{\begin{pmatrix}
1& v_{1,2}&...&v_{1,n-1}&v_{1,n}\\
& 1&...&v_{2,n-1}&v_{2,n}\\
& &...&.&.\\
& & &1&v_{n-1,n}\\
\otimes& & & &1
\end{pmatrix}\in V_d^{\otimes \frac{n(n-1)}{2}}\; \vert \; v_{i,j}\in \mathcal{B}_d\},$$ is a system of generators for $V_d^{\otimes \frac{n(n-1)}{2}}$ and, in particular, their images will form a system of generators for ${\Lambda}^{S^2}_{V_d}(n+1)$. Notice that to every element in $\mathcal{G}_{\mathcal{B}_d}(n+1)$ we can associate a $d$-partition of the graph $K_{n}$.
Indeed, if
$$f=\begin{pmatrix}
1& v_{1,2}&...&v_{1,n-1}&v_{1,n}\\
& 1&...&v_{2,n-1}&v_{2,n}\\
& &...&.&.\\
& & &1&v_{n-1,n}\\
\otimes& & & &1
\end{pmatrix}\in\mathcal{G}_{\mathcal{B}_d}(n+1)$$ we consider the  sub-graphs $\Gamma_i(f)$ of $K_{n}$ constructed as follows, for every $1\leq i\leq d$ we take $V(\Gamma_i(f))=\{1,2,\dots,2d\}$ and $E(\Gamma_i(f))=\{(s,t)|v_{s,t}=e_i\}$. One can easily see that $\Gamma(f)=(\Gamma_1(f),\dots,\Gamma_d(f))$ is a $d$-partition of $K_{n}$. Moreover, $f\mapsto \Gamma(f)$  is a bijective map from the set $\mathcal{G}_{\mathcal{B}_d}(n+1)$ to the set of $d$-partitions of $K_{n}$. We will denote by $f_{(\Gamma_1,\dots,\Gamma_d)}$ the element in $\mathcal{G}_{\mathcal{B}_d}(n+1)$ corresponding to the partition $(\Gamma_1,\dots,\Gamma_d)$, and by $\hat{f}_{(\Gamma_1,\dots,\Gamma_d)}$ the image of $f_{(\Gamma_1,\dots,\Gamma_d)}$ in ${\Lambda}^{S^2}_{V_d}(n+1)$.

\begin{example} Let $f=\begin{pmatrix}
	1& e_1&e_2&e_1\\
	&1&e_1&e_2\\
	& &1&e_2\\
	\otimes&&&1
\end{pmatrix}$, $g=\begin{pmatrix}
	1& e_1&e_1&e_1\\
	&1&e_2&e_2\\
	& &1&e_2\\
	\otimes&&&1
\end{pmatrix}$, $h=\begin{pmatrix}
	1& e_1&e_1&e_2\\
	&1&e_2&e_2\\
	& &1&e_2\\
	\otimes&&&1
\end{pmatrix}\in \mathcal{G}_{\mathcal{B}_2}(5).$  Then we have $\Gamma(f)=(\Gamma_1, \Gamma_2)$, $\Gamma(g)=(\Gamma_1', \Gamma_2')$, and $\Gamma(h)=(\Gamma_1'', \Gamma_2'')$ where $(\Gamma_1, \Gamma_2)$,  $(\Gamma_1', \Gamma_2')$ and $(\Gamma_1'', \Gamma_2'')$ are the $2$-partitions from  Example \ref{expar} (see also Figure \ref{fig1}, Figure \ref{fig2} and  Figure \ref{fig3}).
\end{example}

\begin{remark} In what follows, we will be interested in $d$-partitions of the complete graph $K_{2d}$. These will corespond to  elements in $\mathcal{G}_{\mathcal{B}_d}(2d+1)$ and their image in ${\Lambda}^{S^2}_{V_d}(n+1)$ will be a set of generators.
\end{remark}

\begin{remark}
A consequence of Lemma \ref{relemma1} is that  the image in ${\Lambda}^{S^2}_V(n+1)$ of any generator $f\in\mathcal{G}_{\mathcal{B}_d}(n+1)$, with a sub-matrix of the form $\begin{pmatrix}
 1&e_i&e_j\\
 &1&e_k\\
 &&1\\
\end{pmatrix}$ on the positions $(x,y)$, $(x,z)$ and $(y,z)$,  can be expressed in terms of five other generators that have the same entries as $f$, except on the positions $(x,y)$, $(x,z)$ and $(y,z)$ of the sub-matrix, where these are permuted in the five remaining possible ways.

If we denote by $f$, any of the six matrices in Lemma \ref{relemma1}, we have that $\Gamma(f)=(\Gamma_r, \Gamma_o, \Gamma_b)$ is a 3-partition of $K_3$ which can be identified with a triangle as in Figure \ref{triangle}. In Figure \ref{fig4} we have the second condition of the Lemma \ref{relemma1} written in terms of graphs. When we have only two distinct vectors $e_i$ and $e_j$ we get the identity in Figure \ref{3triangles}.

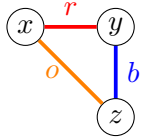
\begin{figure}[h]
\centering
\begin{tikzpicture}
  [scale=1.2,auto=left]
	\node[shape=circle,draw=black,minimum size = 14pt,inner sep=0.3pt] (n1) at (0,1) {$x$};
	\node[shape=circle,draw=black,minimum size = 14pt,inner sep=0.3pt] (n2) at (1,1) {$y$};
	\node[shape=circle,draw=black,minimum size = 14pt,inner sep=0.3pt] (n3) at (1,0) {$z$};
	
			 \draw[line width=0.5mm,red]  (n1) edge[] node [above] {$r$} (n2)  ;
			 \draw[line width=0.5mm,orange]  (n1) edge[] node [left] {$o$} (n3)  ;
			 \draw[line width=0.5mm,blue]  (n2) edge[] node [right] {$b$} (n3)  ;
\end{tikzpicture}
\caption{The face $(x,y,z)$}\label{triangle}
\end{figure}

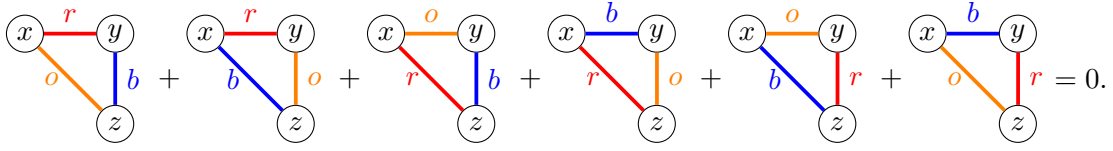
\begin{figure}[ht]
\centering
\begin{tikzpicture}
  [scale=1.2,auto=left]
	\node[shape=circle,draw=black,minimum size = 14pt,inner sep=0.3pt] (n1) at (0,1) {$x$};
	\node[shape=circle,draw=black,minimum size = 14pt,inner sep=0.3pt] (n2) at (1,1) {$y$};
	\node[shape=circle,draw=black,minimum size = 14pt,inner sep=0.3pt] (n3) at (1,0) {$z$};
	\node[shape=circle,minimum size = 14pt,inner sep=0.3pt] (n4) at (1.6,0.5) {+};

			 \draw[line width=0.5mm,red]  (n1) edge[] node [above] {$r$} (n2)  ;
			 \draw[line width=0.5mm,orange]  (n1) edge[] node [left] {$o$} (n3)  ;
			 \draw[line width=0.5mm,blue]  (n2) edge[] node [right] {$b$} (n3)  ;
		
	\node[shape=circle,draw=black,minimum size = 14pt,inner sep=0.3pt] (n5) at (2,1) {$x$};
	\node[shape=circle,draw=black,minimum size = 14pt,inner sep=0.3pt] (n6) at (3,1) {$y$};
	\node[shape=circle,draw=black,minimum size = 14pt,inner sep=0.3pt] (n7) at (3,0) {$z$};
	\node[shape=circle,minimum size = 14pt,inner sep=0.3pt] (n8) at (3.6,0.5) {+};
	
		\draw[line width=0.5mm,red]  (n5) edge[] node [above] {$r$} (n6)  ;
		\draw[line width=0.5mm,blue]  (n5) edge[] node [left] {$b$} (n7);
		\draw[line width=0.5mm,orange]  (n6) edge[] node [right] {$o$} (n7);	
			
	\node[shape=circle,draw=black,minimum size = 14pt,inner sep=0.3pt] (n9) at (4,1) {$x$};
	\node[shape=circle,draw=black,minimum size = 14pt,inner sep=0.3pt] (n10) at (5,1) {$y$};
	\node[shape=circle,draw=black,minimum size = 14pt,inner sep=0.3pt] (n11) at (5,0) {$z$};
	\node[shape=circle,minimum size = 14pt,inner sep=0.3pt] (n12) at (5.6,0.5) {+};
	
		\draw[line width=0.5mm,orange]  (n9) edge[] node [above] {$o$} (n10)  ;
		\draw[line width=0.5mm,red]  (n9) edge[] node [left] {$r$} (n11);
		\draw[line width=0.5mm,blue]  (n10) edge[] node [right] {$b$} (n11);	
			
	\node[shape=circle,draw=black,minimum size = 14pt,inner sep=0.3pt] (n13) at (6,1) {$x$};
	\node[shape=circle,draw=black,minimum size = 14pt,inner sep=0.3pt] (n14) at (7,1) {$y$};
	\node[shape=circle,draw=black,minimum size = 14pt,inner sep=0.3pt] (n15) at (7,0) {$z$};
	\node[shape=circle,minimum size = 14pt,inner sep=0.3pt] (n16) at (7.6,0.5) {+};
	
		\draw[line width=0.5mm,blue]  (n13) edge[] node [above] {$b$} (n14)  ;
		\draw[line width=0.5mm,red]  (n13) edge[] node [left] {$r$} (n15);
		\draw[line width=0.5mm,orange]  (n14) edge[] node [right] {$o$} (n15);			
					
	\node[shape=circle,draw=black,minimum size = 14pt,inner sep=0.3pt] (n17) at (8,1) {$x$};
	\node[shape=circle,draw=black,minimum size = 14pt,inner sep=0.3pt] (n18) at (9,1) {$y$};
	\node[shape=circle,draw=black,minimum size = 14pt,inner sep=0.3pt] (n19) at (9,0) {$z$};
	\node[shape=circle,minimum size = 14pt,inner sep=0.3pt] (n20) at (9.6,0.5) {+};
	
		\draw[line width=0.5mm,orange]  (n17) edge[] node [above] {$o$} (n18)  ;
		\draw[line width=0.5mm,blue]  (n17) edge[] node [left] {$b$} (n19);
		\draw[line width=0.5mm,red]  (n18) edge[] node [right] {$r$} (n19);

	\node[shape=circle,draw=black,minimum size = 14pt,inner sep=0.3pt] (n21) at (10,1) {$x$};
	\node[shape=circle,draw=black,minimum size = 14pt,inner sep=0.3pt] (n22) at (11,1) {$y$};
	\node[shape=circle,draw=black,minimum size = 14pt,inner sep=0.3pt] (n23) at (11,0) {$z$};
	\node[shape=circle,minimum size = 14pt,inner sep=0.3pt] (n24) at (11.7,0.5) {=\;0.};
	
		\draw[line width=0.5mm,blue]  (n21) edge[] node [above] {$b$} (n22)  ;
		\draw[line width=0.5mm,orange]  (n21) edge[] node [left] {$o$} (n23);
		\draw[line width=0.5mm,red]  (n22) edge[] node [right] {$r$} (n23);									
				
\end{tikzpicture}
\caption{Generating relations for the ideal ${\mathcal E}^{S^2}_{V_d}$ } \label{fig4}
\end{figure}

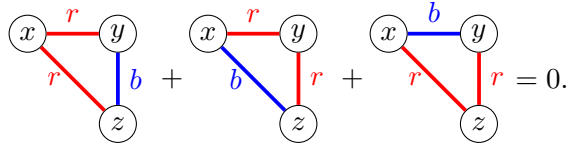
\begin{figure}[ht]
\centering\begin{tikzpicture}
  [scale=1.2,auto=left]
	\node[shape=circle,draw=black,minimum size = 14pt,inner sep=0.3pt] (n1) at (0,1) {$x$};
	\node[shape=circle,draw=black,minimum size = 14pt,inner sep=0.3pt] (n2) at (1,1) {$y$};
	\node[shape=circle,draw=black,minimum size = 14pt,inner sep=0.3pt] (n3) at (1,0) {$z$};
	\node[shape=circle,minimum size = 14pt,inner sep=0.3pt] (n4) at (1.6,0.5) {+};
	
			\draw[line width=0.5mm,red]  (n1) edge[] node [above] {$r$} (n2)  ;
			\draw[line width=0.5mm,red]  (n1) edge[] node [left] {$r$} (n3)  ;
			\draw[line width=0.5mm,blue]  (n2) edge[] node [right] {$b$} (n3)  ;
		
	\node[shape=circle,draw=black,minimum size = 14pt,inner sep=0.3pt] (n5) at (2,1) {$x$};
	\node[shape=circle,draw=black,minimum size = 14pt,inner sep=0.3pt] (n6) at (3,1) {$y$};
	\node[shape=circle,draw=black,minimum size = 14pt,inner sep=0.3pt] (n7) at (3,0) {$z$};
	\node[shape=circle,minimum size = 14pt,inner sep=0.3pt] (n8) at (3.6,0.5) {+};

			\draw[line width=0.5mm,red]  (n5) edge[] node [above] {$r$} (n6)  ;
			\draw[line width=0.5mm,blue]  (n5) edge[] node [left] {$b$} (n7)  ;
			\draw[line width=0.5mm,red]  (n6) edge[] node [right] {$r$} (n7)  ;		
			
	\node[shape=circle,draw=black,minimum size = 14pt,inner sep=0.3pt] (n9) at (4,1) {$x$};
	\node[shape=circle,draw=black,minimum size = 14pt,inner sep=0.3pt] (n10) at (5,1) {$y$};
	\node[shape=circle,draw=black,minimum size = 14pt,inner sep=0.3pt] (n11) at (5,0) {$z$};
	\node[shape=circle,minimum size = 14pt,inner sep=0.3pt] (n24) at (5.7,0.5) {=\;0.};
	
			\draw[line width=0.5mm,blue]  (n9) edge[] node [above] {$b$} (n10)  ;
			\draw[line width=0.5mm,red]  (n9) edge[] node [left] {$r$} (n11)  ;
			\draw[line width=0.5mm,red]  (n10) edge[] node [right] {$r$} (n11)  ;

\end{tikzpicture}
\caption{Case of two components} \label{3triangles}
\end{figure}

\end{remark}

\begin{remark}
One should notice that these equalities extend beyond ${\Lambda}^{S^2}_V(3)$. As Remark \ref{remba} points out, we have similar relations in ${\Lambda}^{S^2}_V(n+1)$ among the elements that are different only on the face $(x,y,z)$.  More precisely, suppose that   $f_{(\Gamma_1^{(k)},\dots,\Gamma_d^{(k)})}\in \mathcal{G}_{\mathcal{B}_d}(n+1)$ for $1\leq k\leq 6$, where $(\Gamma_1^{(k)},\dots,\Gamma_d^{(k)})$ are $d$-partitions of $K_n$ that are equal on all edges except for the face $(x,y,z)$ where they coincide with one of the six possible partitions of the face $K_3$ from Figure \ref{fig4}. Then,
$$\sum_{k=1}^6\hat{f}_{(\Gamma_1^{(k)},\dots,\Gamma_d^{(k)})}=0\in {\Lambda}^{S^2}_V(n+1).$$
Notice that we need to assume $char(k)\neq 2$ and $char(k)\neq 3$  in order to recover relations (\ref{equ0}) and (\ref{equ1}).
\end{remark}

\subsection{Homogeneous cycle-free $d$-partitions of $K_{2d}$}

We will now concentrate on the component of degree $2d+1$. We already know that the set of $d$-partitions of the graph $K_{2d}$ provide a system of generators for ${\Lambda}^{S^2}_V(2d+1)$. In this section we show that we can restrict to $d$-partitions that are homogeneous and cycle-free.

\begin{lemma}\label{generators} The set $\hat{\mathcal{G}}_{\mathcal{B}_d}^{cf}(2d+1)$ is a system on generators for ${\Lambda}^{S^2}_V(2d+1)$, where
$$\hat{\mathcal{G}}_{\mathcal{B}_d}^{cf}(2d+1)=\{\hat{f}_{(\Gamma_1,...,\Gamma_d)}\vert (\Gamma_1,...,\Gamma_d) {\rm ~is~ a ~homogeneous,~cycle}{\text -}{\rm free~}d{\text -}{\rm partition~of~}K_{2d}\}.$$
\end{lemma}
\begin{proof}
We need to show that the only elements from $\mathcal{G}_{\mathcal{B}_d}(2d+1)$ that have a chance of being nonzero in
${\Lambda}^{S^2}_V(2d+1)$ are those corresponding to homogeneous partitions.

Indeed, because the graph $K_{2d}$ has $d(2d-1)$ edges any $d$-partition $\Gamma=(\Gamma_1,  \dots, \Gamma_d)$, which is not homogeneous,  will have a sub-graph $\Gamma_i$ with $2d$ vertices and at least $2d$ edges. By Lemma \ref{lemmagraph} this implies that $\Gamma_i$ will have a cycle.

We can prove by induction that the existence of a cycle for a partition $\Gamma=(\Gamma_1,\dots, \Gamma_d)$ implies that the corresponding element $\hat{f}_{(\Gamma_1, \dots, \Gamma_d)}$ in ${\Lambda}^{S^2}_V(2d+1)$ is zero. If the cycle has length 3, then $f_{(\Gamma_1, \dots, \Gamma_d)}$ is in the ideal $\mathcal{E}_V^{S^2}$, so the assertion is clear.

Assume now that $\Gamma_i$ has a cycle of length $n>3$. We denote this cycle by $(a_1, a_2, \dots, a_n)$ and we represent it with red edges as in Figure \ref{figcycle}.
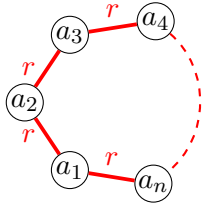
\begin{figure}[h]
\begin{tikzpicture}
  [scale=0.6,auto=left]
  \node[shape=circle,draw=black,minimum size = 14pt,inner sep=0.3pt] (n1) at (0,0)  {$a_1$};
  \node[shape=circle,draw=black,minimum size = 14pt,inner sep=0.3pt] (n2) at (-1,1.5)  {$a_2$};
  \node[shape=circle,draw=black,minimum size = 14pt,inner sep=0.3pt] (n3) at (0,3)  {$a_3$};
  \node[shape=circle,draw=black,minimum size = 14pt,inner sep=0.3pt] (n4) at (1.9,3.3)  {$a_4$};
  \node[shape=circle,draw=black,minimum size = 14pt,inner sep=0.3pt] (n5) at (1.85,-0.3) {$a_n$};

		\draw[line width=0.5mm,red]  (n1) edge[] node [left] {$r$} (n2)  ;	
		\draw[line width=0.5mm,red]  (n2) edge[] node [left] {$r$} (n3)  ;	
		\draw[line width=0.5mm,red]  (n1) edge[] node [above] {$r$} (n5)  ;	
		\draw[line width=0.5mm,red]  (n3) edge[] node [above] {$r$} (n4)  ;	
		
    \path[dashed, thick,red] (n4) edge[bend left=50] node [above] {} (n5);
\end{tikzpicture} \hspace{20mm}
\caption{$\Gamma_i$, a cycle of length $n$} \label{figcycle}
\end{figure}


Since $n>3$, the edge connecting $a_1$ and $a_3$ does not belong to $\Gamma_i$, and we will represent it with blue. Then, using the equation in Figure \ref{3triangles}, we have the identity in Figure \ref{Cycleinduction}.

  \begin{figure}[h]
  \centering
  \begin{tikzpicture}
  [scale=0.6,auto=left]
  \node[shape=circle,draw=black,minimum size = 14pt,inner sep=0.3pt] (n1) at (0,0)  {$a_1$};
  \node[shape=circle,draw=black,minimum size = 14pt,inner sep=0.3pt] (n2) at (-1,1.5)  {$a_2$};
  \node[shape=circle,draw=black,minimum size = 14pt,inner sep=0.3pt](n3) at (0,3)  {$a_3$};
  \node[shape=circle,draw=black,minimum size = 14pt,inner sep=0.3pt] (n4) at (1.9,3.3)  {$a_4$};
  \node[shape=circle,draw=black,minimum size = 14pt,inner sep=0.3pt] (n5) at (1.85,-0.3) {$a_n$};
  \node[minimum size = 14pt,inner sep=0.3pt] (n8) at (3.75,1.5) {=};
  \node[minimum size = 14pt,inner sep=0.3pt] (n21) at (4.75,1.5) {$-$};

		\draw[line width=0.5mm,red]  (n1) edge[] node [left] {$r$} (n2)  ;	
		\draw[line width=0.5mm,red]  (n2) edge[] node [left] {$r$} (n3)  ;	
		\draw[line width=0.5mm,red]  (n1) edge[] node [above] {$r$} (n5)  ;	
		\draw[line width=0.5mm,red]  (n3) edge[] node [above] {$r$} (n4)  ;	
    \draw[line width=0.5mm,blue]  (n1) edge[] node [right] {$b$} (n3)  ;	
   \path[dashed, thick,red] (n4) edge[bend left=50] node [above] {} (n5);

 \node[shape=circle,draw=black,minimum size = 14pt,inner sep=0.3pt] (n6) at (7,0)  {$a_1$};
 \node[shape=circle,draw=black,minimum size = 14pt,inner sep=0.3pt] (n7) at (6,1.5)  {$a_2$};
 \node[shape=circle,draw=black,minimum size = 14pt,inner sep=0.3pt] (n8) at (7,3)  {$a_3$};
 \node[shape=circle,draw=black,minimum size = 14pt,inner sep=0.3pt] (n9) at (8.9,3.3)  {$a_4$};
 \node[shape=circle,draw=black,minimum size = 14pt,inner sep=0.3pt] (n10) at (8.85,-0.3) {$a_n$};
 \node[shape=circle,minimum size = 14pt,inner sep=0.3pt] (n11) at (10.75,1.5) {$-$};

			\draw[line width=0.5mm,red]  (n7) edge[] node [left] {$r$} (n8)  ;	
		\draw[line width=0.5mm,red]  (n6) edge[] node [right] {$r$} (n8)  ;	
		\draw[line width=0.5mm,red]  (n6) edge[] node [above] {$r$} (n10)  ;	
		\draw[line width=0.5mm,red]  (n8) edge[] node [above] {$r$} (n9)  ;	
    \draw[line width=0.5mm,blue]  (n6) edge[] node [left] {$b$} (n7)  ;	
   \path[dashed, thick,red] (n9) edge[bend left=50] node [above] {} (n10);

  \node[shape=circle,draw=black,minimum size = 14pt,inner sep=0.3pt] (n12) at (13,0)  {$a_1$};
  \node[shape=circle,draw=black,minimum size = 14pt,inner sep=0.3pt] (n13) at (12,1.5) {$a_2$};
  \node[shape=circle,draw=black,minimum size = 14pt,inner sep=0.3pt] (n14) at (13,3)  {$a_3$};
  \node[shape=circle,draw=black,minimum size = 14pt,inner sep=0.3pt] (n15) at (14.9,3.3)  {$a_4$};
  \node[shape=circle,draw=black,minimum size = 14pt,inner sep=0.3pt] (n16) at (14.85,-0.3) {$a_n$};

	
			\draw[line width=0.5mm,red]  (n12) edge[] node [left] {$r$} (n13)  ;	
		\draw[line width=0.5mm,red]  (n12) edge[] node [right] {$r$} (n14)  ;	
		\draw[line width=0.5mm,red]  (n14) edge[] node [above] {$r$} (n15)  ;	
		\draw[line width=0.5mm,red]  (n12) edge[] node [above] {$r$} (n16)  ;	
    \draw[line width=0.5mm,blue]  (n13) edge[] node [left] {$b$} (n14)  ;	
   \path[dashed, thick,red] (n15) edge[bend left=50] node [above] {} (n16);

\end{tikzpicture} \hspace{20mm}
\caption{Cycle reduction} \label{Cycleinduction}
\end{figure}
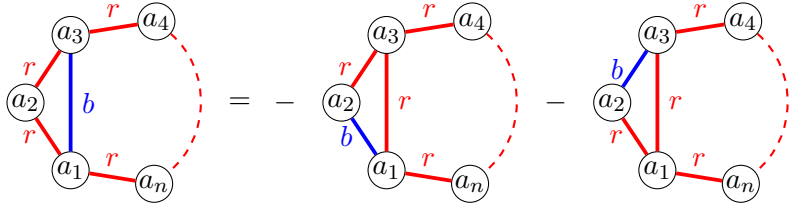

Note that on the right side of the equality in Figure \ref{Cycleinduction}  we have two cycles of length $n-1$ and, by the inductive hypothesis, the corresponding elements for these partitions are zero in  ${\Lambda}^{S^2}_V(2d+1)$. This implies that so is the image of $f_{(\Gamma_1,\dots,\Gamma_d)}$.

\end{proof}

\begin{remark}  Let $(\Gamma_1,...,\Gamma_d)$ be a cycle-free, homogeneous $d$-partition of the graph $K_{2d}$.
Since the partition is homogeneous then each $\Gamma_i$ has exactly $2d-1$ edges. Since the partition is cycle-free from  Lemma \ref{lemmagraph} we get that each $\Gamma_i$ is connected. In particular, if we take three distinct vertices $1\leq x<y<z\leq 2d$, then for each $1\leq i\leq d$ we can find a unique  path in $\Gamma_i$ that contains all three vertices $x$, $y$ and $z$ in some order. So, we can say that in the graph $\Gamma_i$ one of the three vertices is between the other two. \\
\end{remark}

We are ready now to state and prove the following key lemma which will play an essential role in the rest of the paper.
\begin{lemma} Take $(\Gamma_1,\dots, \Gamma_d)$ a cycle-free homogenous $d$-partition of $K_{2d}$, and pick three distinct integers $1\leq x,y,z\leq 2d$. Then, there exist $(\Lambda_1,\dots, \Lambda_d)$, a unique cycle-free homogenous $d$-partition of $K_{2d}$  such that the two partitions $(\Gamma_1,\dots,\Gamma_d)$  and $(\Lambda_1,\dots,\Lambda_d)$ coincide on every edge of $K_{2d}$ except on the face $(x,y,z)$ where they are different on at least two edges. Moreover, with the above notations we have $$\hat{f}_{(\Gamma_1,\dots,\Gamma_d)}=-\hat{f}_{(\Lambda_1,\dots,\Lambda_d)}\in {\Lambda}^{S^2}_{V_d}(2d+1).$$
We will denote $(\Lambda_1,\dots,\Lambda_d)$ by $(\Gamma_1,\dots,\Gamma_d)^{(x,y,z)}$.
\label{keylemma}
\end{lemma}
\begin{proof}  First notice that the three edges of the face $(x,y,z)$ cannot belong to the same graph
$\Gamma_i$, otherwise $\Gamma_i$ would have cycle.

Next, suppose that the edges of the face $(x,y,z)$  belong to the three distinct graphs $\Gamma_r$, $\Gamma_o$ and $\Gamma_b$ (i.e. have colors  red, orange and blue).  We remove the three edges $(x,y)$, $(x,z)$ and $(y,z)$.   On the face $(x,y,z)$ we must have of one of the three situations described in  Figure \ref{fig5}. The edges that we removed are now dashed, and the wiggling arcs (red, orange, and blue), represent paths that connect  vertices $x$, $y$, and $z$ in $\Gamma_r$, $\Gamma_o$ and $\Gamma_b$ respectively. We can obviously ignore case (III) since in that situation no mater what color we assign to the  edge $(x,y)$ the partition $(\Gamma_1,\dots,\Gamma_d)$ will not be cycle free.
\begin{figure}[h]
\centering
\begin{tikzpicture}
  [scale=2,auto=left]
	\node[shape=circle,draw=black,minimum size = 14pt,inner sep=0.3pt] (n1) at (0,1) {$x$};
	\node[shape=circle,draw=black,minimum size = 14pt,inner sep=0.3pt] (n2) at (1,1) {$y$};
	\node[shape=circle,draw=black,minimum size = 14pt,inner sep=0.3pt] (n3) at (1,0) {$z$};
	\node[shape=circle,minimum size = 14pt,inner sep=0.3pt] (m4) at (-0.5,0) {(I)};

		\path[line width=0.5mm,red] (n1) edge[bend left=120,snake it] node [above] {$\;\;r$} (n2);
		\path[line width=0.5mm,blue] (n1) edge[bend left=45,snake it] node [below] {$b$} (n2);
		\path[line width=0.5mm,orange] (n1) edge[bend right=60,snake it] node [above] {$o$} (n3);
	  \draw[line width=0.5mm,dashed]  (n1) -- (n2)  ;
		\draw[line width=0.5mm,dashed]  (n1) -- (n3);
		\draw[line width=0.5mm,dashed]  (n2) -- (n3);	
		
	\node[shape=circle,draw=black,minimum size = 14pt,inner sep=0.3pt] (n4) at (3,1) {$x$};
	\node[shape=circle,draw=black,minimum size = 14pt,inner sep=0.3pt] (n5) at (4,1) {$y$};
	\node[shape=circle,draw=black,minimum size = 14pt,inner sep=0.3pt] (n6) at (4,0) {$z$};
	\node[shape=circle,minimum size = 14pt,inner sep=0.3pt] (m4) at (2.5,0) {(II)};

		\path[line width=0.5mm,red] (n4) edge[bend left=120,snake it] node [above] {$\;\;r$} (n5);
		\path[line width=0.5mm,orange] (n4) edge[bend right=60,snake it] node [above] {$o$} (n6);
		\path[line width=0.5mm,blue] (n5) edge[bend left=60,snake it] node [left] {$b$} (n6);
	  \draw[line width=0.5mm,dashed]  (n4) -- (n5)  ;
		\draw[line width=0.5mm,dashed]  (n4) -- (n6);
		\draw[line width=0.5mm,dashed]  (n5) -- (n6);

  \node[shape=circle,draw=black,minimum size = 14pt,inner sep=0.3pt] (n7) at (6,1) {$x$};
	\node[shape=circle,draw=black,minimum size = 14pt,inner sep=0.3pt] (n8) at (7,1) {$y$};
	\node[shape=circle,draw=black,minimum size = 14pt,inner sep=0.3pt] (n9) at (7,0) {$z$};
		\node[shape=circle,minimum size = 14pt,inner sep=0.3pt] (m4) at (5.5,0) {(III)};

		\path[line width=0.5mm,red] (n7) edge[bend left=120,snake it] node [above] {$\;\;r$} (n8);
		\path[line width=0.5mm,orange] (n7) edge[bend right=30,snake it] node [below] {$o$} (n8);
		\path[line width=0.5mm,blue] (n7) edge[bend left=45,snake it] node [below] {$b$} (n8);
	  \draw[line width=0.5mm,dashed]  (n7) -- (n8)  ;
		\draw[line width=0.5mm,dashed]  (n7) -- (n9);
		\draw[line width=0.5mm,dashed]  (n8) -- (n9);	

\end{tikzpicture}

\caption{Paths among vertices $x$, $y$, and $z$} \label{fig5}
\end{figure}
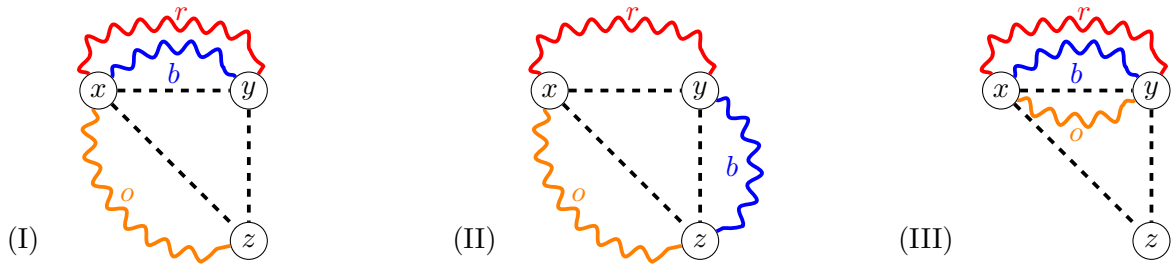

In the other two cases the only possible colorings that will give a cycle free homogeneous partition are the ones in Figure \ref{fig6} and Figure \ref{fig7}. Indeed, let's look to case (I), since there is are red and blue paths between $x$ and $y$, it means that the edge $(x,y)$ must be orange. The other two edges $(x,z)$ and $(y,z)$ must be red and blue, or blue and red. One of these two coloring corresponds to the initial partition $(\Gamma_1,\dots, \Gamma_d)$, the other one is $(\Lambda_1,\dots, \Lambda_d)$. Case (II) is similar.

\begin{figure}[h]
\centering
\begin{tikzpicture}
  [scale=2,auto=left]
	\node[shape=circle,draw=black,minimum size = 14pt,inner sep=0.3pt] (n1) at (0,1) {$x$};
	\node[shape=circle,draw=black,minimum size = 14pt,inner sep=0.3pt] (n2) at (1,1) {$y$};
	\node[shape=circle,draw=black,minimum size = 14pt,inner sep=0.3pt] (n3) at (1,0) {$z$};

		\path[line width=0.5mm,red] (n1) edge[bend left=120, snake it] node [above] {$\;\;r$} (n2);
		\path[line width=0.5mm,blue] (n1) edge[bend left=45, snake it] node [below] {$b$} (n2);
		\path[line width=0.5mm,orange] (n1) edge[bend right=60, snake it] node [above] {$o$} (n3);
	  \draw[line width=0.5mm,orange]  (n1) edge[] node [below] {$o$} (n2)  ;
		 \draw[line width=0.5mm,blue]  (n1) edge[] node [left] {$b$} (n3)  ;
		 \draw[line width=0.5mm,red]  (n2) edge[] node [right] {$r$} (n3)  ;

	\node[shape=circle,draw=black,minimum size = 14pt,inner sep=0.3pt] (n4) at (3,1) {$x$};
	\node[shape=circle,draw=black,minimum size = 14pt,inner sep=0.3pt] (n5) at (4,1) {$y$};
	\node[shape=circle,draw=black,minimum size = 14pt,inner sep=0.3pt] (n6) at (4,0) {$z$};

		\path[line width=0.5mm,red] (n4) edge[bend left=120, snake it] node [above] {$\;\;r$} (n5);
		\path[line width=0.5mm,blue] (n4) edge[bend left=45, snake it] node [below] {$b$} (n5);
		\path[line width=0.5mm,orange] (n4) edge[bend right=60, snake it] node [above] {$o$} (n6);
	  \draw[line width=0.5mm,orange]  (n4) edge[] node [below] {$o$} (n5)  ;
		 \draw[line width=0.5mm,red]  (n4) edge[] node [left] {$r$} (n6)  ;
		 \draw[line width=0.5mm,blue]  (n5) edge[] node [right] {$b$} (n6)  ;

\end{tikzpicture}
\caption{Case (I)} \label{fig6}
\end{figure}

\begin{figure}[h]
\centering
\begin{tikzpicture}
  [scale=2,auto=left]
	\node[shape=circle,draw=black,minimum size = 14pt,inner sep=0.3pt] (n1) at (0,1) {$x$};
	\node[shape=circle,draw=black,minimum size = 14pt,inner sep=0.3pt] (n2) at (1,1) {$y$};
	\node[shape=circle,draw=black,minimum size = 14pt,inner sep=0.3pt] (n3) at (1,0) {$z$};

		\path[line width=0.5mm,red] (n1) edge[bend left=120, snake it] node [below] {$r$} (n2);
		\path[line width=0.5mm,orange] (n1) edge[bend right=60, snake it] node [above] {$o$} (n3);
		\path[line width=0.5mm,blue] (n2) edge[bend left=60, snake it] node [left] {$b$} (n3);
	  \draw[line width=0.5mm,orange]  (n1) edge[] node [below] {$o$} (n2)  ;
		 \draw[line width=0.5mm,blue]  (n1) edge[] node [left] {$b$} (n3)  ;
		 \draw[line width=0.5mm,red]  (n2) edge[] node [left] {$r$} (n3)  ;

	\node[shape=circle,draw=black,minimum size = 14pt,inner sep=0.3pt] (n4) at (3,1) {$x$};
	\node[shape=circle,draw=black,minimum size = 14pt,inner sep=0.3pt] (n5) at (4,1) {$y$};
	\node[shape=circle,draw=black,minimum size = 14pt,inner sep=0.3pt] (n6) at (4,0) {$z$};

		\path[line width=0.5mm,red] (n4) edge[bend left=120, snake it] node [below] {$r$} (n5);
		\path[line width=0.5mm,orange] (n4) edge[bend right=60, snake it] node [above] {$o$} (n6);
		\path[line width=0.5mm,blue] (n5) edge[bend left=60, snake it] node [left] {$b$} (n6);
	  \draw[line width=0.5mm,blue]  (n4) edge[] node [below] {$b$} (n5)  ;
		\draw[line width=0.5mm,red]  (n4) edge[] node [left] {$r$} (n6)  ;
		\draw[line width=0.5mm,orange]  (n5) edge[] node [left] {$o$} (n6)  ;

\end{tikzpicture}
\caption{Case (II)} \label{fig7}
\end{figure}
Combining the above observation with the results from Lemma \ref{relemma1} and the equality in  Figure \ref{fig4} we get that $$\hat{f}_{(\Gamma_1,\dots,\Gamma_d)}=-\hat{f}_{(\Lambda_1,\dots,\Lambda_d)}\in {\Lambda}^{S^2}_V(2d+1).$$
A similar analysis can be done for the case when the edges of the face $(x,y,x)$  belong to only two distinct graphs $\Gamma_r$, and $\Gamma_b$  (see Figure \ref{fig8}).

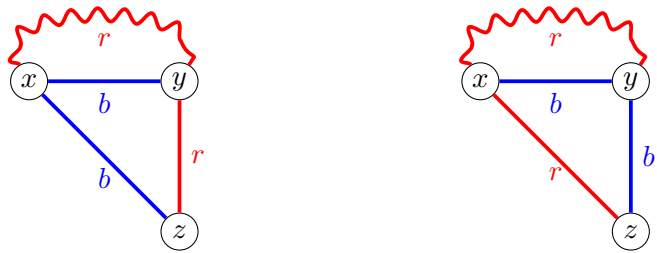
\begin{figure}[h]
\centering
\begin{tikzpicture}
  [scale=2,auto=left]
	\node[shape=circle,draw=black,minimum size = 14pt,inner sep=0.3pt] (n1) at (0,1) {$x$};
	\node[shape=circle,draw=black,minimum size = 14pt,inner sep=0.3pt] (n2) at (1,1) {$y$};
	\node[shape=circle,draw=black,minimum size = 14pt,inner sep=0.3pt] (n3) at (1,0) {$z$};

		\path[line width=0.5mm,red] (n1) edge[bend left=120, snake it] node [below] {$r$} (n2);
	  \draw[line width=0.5mm,blue]  (n1) edge[] node [below] {$b$} (n2)  ;
		\draw[line width=0.5mm,blue]  (n1) edge[] node [below] {$b$} (n3);
		\draw[line width=0.5mm,red]  (n2) edge[] node [right] {$r$} (n3);	
				
	\node[shape=circle,draw=black,minimum size = 14pt,inner sep=0.3pt] (n11) at (3,1) {$x$};
	\node[shape=circle,draw=black,minimum size = 14pt,inner sep=0.3pt] (n21) at (4,1) {$y$};
	\node[shape=circle,draw=black,minimum size = 14pt,inner sep=0.3pt] (n31) at (4,0) {$z$};

		\path[line width=0.5mm,red] (n11) edge[bend left=120, snake it] node [below] {$r$} (n21);
	  \draw[line width=0.5mm,blue]  (n11) edge[] node [below] {$b$} (n21)  ;
		\draw[line width=0.5mm,red]  (n11) edge[] node [below] {$r$} (n31);
		\draw[line width=0.5mm,blue]  (n21) edge[] node [right] {$b$} (n31);		
		
		
\end{tikzpicture}
\caption{One edge in $\Gamma_r$ and two edges in $\Gamma_b$} \label{fig8}
\end{figure}

\end{proof}

\section{Generators for ${\Lambda}^{S^2}_{V_d}(2d+1)$}

Recall the definition of the elements $E_d$.

\begin{center} $E_1=\begin{pmatrix}
1& e_1\\
\wedge&1
\end{pmatrix}\in {\Lambda}^{S^2}_{V_1}(3), $\;$ $\;$ $\;$ E_d=\begin{pmatrix}
E_{d-1}& A_{d}\\
\wedge & \begin{pmatrix}
1& e_d\\
&1
\end{pmatrix}
\end{pmatrix}\in {\Lambda}^{S^2}_{V_d}(2d+1),$ \end{center}
where
$$A_{d}=\otimes \begin{pmatrix}
e_d&e_1\\
e_1&e_d\\
e_d&e_2\\
e_2&e_d\\
.&.\\
.&.\\
e_d&e_{d-1}\\
e_{d-1}&e_d
\end{pmatrix}\in {V_d}^{\otimes 2(2d-2)}.$$

\begin{remark} It is obvious that $\{E_1\}$ is a basis for ${\Lambda}^{S^2}_{V_1}(3)$. One can see that
$$E_2=\begin{pmatrix}
1& e_1&\boxed{e_2}&\boxed{e_1}\\
&1&e_1&e_2\\
& &1&\boxed{e_2}\\
\wedge&&&1
\end{pmatrix}=-\begin{pmatrix}
1& e_1&e_2&e_2\\
&1&e_1&e_2\\
& &1&e_1\\
\wedge&&&1
\end{pmatrix},$$
and so from \cite{sta2} we know that $\{ E_2\} $ is  a basis  for ${\Lambda}^{S^2}_{V_2}(5)$.
\end{remark}

Next we need the following lemmas.
\begin{lemma}
For every $1\leq i<j\leq 2d$, and for all $1\leq k\leq d$, there exists $$Z_{i,j;k}=\begin{pmatrix}
1& v_{1,2}&...&v_{1,2d-1}&v_{1,2d}\\
& 1&...&v_{2,2d-1}&v_{2,2d}\\
& &...&.&.\\
& & &1&v_{2d-1,2d}\\
\otimes & & & &1
\end{pmatrix}\in \mathcal{G}_{\mathcal{B}_d}(2d+1)$$ such that $v_{i,j}=e_k$, and the image of $Z_{i,j;k}$ in
${\Lambda}^{S^2}_{V_d}(2d+1)$   is $E_d$,   $$\hat{Z}_{i,j;k}=\begin{pmatrix}
1& v_{1,2}&...&v_{1,2d-1}&v_{1,2d}\\
& 1&...&v_{2,2d-1}&v_{2,2d}\\
& &...&.&.\\
& & &1&v_{2d-1,2d}\\
\wedge& & & &1
\end{pmatrix}=E_d\in {\Lambda}^{S^2}_{V_d}(2d+1).$$
\label{lemma3x}
\end{lemma}

\begin{proof} Follows directly from Lemma \ref{GLLem}, and Lemma \ref{actions2n}.
\end{proof}

\begin{lemma} Let $\gamma\in {\Lambda}^{S^2}_{V_d}(2d+1)$ such that
\begin{eqnarray}
\gamma=\begin{pmatrix}
E_{d-1}& B\\
 \wedge& \begin{pmatrix}
1& e_d\\
&1
\end{pmatrix}
\end{pmatrix},
\label{gam}
\end{eqnarray}
for some
$$B= \otimes \begin{pmatrix} b_{1,1}&b_{1,2}\\
b_{2,1}&b_{2,2}\\
b_{3,1}&b_{3,2}\\
.&.\\
.&.\\
b_{2d-3,1}&b_{2d-3,2}\\
b_{2d-2,1}&b_{2d-2,2}
\end{pmatrix}\in {V_d}^{\otimes 2(2d-2)}.$$
 Then, $\gamma$ is a multiple of $E_d$.
\label{lemmagam}
\end{lemma}

\begin{proof} First take  $X=\begin{pmatrix}
1& x_{1,2}&...&x_{1,2d-1}&x_{1,2d}\\
& 1&...&x_{2,2d-1}&x_{2,2d}\\
& &...&.&.\\
& & &1&x_{2d-1,2d}\\
& & & &1
\end{pmatrix}\in V^{\otimes d(2d-1)}$ such that $\hat{X}=\gamma$. During the proof we will use several representatives for $\gamma$ that are convenient for specific computations. For example we may assume that $x_{2d-1,2d}=e_d$ (this follows directly from the assumption on $\gamma$).

Without loss of generality, we may assume that all entries in $X$ and $B$ are elements of the basis $\{e_1,e_2,...,e_d\}$. We also assume that $\hat{X}$ is not zero in ${\Lambda}^{S^2}_{V_d}(2d+1)$ (otherwise the result is trivial). Since in $E_{d-1}$ we already have $2d-3$ entries equal to $e_k$ for all $1\leq k\leq d-1$,  by Lemma \ref{lemma2} we know that the $4d-4$ entries in the tensor matrix $B$ must be: $e_d$ with multiplicity $2d-2$, and $e_k$  with multiplicity $2$ (for all $1\leq k\leq d-1$).

If there exists an $1\leq i\leq 2d-2$ such that  $x_{i,2d-1}=b_{i,1}=e_d$ and $x_{i,2d}=b_{i,2}=e_d$ then  $\hat{X}=0$ (that is because $x_{2d-1,2d}=e_d$).  So we may assume that there is exactly one $e_d$ in each row of the tensor matrix $B$.

If there exist $1\leq i<j\leq 2d-2$ and $1\leq k\leq d-1$ such that $x_{i,2d-1}=b_{i,1}=e_k=b_{j,1}=x_{j,2d-1}$, then again $\hat{X}=0$. To prove this we use Lemma \ref{lemma3x} to find a representative for $E_{d-1}$ such that the entry in the position $(i,j)$ is $e_k$. This gives a representative for $\gamma$ that  has the entry $e_k$ on positions $(i,j)$, $(i,2d-1)$ and $(j,2d-1)$, which implies that $\gamma=0$. One has a similar statement for the second column of $B$. So, we may assume that for each $1\leq k\leq d-1$, there is exactly one $e_k$ in the first column of $B$, and exactly one $e_k$ in the second column of $B$.

To summarize, we may assume that in each row of the matrix $B$ there is exactly one entry equal to $e_{d}$ and in each column of $B$ we have $d-1$ entries equal to $e_d$, and one entry equal to $e_k$, for all $1\leq k\leq d-1$. This accounts for all the entries in $B$.

Claim 1:  If $x_{i,2d-1}=e_s$ and $x_{j,2d-1}=e_t$, for some $1\leq i<j\leq 2d-2$ and $1\leq s\neq t\leq d-1$, we can interchange $e_s$ and $e_t$ such that Equation \ref{gam} still holds true up to a minus sign. Indeed, using Lemma \ref{lemma3x}, we may assume that $x_{i,j}=e_s$.  Now, we use the identity
$$\begin{pmatrix}
1& e_s&e_s\\
&1&e_t\\
\otimes& &1
\end{pmatrix}+\begin{pmatrix}
1& e_t&e_s\\
&1&e_s\\
\otimes& &1
\end{pmatrix}+\begin{pmatrix}
1& e_s&e_t\\
&1&e_s\\
\otimes& &1
\end{pmatrix}=0,$$
for positions $(i,j)$, $(i,2d-1)$ and $(j,2d-1)$. Notice that by  Lemma \ref{lemma2} the second term will become zero in ${\Lambda}^{S^2}_{V_{d}}(2d+1)$, because its corresponding matrix has $2d-2$ entries equal to $e_t$ in the first $2d-2$ columns.
This means that we can interchange $b_{i,1}=e_s$ and $b_{j,1}=e_t$ in the tensor matrix $B$ with the price of changing $\gamma$ by a minus sign. A similar statement is true for the second column in $B$.

Claim 2:  Assume that $b_{i,1}=e_d=b_{i+1,2}$ and $b_{i,2}=e_k=b_{i+1,1}$ for some $1\leq i<2d-2$, and some $1\leq k\leq d-1$.  Then we can find another tensor matrix $B'$ that satisfies Equation \ref{gam} (up to a minus sign), and is obtained from $B$
by changing the rows $i$ and $i+1$ according to the following rule, if
$$B=\otimes\begin{pmatrix}
*&*\\
b_{i,1}& b_{i,2}\\
b_{i+1,1}&b_{i+1,2}\\
*&*
\end{pmatrix}=\otimes\begin{pmatrix}
*&*\\
e_d& e_k\\
e_k&e_d\\
*&*
\end{pmatrix},$$  then
$$ B'=\otimes\begin{pmatrix}
*&*\\
b'_{i,1}& b'_{i,2}\\
b'_{i+1,1}&b'_{i+1,2}\\
*&*
\end{pmatrix}=\otimes\begin{pmatrix}
*&*\\
e_k& e_d\\
e_d&e_k\\
*&*
\end{pmatrix}.$$

To prove this claim notice that by Lemma \ref{lemma3x} we may assume $x_{i,i+1}=e_k$. Then from the Equation
(\ref{act1}) we use the identity
\begin{eqnarray*}
\begin{pmatrix}
1& e_k&e_d&e_k\\
&1&e_k&e_d\\
& &1&e_d\\
\wedge&&&1
\end{pmatrix}=
-\begin{pmatrix}
1& e_k&e_k&e_d\\
&1&e_d&e_k\\
& &1&e_d\\
\wedge&&&1
\end{pmatrix},
\label{eq12}
\end{eqnarray*}
on the positions $(i,i+1)$, $(i,2d-1)$, $(i,2d)$, $(i+1,2d-1)$, $(i+1,2d)$, and $(2d-1,2d)$ to prove our statement.

Using repeatedly Claim 1 and Claim 2, we can move the entries $e_d$ in the positions $(1,2d-1)$, $(2,2d)$, $(3,2d-1)$, $(4,2d)$,..., $(2d-3,2d-1)$, and $(2d-2,2d-1)$. Finally, using Claim 1, for every $1\leq k\leq d-1$ we can move $e_k$ to the positions $(2k-1,2d)$ and $(2k,2d+1)$, which  means that up to a constant $\gamma=E_d$.
\end{proof}
We are now ready to give a single generator for the case $dim (V_3)=3$.
\begin{proposition}\label{dimless1}
Any element $\hat{x}\in {\Lambda}^{S^2}_{V_3}(7)$ can be written as  $\hat{x}=\begin{pmatrix}
y& B\\
 \wedge& \begin{pmatrix}
1& e_3\\
&1
\end{pmatrix}
\end{pmatrix}$, for some $\hat{y}\in {\Lambda}^{S^2}_{V_{2}}(5)$, and $B\in V_3^{\otimes 2\cdot 4}={V_3}^{\otimes 2(2\cdot 3-2)}$. In particular, we have that $\hat{x}$ is a multiple of $E_3$ so $dim({\Lambda}^{S^2}_{V_3}(7)\leq 1$.
\end{proposition}
\begin{proof}


Essentially, we want to find an element equivalent to $\hat{x}$, which has every entry $e_3$ in the last two columns of its matrix. Without loss of generality, we can assume that $x\in \hat{\mathcal{G}}_{\mathcal{B}_3}^{cf}(7)$, where
$$\hat{\mathcal{G}}_{\mathcal{B}_d}^{cf}(7)=\{\hat{f}_{(\Gamma_1,\Gamma_2\Gamma_3)}\vert (\Gamma_1,\Gamma_2,\Gamma_3) {\rm ~is~ a ~homogeneous,~cycle}{\text -}{\rm free~}d{\text -}{\rm partition~of~}K_{6}\}.$$


In particular, we may assume that the entries of $x$ consist of five $e_1$'s, $e_2$'s and $e_3$'s.  Recall from Lemma \ref{lemma5E} that there are only six types of cycle-free graphs with 6 vertices and 5 edges, namely $I_6$, $Y_6$, $E_6$, $H_6$, $C_6$, and $S_6$.

Note that $\Gamma_3$ cannot  be of the form $S_6$ or $C_6$. Indeed, in both of these cases we would have a vertex $v_0$ in $\Gamma_3$ which is connected to at least 4 other vertices. This would imply that  either $\Gamma_1$ or $\Gamma_2$ will not have an edge connected to $v_0$, and so by Lemma \ref{lemmagraph} they will have a cycle, which is an contradiction with the fact that  $(\Gamma_1,\Gamma_2,\Gamma_3)$ is cycle-free.

If $\Gamma_3$ is $H_6$, then after acting with a permutation in $S_6$ we can realize it as in Figure \ref{4easy}, and so  $e_3$ is  located in positions on the last two columns of matrix $x$.

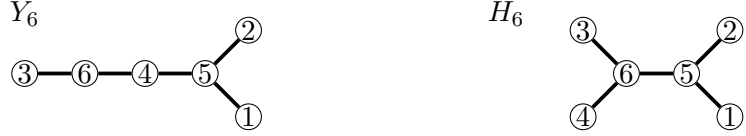
\begin{figure}[h]
\centering

\begin{tikzpicture}
  [scale=0.8,auto=left]
  \node[shape=circle,draw=black,minimum size = 5pt,inner sep=0.3pt] (n1) at (0,0) {3};
  \node[shape=circle,draw=black,minimum size = 5pt,inner sep=0.3pt] (n2) at (1,0)  {6};
  \node[shape=circle,draw=black,minimum size = 5pt,inner sep=0.3pt] (n3) at (2,0)  {4};
  \node[shape=circle,draw=black,minimum size = 5pt,inner sep=0.3pt] (n4) at (3,0)  {5};
	\node[shape=circle,draw=black,minimum size = 5pt,inner sep=0.3pt] (n5) at (3.72,-0.72)   {1};
	\node[shape=circle,draw=black,minimum size = 5pt,inner sep=0.3pt] (n6) at (3.72,0.72)   {2};
	\node[shape=circle,minimum size = 14pt,inner sep=0.3pt] (n24) at (0,1) {$Y_6$};
  \foreach \from/\to in {n1/n2,n2/n3,n3/n4,n4/n5,n4/n6}
    \draw[line width=0.5mm]  (\from) -- (\to);	
		
  \node[shape=circle,draw=black,minimum size = 5pt,inner sep=0.3pt]  (n11) at (9.28,-0.72) {4};
  \node[shape=circle,draw=black,minimum size = 5pt,inner sep=0.3pt]  (n21) at (9.28,0.72)  {3};
  \node[shape=circle,draw=black,minimum size = 5pt,inner sep=0.3pt]  (n31) at (10,0)  {6};
  \node[shape=circle,draw=black,minimum size = 5pt,inner sep=0.3pt]  (n41) at (11,0)  {5};
	\node[shape=circle,draw=black,minimum size = 5pt,inner sep=0.3pt]  (n51) at (11.72,-0.72)   {1};
	\node[shape=circle,draw=black,minimum size = 5pt,inner sep=0.3pt]  (n61) at (11.72,0.72)   {2};
	\node[shape=circle,minimum size = 14pt,inner sep=0.3pt] (n24) at (8,1) {$H_6$};
  \foreach \from/\to in {n11/n31,n21/n31,n31/n41,n41/n51,n41/n61}
    \draw[line width=0.5mm]  (\from) -- (\to);	

\end{tikzpicture}

\caption{Cycle-free graphs $Y_6$ and $H_6$ with $e_3$ in the last two columns} \label{4easy}
\end{figure}

Moreover, in this situation we have $$x=\begin{pmatrix}
y& B\\
 \wedge & \begin{pmatrix}
1& e_3\\
&1
\end{pmatrix}
\end{pmatrix},$$ with $y$ being a multiple of $E_2$ and, by Lemma \ref{lemmagam}, we get that $\hat{x}$ is a multiple of $E_3$.

If $\Gamma_3$ is $Y_6$, after acting with a permutation in $S_6$ we can realize it as in Figure \ref{4easy}. Then we can apply Lemma \ref{keylemma} to face $(4, 5, 6)$ to obtain an element equivalent to $-x$, which has $e_3$ in position $(5, 6)$, which  from Lemma \ref{lemmagam} gives again that $\hat{x}$ is a multiple of $E_3$.


If $\Gamma_3$, is of type $I_6$, then after acting with element in $S_6$ we can assume that we have the labeling from Figure \ref{easy123} (left most graph).  Using Lemma \ref{keylemma} for face $(3, 4, 5)$, we can reduce the problem to the case when $\Gamma_3$ is either of type $Y_6$ or $E_6$ (see Figure \ref{easy123}).

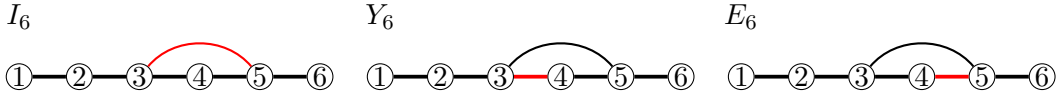
\begin{figure}[ht]
\centering
\begin{tikzpicture}
  [scale=0.8,auto=left]
  \node[shape=circle,draw=black,minimum size = 5pt,inner sep=0.3pt] (n1) at (0,0) {1};
  \node[shape=circle,draw=black,minimum size = 5pt,inner sep=0.3pt] (n2) at (1,0)  {2};
  \node[shape=circle,draw=black,minimum size = 5pt,inner sep=0.3pt] (n3) at (2,0)  {3};
  \node[shape=circle,draw=black,minimum size = 5pt,inner sep=0.3pt] (n4) at (3,0)  {4};
	\node[shape=circle,draw=black,minimum size = 5pt,inner sep=0.3pt] (n5) at (4,0)   {5};
	\node[shape=circle,draw=black,minimum size = 5pt,inner sep=0.3pt] (n6) at (5,0)   {6};
	\node[shape=circle,minimum size = 14pt,inner sep=0.3pt] (n24) at (0,1) {$I_6$};
  \foreach \from/\to in {n1/n2,n2/n3,n3/n4,n4/n5,n5/n6}
    \draw[line width=0.5mm]  (\from) -- (\to);	
    \path[thick,red] (n3) edge[bend left=50] node [above] {} (n5);

 \node[shape=circle,draw=black,minimum size = 5pt,inner sep=0.3pt] (n1) at (6,0) {1};
  \node[shape=circle,draw=black,minimum size = 5pt,inner sep=0.3pt] (n2) at (7,0)  {2};
  \node[shape=circle,draw=black,minimum size = 5pt,inner sep=0.3pt] (n3) at (8,0)  {3};
  \node[shape=circle,draw=black,minimum size = 5pt,inner sep=0.3pt] (n4) at (9,0)  {4};
	\node[shape=circle,draw=black,minimum size = 5pt,inner sep=0.3pt] (n5) at (10,0)   {5};
	\node[shape=circle,draw=black,minimum size = 5pt,inner sep=0.3pt] (n6) at (11,0)   {6};
	\node[shape=circle,minimum size = 14pt,inner sep=0.3pt] (n24) at (6,1) {$Y_6$};
  \foreach \from/\to in {n1/n2,n2/n3,n4/n5,n5/n6}
    \draw[line width=0.5mm]  (\from) -- (\to);	
    \path[thick,black] (n3) edge[bend left=50] node [above] {} (n5);
    \draw[line width=0.5mm,red] (n3) -- (n4);

 \node[shape=circle,draw=black,minimum size = 5pt,inner sep=0.3pt] (n1) at (12,0) {1};
  \node[shape=circle,draw=black,minimum size = 5pt,inner sep=0.3pt] (n2) at (13,0)  {2};
  \node[shape=circle,draw=black,minimum size = 5pt,inner sep=0.3pt] (n3) at (14,0)  {3};
  \node[shape=circle,draw=black,minimum size = 5pt,inner sep=0.3pt] (n4) at (15,0)  {4};
	\node[shape=circle,draw=black,minimum size = 5pt,inner sep=0.3pt] (n5) at (16,0)   {5};
	\node[shape=circle,draw=black,minimum size = 5pt,inner sep=0.3pt] (n6) at (17,0)   {6};
	\node[shape=circle,minimum size = 14pt,inner sep=0.3pt] (n24) at (12,1) {$E_6$};
  \foreach \from/\to in {n1/n2,n2/n3,n3/n4,n5/n6}
    \draw[line width=0.5mm]  (\from) -- (\to);	
    \path[thick,black] (n3) edge[bend left=50] node [above] {} (n5);
    \draw[line width=0.5mm,red] (n4) -- (n5);
\end{tikzpicture}
\caption{Reduction of $I_6$ to $Y_6$ or $E_6$}\label{easy123}
\end{figure}

 Therefore, we only need to justify the case $\Gamma_3=E_6.$ By using an appropriate permutation, we may assume that $\Gamma_3$ is as in Figure \ref{hard}.

 \begin{figure}[h]
 \centering

\begin{tikzpicture}
  [scale=0.8,auto=left]

 \node[shape=circle,draw=black,minimum size = 5pt,inner sep=0.3pt] (n1) at (6,0) {6};
  \node[shape=circle,draw=black,minimum size = 5pt,inner sep=0.3pt] (n2) at (7,0)  {4};
  \node[shape=circle,draw=black,minimum size = 5pt,inner sep=0.3pt] (n3) at (8,0)  {1};
  \node[shape=circle,draw=black,minimum size = 5pt,inner sep=0.3pt] (n4) at (8,1)  {3};
	\node[shape=circle,draw=black,minimum size = 5pt,inner sep=0.3pt] (n5) at (9,0)   {2};
	\node[shape=circle,draw=black,minimum size = 5pt,inner sep=0.3pt] (n6) at (10,0)   {5};
	\node[shape=circle,minimum size = 14pt,inner sep=0.3pt] (n24) at (6,1) {$E_6$};
  \foreach \from/\to in {n1/n2,n2/n3,n3/n4,n3/n5,n5/n6}
    \draw[line width=0.5mm]  (\from) -- (\to);

\end{tikzpicture}
\caption{$\Gamma_3=E_6$}\label{hard}
\end{figure}
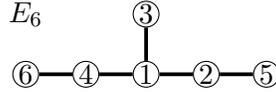

Then, element $\hat{x}$ has entries

$$\hat{x}=\begin{pmatrix}
	1& e_3& e_3& e_3& * & *\\
	&1&*&*& *&*\\
	& &1&*&e_3&*\\
	&&&1& *&e_3\\
	&&&&1&*\\
	\wedge&&&&&1
\end{pmatrix}.$$

The remaining ten positions need to be filled by five $e_1$'s and five $e_2$'s. Without the loss of generality, we will assume that in position $(3,4)$ we have $e_2$.

If the element in position $(4,5)$ is also $e_2$, then we have

$$x=\begin{pmatrix}
	1& e_3& e_3& e_3& * & *\\
	&1&*&*& *&*\\
	& &1&\boxed{e_2}&\boxed{e_3}&*\\
	&&&1& \boxed{e_2}&e_3\\
	&&&&1&*\\
	\wedge&&&&&1
\end{pmatrix}=-\begin{pmatrix}
	1& e_3& e_3& e_3& * & *\\
	&1&*&*& *&*\\
	& &1&e_2&e_2&*\\
	&&&1& e_3&e_3\\
	&&&&1&*\\
	\wedge&&&&&1
\end{pmatrix}-\boxed{\begin{pmatrix}
	1& e_3& \colorbox{orange}{$e_3$}& \colorbox{orange}{$e_3$}& * & *\\
	&1&*&*& *&*\\
	& &1&\colorbox{orange}{$e_3$}&e_2&*\\
	&&&1& e_2&e_3\\
	&&&&1&*\\
	\wedge&&&&&1
\end{pmatrix}}.$$

The first matrix has $\Gamma_3=H$ and the second is zero since $\Gamma_3$ has the cycle $(1,3,4)$, so the Lemma follows from our previous discussion.
If the element in position $(3,6)$ is $e_2$, the same relation, applied to face $(3,4,6)$, results in an identical situation. Thus, we may assume that $e_1$ is in positions $(4,5)$ and $(3,6)$.

If $e_1$ is in position $(5,6)$, by applying the triangle identity to the face $(4,5,6)$, we get
$$x=\begin{pmatrix}
	1& e_3& e_3& e_3& * & *\\
	&1&*&*& *&*\\
	& &1&e_2&e_3&e_1\\
	&&&1&\boxed{ e_1}&\boxed{e_3}\\
	&&&&1&\boxed{e_1}\\
	\wedge&&&&&1
\end{pmatrix}=-\begin{pmatrix}
	1& e_3& e_3& e_3& * & *\\
	&1&*&*& *&*\\
	& &1&e_2&e_3&e_1\\
	&&&1& e_1&e_1\\
	&&&&1&e_3\\
	\wedge&&&&&1
\end{pmatrix}-\boxed{\begin{pmatrix}
	1& e_3& \colorbox{orange}{$e_3$}& \colorbox{orange}{$e_3$}& * & *\\
	&1&*&*& *&*\\
	& &1&e_2&\colorbox{orange}{$e_3$}&e_1\\
	&&&1& \colorbox{orange}{$e_3$}&e_1\\
	&&&&1&e_1\\
	\wedge&&&&&1
\end{pmatrix}},$$
where the first matrix has $\Gamma_3=Y_6$, and the second one being zero since its $\Gamma_3$ has the $4$-cycle $(1,3,5,4)$. Again, we get that $x$ can be written as a multiple of $E_3$.

Thus, it remains to justify the case when $e_2$ is in position $(5,6)$, so $$x=\begin{pmatrix}
	1& e_3& e_3& e_3& * & *\\
	&1&*&*& *&*\\
	& &1&e_2&e_3&e_1\\
	&&&1& e_1&e_3\\
	&&&&1&e_2\\
	\wedge&&&&&1
\end{pmatrix}.$$


Now, using Lemma \ref{relemma1} (b), for the face $(3,4,5)$, we can express $x$ as a difference of five terms, one for each different permutation of the elements on positions $(3,4)$, $(3,5),$ and $(4,5)$.
\begin{eqnarray*}
x=\begin{pmatrix}
	1& e_3& e_3& e_3& * & *\\
	&1&*&*& *&*\\
	& &1&\boxed{e_2}&\boxed{e_3}&e_1\\
	&&&1& \boxed{e_1}&e_3\\
	&&&&1&e_2\\
	\wedge&&&&&1
\end{pmatrix}=-\begin{pmatrix}
	1& e_3& e_3& e_3& * & *\\
	&1&*&*& *&*\\
	& &1&e_2&e_1&e_1\\
	&&&1& \colorbox{green}{$e_3$}&e_3\\
	&&&&1&e_2\\
	\wedge&&&&&1
\end{pmatrix}-\boxed{\begin{pmatrix}
	1& e_3& \colorbox{orange}{$e_3$}& \colorbox{orange}{$e_3$}& * & *\\
	&1&*&*& *&*\\
	& &1&\colorbox{orange}{$e_3$}&e_2&e_1\\
	&&&1& e_1&e_3\\
	&&&&1&e_2\\
	\wedge&&&&&1
\end{pmatrix}}\\
-\begin{pmatrix}
	1& e_3& e_3& e_3& * & *\\
	&1&*&*& *&*\\
	& &1&e_1&e_2&e_1\\
	&&&1& \colorbox{green}{$e_3$}&e_3\\
	&&&&1&e_2\\
	\wedge&&&&&1
\end{pmatrix}
-\begin{pmatrix}
	1& e_3& e_3& e_3& * & *\\
	&1&*&*& *&*\\
	& &1&e_1&e_3&e_1\\
	&&&1& e_2&e_3\\
	&&&&1&e_2\\
	\wedge&&&&&1
\end{pmatrix}
-\boxed{\begin{pmatrix}
	1& e_3& \colorbox{orange}{$e_3$}& \colorbox{orange}{$e_3$}& * & *\\
	&1&*&*& *&*\\
	& &1&\colorbox{orange}{$e_3$}&e_1&e_1\\
	&&&1& e_2&e_3\\
	&&&&1&e_2\\
	\wedge&&&&&1
\end{pmatrix}}
\end{eqnarray*}
Note that when $e_3$ occupies position $(3,4)$, $\Gamma_3$ has the cycle $(1,3,4)$, so the corresponding two matrices are zero in  ${\Lambda}^{S^2}_{V_3}(7)$. When $e_3$ is in position $(4,5)$, we have $\Gamma_3=H_6$, so the corresponding two matrices are multiples of $E_3$. Finally, the last matrix is
$$
\begin{pmatrix}
	1& e_3& e_3& e_3& * & *\\
	&1&*&*& *&*\\
	& &1&e_1&e_3&e_1\\
	&&&1& e_2&e_3\\
	&&&&1&e_2\\
	\wedge&&&&&1
\end{pmatrix}
$$

Using the identity (\ref{equ1}) for face $(4, 5, 6)$ we have

\begin{eqnarray*}
\begin{pmatrix}
	1& e_3& e_3& e_3& * & *\\
	&1&*&*& *&*\\
	& &1&e_1&e_3&e_1\\
	&&&1& e_2&e_3\\
	&&&&1&e_2\\
	\wedge&&&&&1
\end{pmatrix}=-\begin{pmatrix}
	1& e_3& e_3& e_3& * & *\\
	&1&*&*& *&*\\
	& &1&e_1&e_3&e_1\\
	&&&1& e_2&e_2\\
	&&&&1&e_3\\
	\wedge&&&&&1
\end{pmatrix}-\boxed{\begin{pmatrix}
	1& e_3& \colorbox{orange}{$e_3$}& \colorbox{orange}{$e_3$}& * & *\\
	&1&*&*& *&*\\
	& &1&e_1&\colorbox{orange}{$e_3$}&e_1\\
	&&&1& \colorbox{orange}{$e_3$}&e_2\\
	&&&&1&e_2\\
	\wedge&&&&&1
\end{pmatrix}}
\end{eqnarray*}

For the first matrix $\Gamma_3=Y_6$ (which we already discussed), and for the second $\Gamma_3$ has the cycle $(1,3,5,4)$, thus the corresponding element is $0$.

\end{proof}

\begin{remark}
The results in this section could be used to prove that $dim({\Lambda}^{S^2}_{V_d}(2d+1))\leq 1$, for every $d\geq 1$. The only ingredient that is missing is showing that, modulo the equivalence relation defined by Lemma \ref{keylemma}, every homogeneous cycle-free partition of $K_{2d}$ is equivalent with a partition $(\Gamma_1,\dots, \Gamma_{d-1},\Gamma_d)$ such that $\Gamma_d$ is a twin star (i.e. there exists two vertices $t_1$ and $t_2$ such that every edge in $\Gamma_d$ is connected to $t_1$ or $t_2$).
\end{remark}

\section{A determinant like function}

\subsection{Revisiting the case $d=2$}
We denote by $\mathcal{P}^{h,cf}_d(K_{2d})$ the set of homogeneous, cycle-free $d$-partitions for $K_{2d}$. Notice that on  $\mathcal{P}_d^{h,cf}(K_{2d})$ there are two natural actions of the groups $S_{2d}$ and $S_d$. For $\sigma\in S_{2d}$,  and $(\Gamma_1,...,\Gamma_d)\in \mathcal{P}_d^{h,cf}(K_{2d})$, we define
$$\sigma*(\Gamma_1,...,\Gamma_d)=(\sigma*\Gamma_1,...,\sigma*\Gamma_d),$$
where for a subgraph $\Gamma$ of $K_{2d}$ with $E(\Gamma)=\{(i_1,j_1),...,(i_{2d-1},j_{2d-1})\}$ we take  the edges of $\sigma*\Gamma$ to be
$E(\sigma*\Gamma)=\{(\sigma(i_1),\sigma(j_1)),...,(\sigma(i_{2d-1}),\sigma(j_{2d-1}))\}$.

For $\tau\in S_d,$ we define $$\tau*(\Gamma_1,...,\Gamma_d)=(\Gamma_{\tau^{-1}(1)},...,\Gamma_{\tau^{-1}(d)}).$$
These two actions commute with each other, so they can be combined in an action of the group $S_{2d}\times S_d$ on $\mathcal{P}_d^{h,cf}(K_{2d})$.

Finally, from Lemma \ref{keylemma}, we know that for every $1\leq x\leq y\leq z\leq 2d$, there is an involution on $\mathcal{P}^{h,cf}_d(K_{2d})$ given by $$(\Gamma_1,...,\Gamma_d)\mapsto (\Gamma_1,...,\Gamma_d)^{(x,y,z)}.$$

When $d=2$ one can see that $\mathcal{P}^{h,cf}_2(K_{4})$ has $12$ elements. One of these elements is the partition $P_0=\Gamma(f)$  presented in Figure \ref{fig1}.  Notice that $$Stab_{S_4\times  S_2}(P_{0})=\{e_{S_4}\times e_{S_2}, (1,2)(3,4)\times e_{S_2}, (1,4,2,3)\times (1,2), (1,3,2,4)\times (1,2)\}.$$ 
In particular, this means that the orbit of $P_0$ has $12$ elements and so it is equal to $\mathcal{P}^{h,cf}_2(K_{4})$. One can define $\varepsilon_2^{S^2}: \mathcal{P}^{h,cf}_2(K_{4})\to \{1, -1\},$
determined by $$\varepsilon_2^{S^2}((\sigma\times \tau)\cdot P_0)=\varepsilon_{S_4}(\sigma)\varepsilon_{S_2}(\tau)
,$$ which is well defined because $\varepsilon_2^{S^2}(Stab_{S_4\times  S_2}(P_{0}))=1$.

With these notations we are ready to reformulate the definition of $det^{S^2}$ from Remark \ref{rem14}.
\begin{remark}
Take $v_{i,j}=\alpha_{i,j}e_1+\beta_{i,j}e_2\in V_2$ for all $1\leq i<j\leq 4$. For a partition $P=(\Gamma_1,\Gamma_2)\in \mathcal{P}^{h,cf}_2(K_{4})$ we define
$$M_{(\Gamma_1,\Gamma_2)}((v_{i,j})_{1\leq i<j\leq 4})=\prod_{(u_1,v_1)\in E(\Gamma_{1})}\alpha_{u_1,v_1}
\prod_{(u_2,v_2)\in E(\Gamma_{2})}\beta_{u_2,v_2}.$$ Next we take $Det^{S^2}:V_2^6\to k$ determined by
$$Det^{S^2}((v_{i,j})_{1\leq i<j\leq 4})=\sum_{(\Gamma_1,\Gamma_2)\in \mathcal{P}^{h,cf}_2(K_{4})} \varepsilon_2^{S^2}((\Gamma_1,\Gamma_2))M_{(\Gamma_1,\Gamma_2)}((v_{i,j})_{1\leq i<j\leq 4}).$$
The map $Det^{S^2}$ is the unique multi-linear map on $V_2^6$ with the property that $Det^{S^2}((v_{i,j})_{1\leq i<j\leq 4})=0$, if there exist $1\leq x<y<z\leq 4$ such that $v_{x,y}=v_{x,z}=v_{y,z}$,  and $Det^{S^2}(E_2)=1$.  With the notation from section one $Det^{S^2}$ induces the map $-det^{S^2}:V_2^{\otimes 6}\to k$.
\end{remark}

\subsection{Main Result}
The case $d=3$ is much more complicated. One can show that $\mathcal{P}_3^{h,cf}(K_{6})$ has $\num{66240}$ elements, so the action of $S_6\times S_3$ is not transitive anymore (an explicit description for the equivalence classes is given in the Appendix).  In the rest of the subsection we present an explicit construction of the $det^{S^2}$ map in the case $d=3$. First we need the following result. 

\begin{theorem} There exits a map $\varepsilon^{S^2}:\mathcal{P}_3^{h,cf}(K_{6})\to \{-1,1\}$ such that
$$\varepsilon_3^{S^2}((\Gamma_1,\Gamma_2,\Gamma_3)^{(x,y,z)})=-\varepsilon_3^{S^2}((\Gamma_1,\Gamma_2,\Gamma_3)),$$
for all $1\leq x\leq y\leq z\leq 6$. \label{th1A}
\end{theorem}
\begin{proof} This result was established by direct computation using MATLAB. There are $\num{756756}$ homogeneous partitions of $K_6$, out of which $\num{66240}$ are cycle-free. Lemma \ref{keylemma} played a key role in checking the existence of the function $\varepsilon_3^{S^2}$. A explicit description of the map $\varepsilon_3^{S^2}$ is given in the Appendix.
\end{proof}

\begin{remark} What is really intriguing about the map $\varepsilon_3^{S^2}$ is its compatibility with the  intrinsic  transformations $(x,y,z)$  defined on $\mathcal{P}^{h,cf}_3(K_{6})$.  This exhibits a natural orientation on the set of homogeneous cycle-free $3$-partitions of $K_6$.  To be more precise, once we pick a partition we can say if a second partition has the same orientation or not.
It is similar with the way the signature map $\varepsilon_n:S_n\to \{-1,1\}$ splits the set of permutations  $S_n$ into even and odd permutations.
\end{remark}

Let $(\Gamma_1,\Gamma_2,\Gamma_3)\in \mathcal{P}_3^{h,cf}(K_{6})$ and $v_{i,j}=\alpha_{i,j}e_1+\beta_{i,j}e_2+\gamma_{i,j}e_3\in V_3$, for all $1\leq i<j\leq 6$. We define
$$M_{(\Gamma_1,\Gamma_2,\Gamma_3)}((v_{i,j})_{1\leq i<j\leq 6})=\prod_{(u_1,v_1)\in E(\Gamma_{1})}\alpha_{u_1,v_1}
\prod_{(u_2,v_2)\in E(\Gamma_{2})}\beta_{u_2,v_2}\prod_{(u_3,v_3)\in E(\Gamma_{3})}\gamma_{u_3,v_3}.$$
Obviously $M_{(\Gamma_1,\Gamma_2,\Gamma_3)}$ is multi-linear so we get a map from $V_3^{\otimes 15}\to k$.  Notice that for any $d$-partition $(\Delta_1,\Delta_2,\Delta_3)$ (homogeneous or not)
$$M_{(\Gamma_1,\Gamma_2,\Gamma_3)}(f_{(\Delta_1,\Delta_2,\Delta_3)})=\delta_{(\Gamma_1,\Gamma_2,\Gamma_3),(\Delta_1,\Delta_2,\Delta_3)},$$ 
in other words $M_{(\Gamma_1,\Gamma_2,\Gamma_3)}$ is a subset of the dual basis of $\mathcal{G}_{\mathcal{B}_3}(7)$. 

\begin{definition} Define $Det^{S^2}:V^{15}\to k$, determined by
\begin{eqnarray}
Det^{S^2}((v_{i,j})_{1\leq i<j\leq 6})=\sum_{(\Gamma_1,\Gamma_2,\Gamma_3)\in \mathcal{P}^{h,cf}_3(K_{6})} \varepsilon_3^{S^2}((\Gamma_1,\Gamma_2,\Gamma_3))M_{(\Gamma_1,\Gamma_2,\Gamma_3)}((v_{i,j})_{1\leq i<j\leq 6}).
\label{detS2d3}
\end{eqnarray}
\end{definition}

\begin{lemma} The map $Det^{S^2}:V^{15}\to k$ is  $k$-multi-linear,  and  $Det^{S^2}((v_{i,j})_{1\leq i<j\leq 6})=0$ for all $(v_{i,j})_{1\leq i<j\leq 6}\in V^{15}$ with the property that there exist   $1\leq x<y<z\leq 6$ such that $v_{x,y}=v_{x,z}=v_{y,z}$.
\label{lemmadet3}
\end{lemma}
\begin{proof} The fact that $Det^{S^2}$ is multi-linear follows because the monomial $M_{(\Gamma_1,\Gamma_2,\Gamma_3)}$ is linear in each $v_{i,j}$. This means that we get a map from  $V^{\otimes 15}$ to $k$, which by abuse of notation we denoted also by $Det^{S^2}$.

We only need to check that $Det^{S^2}$ is zero on the relations from Remark \ref{remba}. We will assume that all the $*$ entries are elements in the basis $\{e_1,e_2,e_3\}$, and so each term corresponds to a  
$3$-partition of the graph $K_{6}$. 

First consider a $3$-partition $(\Gamma_1,\Gamma_2, \Gamma_3)$ that has a cycle (not necessary of length $3$). 
Since in Equation \ref{detS2d3} we sum over cycle-free partition, we get that $Det^{S^2}(f_{(\Gamma_1,\Gamma_2, \Gamma_3)})=0$.  Relation (\ref{equ02}) corresponds to a $3$-partition of $K_{6}$ that has a $3$-cycle on the component $\Gamma_i$,  so $Det^{S^2}$  of the expression (\ref{equ02}) is zero.
\begin{eqnarray}\begin{pmatrix}
1&*&*&*&*&*\\
&*&\boxed{e_i}&*&\boxed{e_i}&*\\
&&*&*&*&*\\
&&&*&\boxed{e_i}&*\\
&&&&*&*\\
\otimes&&&&&1
\end{pmatrix}\label{equ02}
\end{eqnarray}

The terms in the relation (\ref{equ12})  correspond to three $3$-partitions of $K_6$ such that two edges of the face $(x,y,z)$ belong to $\Gamma_i$, and the third edge is in $\Gamma_j$. If all the corresponding partitions have a cycle, by the argument above, $Det^{S^2}$ will be zero on the expression (\ref{equ12}).  So we can assume that at least one of the terms in the expression  (\ref{equ12}) corresponds to a homogeneous cycle-free $3$-partition of $K_6$.  From Lemma \ref{keylemma} we know that there are exactly two such terms (which correspond to cycle-free $3$-partitions of $K_6$), let's call them  $(\Gamma_1, \Gamma_2,\Gamma_3)$ and
$(\Gamma_1, \Gamma_2,\Gamma_3)^{(x,y,z)}$. After evaluating $Det^{S^2}$ on the expression (\ref{equ12}) the only nonzero monomials are $M_{(\Gamma_1, \Gamma_2,\Gamma_3)}(f_{(\Gamma_1,\Gamma_2,\Gamma_3)})$ and $M_{(\Gamma_1, \Gamma_2,\Gamma_3)^{(x,y,z)}}(f_{(\Gamma_1, \Gamma_2,\Gamma_3)^{(x,y,z)}})$ (which are both equal to $1$). Finally, since $\varepsilon_3^{S^2}((\Gamma_1, \Gamma_2,\Gamma_3)^{(x,y,z)})=-\varepsilon_3^{S^2}((\Gamma_1, \Gamma_2,\Gamma_3))$, the two monomials will cancel each other, so $Det^{S^2}$ of the expression (\ref{equ12}) is zero.

\begin{eqnarray}\begin{pmatrix}
1&*&*&*&*&*\\
&*&\boxed{e_i}&*&\boxed{e_i}&*\\
&&*&*&*&*\\
&&&*&\boxed{e_j}&*\\
&&&&*&*\\
\otimes&&&&&1
\end{pmatrix}+
\begin{pmatrix}
1&*&*&*&*&*\\
&*&\boxed{e_i}&*&\boxed{e_j}&*\\
&&*&*&*&*\\
&&&*&\boxed{e_i}&*\\
&&&&*&*\\
\otimes&&&&&1
\end{pmatrix}+\begin{pmatrix}
1&*&*&*&*&*\\
&*&\boxed{e_j}&*&\boxed{e_i}&*\\
&&*&*&*&*\\
&&&*&\boxed{e_i}&*\\
&&&&*&*\\
\otimes&&&&&1
\end{pmatrix}\label{equ12}
\end{eqnarray}

The terms in the relation (\ref{equ22})  correspond to six $3$-partitions of $K_6$ such that the three edges of the face $(x,y,z)$ belong to distinct $\Gamma_i$, $\Gamma_j$ and $\Gamma_k$.  Just like above we can assume that at least one of the terms in the expression  (\ref{equ22}) corresponds to a homogeneous cycle-free $3$-partition of $K_6$.  From Lemma \ref{keylemma} we know that there are exactly two such terms, let's call them  $(\Gamma_1, \Gamma_2,\Gamma_3)$ and $(\Gamma_1, \Gamma_2,\Gamma_3)^{(x,y,z)}$. From here, the argument follows exactly like in the previous case and we get that  $Det^{S^2}$ of the expression (\ref{equ22}) is zero. 

\begin{eqnarray}\begin{pmatrix}
1&*&*&*&*&*\\
&*&\boxed{e_i}&*&\boxed{e_j}&*\\
&&*&*&*&*\\
&&&*&\boxed{e_k}&*\\
&&&&*&*\\
\otimes&&&&&1
\end{pmatrix}+
\begin{pmatrix}
1&*&*&*&*&*\\
&*&\boxed{e_i}&*&\boxed{e_k}&*\\
&&*&*&*&*\\
&&&*&\boxed{e_j}&*\\
&&&&*&*\\
\otimes&&&&&1
\end{pmatrix}+\begin{pmatrix}
1&*&*&*&*&*\\
&*&\boxed{e_j}&*&\boxed{e_i}&*\\
&&*&*&*&*\\
&&&*&\boxed{e_k}&*\\
&&&&*&*\\
\otimes&&&&&1
\end{pmatrix}\nonumber\\
\label{equ22} \\
+\begin{pmatrix}
1&*&*&*&*&*\\
&*&\boxed{e_k}&*&\boxed{e_i}&*\\
&&*&*&*&*\\
&&&*&\boxed{e_j}&*\\
&&&&*&*\\
\otimes&&&&&1
\end{pmatrix}+
\begin{pmatrix}
1&*&*&*&*&*\\
&*&\boxed{e_j}&*&\boxed{e_k}&*\\
&&*&*&*&*\\
&&&*&\boxed{e_i}&*\\
&&&&*&*\\
\otimes&&&&&1
\end{pmatrix}+\begin{pmatrix}
1&*&*&*&*&*\\
&*&\boxed{e_k}&*&\boxed{e_j}&*\\
&&*&*&*&*\\
&&&*&\boxed{e_i}&*\\
&&&&*&*\\
\otimes&&&&&1
\end{pmatrix}\nonumber
\end{eqnarray}

\end{proof}

To summarize, we proved the  conjecture from \cite{sta2} for vector spaces of dimension $3$.
\begin{theorem} Let $V_3$ be a vector space such that $dim_k(V_3)=3$. Then $dim_k({\Lambda}^{S^2}_{V_3}(7))=1$.
\label{th1B}
\end{theorem}
\begin{proof} From Lemma \ref{lemmadet3} we know that there exists a $k$-linear $det^{S^2}:{\Lambda}^{S^2}_{V_3}(7)\to k$ that corresponds to $Det^{S^2}$. The map $det^{S^2}$ is nontrivial since  $det^{S^2}(E_3)=1$, so we must have that $dim({\Lambda}^{S^2}_{V_3}(7))\geq 1$. Combining this with  Lemma \ref{dimless1}, we get our statement.
\end{proof}

\subsection{Some Remarks}

From Lemma \ref{keylemma}, we know that for  each $1\leq x< y< z\leq 2d$, there exists an involution $(x,y,x): \mathcal{P}^{h,cf}_d(K_{2d})\to \mathcal{P}^{h,cf}_d(K_{2d})$ defined by
$$(\Gamma_1,...,\Gamma_d)\mapsto (\Gamma_1,...,\Gamma_d)^{(x,y,z)}.$$
We consider the subgroup $G^{h,cf}_d$ of the symmetric group $S(\mathcal{P}_3^{h,cf}(K_{6}))$ generated by the involutions $(x,y,z)$ for all $1\leq x<y<z\leq 2d$. For example,  $G^{h,cf}_2$ is a subgroup of $S_{12}$ and
$G^{h,cf}_3$ is a subgroup in $S_{66240}$.

We have an obvious action of $G^{h,cf}_d$ on $\mathcal{P}^{h,cf}_d(K_{2d})$ which is induced by be the inclusion of $G^{h,cf}_d$ in $S(\mathcal{P}_3^{h,cf}(K_{6}))$. With this notation, we are now ready to reformulate the Conjecture from \cite{sta2}.

\begin{remark} Let $V_d$ a $k$-vector space such that $dim(V_d)=d$. \\
(i) If $G^{h,cf}_d$ acts transitively on $\mathcal{P}^{h,cf}_d(K_{2d})$ then $dim_k({\Lambda}^{S^2}_{V_d}(2d+1))\leq 1$.\\
(ii)  Suppose that exists a map $\varepsilon_d^{S^2}:\mathcal{P}_3^{h,cf}(K_{2d})\to \{-1,1\}$ such that
$$\varepsilon_d^{S^2}((\Gamma_1,...,\Gamma_d)^{(x,y,z)})=-\varepsilon((\Gamma_1,...,\Gamma_d)),$$
for all $1\leq x<y<z\leq 2d$. Then $dim_k({\Lambda}^{S^2}_{V_d}(2d+1))\geq 1$. \label{rem51}
\end{remark}
\begin{proof}
(i) It follows from the definition of  $G^{h,cf}_d$ and discussion in the previous section.

(ii)  If $v_{s,t}=\alpha_{s,t}^1e_1+\alpha_{s,t}^2e_2+...+\alpha_{s,t}^de_d\in V_d$ we define
$$M_{(\Gamma_1,...,\Gamma_3)}((v_{i,j})_{1\leq i<j\leq 2d})=\prod_{p=1}^d(\prod_{(u_p,v_p)\in E(\Gamma_{p})}\alpha_{u_p,v_p}^p).$$

Take the map $Det^{S^2}:V_d^{d(2d-1)}\to k$ determined by
\begin{equation}
Det^{S^2}((v_{i,j})_{1\leq i<j\leq 2d})=\sum_{(\Gamma_1,...,\Gamma_d)\in \mathcal{P}^{h,cf}_d(K_{2d})} \varepsilon_d^{S^2}((\Gamma_1,...,\Gamma_d))M_{(\Gamma_1,...,\Gamma_d)}((v_{i,j})_{1\leq i<j\leq 2d}).
\label{detS2d}
\end{equation}
Under the above assumptions, just like in Lemma \ref{lemmadet3}, it follows that $Det^{S^2}((v_{i,j})_{1\leq i<j\leq 2d})=0$ for all $(v_{i,j})_{1\leq i<j\leq 2d}\in V_d^{d(2d-1)}$ with the property that there exist $1\leq x<y<z\leq 2d$ such that $v_{x,y}=v_{x,z}=v_{y,z}$. This implies that ${\Lambda}^{S^2}_{V_d}(2d+1)\neq 0$.
\end{proof}

\begin{remark}
Notice that formula \ref{detS2d} is  based on edge partitions  of $K_{2d}$, so it is natural to ask what is the equivalent of this construction  if we consider vertex partitions for $K_d$. As one  probably expects, we  get the usual formula for the determinant of a square matrix.

Indeed, a vertex homogeneous $d$-partition of the graph $K_d$ is nothing else but an ordered partition of the set $\{1,2,...,d\}$ or, even better, a permutation of the set $\{1,2,...,d\}$. So one can think of  $\mathcal{P}^{h,cf}_d(K_{2d})$ as a generalization for $\mathcal{P}^{vt}_d(K_{d})=S_d$.  The action of $(x,y,z)$ on $\mathcal{P}^{h,cf}_d(K_{2d})$ can also be seen as a generalization of the action of a transposition $(x,y)$ on $\mathcal{P}^{vt}_d(K_{d})=S_d$. Finally, the map $\varepsilon_d^{S^2}:\mathcal{P}_3^{h,cf}(K_{2d})\to \{-1,1\}$ is replacing the usual signature map $\varepsilon_d:S_d\to \{-1,1\}$. Therefore, $det^{S^2}$ map is a natural generalization of the usual determinant.

It is natural to ask if this construction can be generalized to higher dimensional spheres $S^n$. Such a result would hinge on a generalization of Lemma \ref{keylemma}, for example when $n=3$ we would need to investigate homogeneous sphere-free face $d$-partitions of the complete graph $K_{3d}$
\end{remark}

\appendix

\maketitle

\section{Homogeneous, cycle-free $3$-partitions of $K_6$ }
\subsection{Equivalence classes under the $S_6\times S_3$ action}
Using MATLAB we established  that there are $\num{756756}$ homogeneous partitions of $K_6$, out of which $\num{66240}$ are cycle-free. In this section we present a system of representatives for the equivalence relation determined by the action of  $S_6\times S_3$ on $\mathcal{P}_3^{h,cf}(K_{6})$. For each partition we compute its stabilizer, and so we can find the number of elements in that equivalence class. We also give the values of the map $\varepsilon^{S^2}$ from Theorem \ref{th1A}.

If $P=(\Gamma_1,\Gamma_2,\Gamma_3)$ is a homogeneous cycle-free partition of $K_6$, the graphs $\Gamma_i$ must be of type $I_6$, $E_6$, $Y_6$, or $H_6$. We organize the partitions according to their type. For example, we say that $P$ is of type $(I_6,I_6,E_6)$ if $\Gamma_1$ and $\Gamma_2$ are of type $I_6$ and $\Gamma_3$ is of type $E_6$. Notice that there are some types which cannot be realized as a partition (for example  there are no partitions of type $(Y_6,Y_6,Y_6)$.

We give details of the computations for the cases $(I_6, I_6, I_6)$, $(I_6, I_6, E_6)$, and $(I_6, Y_6, E_6)$. The reader could use these ideas to analyze all 19 equivalence classes. One should keep in mind that these results were obtained computationally, so not all details have an elegant theoretical explanation.

\subsubsection{Type $(I_6, I_6, I_6)$}

We present three partitions of the type $(I_6, I_6, I_6)$, calculate their stabilizer  and show that they  generate 5760 homogeneous cycle-free partitions.

Take $P_1=(\Gamma_1^{(1)},\Gamma_2^{(1)},\Gamma_3^{(1)})$ as in Figure \ref{fig11}.
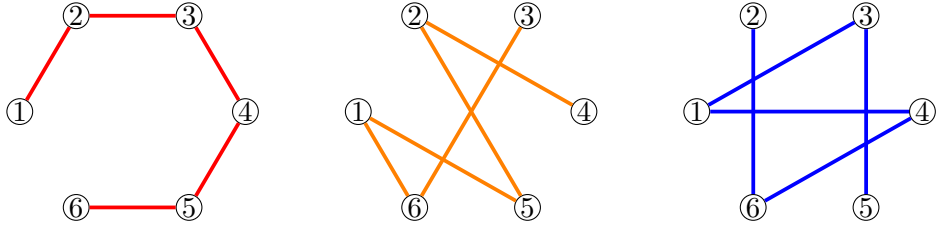
\begin{figure}[h!]
\centering
\begin{tikzpicture}
  [scale=1.5,auto=left,every node/.style={shape = circle, draw, fill = white,minimum size = 1pt, inner sep=0.3pt}]
  \node (n1) at (0,0) {1};
  \node (n2) at (0.5,0.85)  {2};
  \node (n3) at (1.5,0.85)  {3};
  \node (n4) at (2,0)  {4};
	\node (n5) at (1.5,-0.85)  {5};
  \node (n6) at (0.5,-0.85)  {6};
  \foreach \from/\to in {n1/n2,n2/n3,n3/n4,n4/n5,n5/n6}
    \draw[line width=0.5mm,red]  (\from) -- (\to);	
    \node (n11) at (3,0) {1};
  \node (n21) at (3.5,0.85)  {2};
  \node (n31) at (4.5,0.85)  {3};
  \node (n41) at (5,0)  {4};
	\node (n51) at (4.5,-0.85)  {5};
  \node (n61) at (3.5,-0.85)  {6};
  \foreach \from/\to in {n11/n51,n11/n61,n21/n41,n21/n51,n31/n61}
  \draw[line width=0.5mm,orange]  (\from) -- (\to);	
	
	\node (n12) at (6,0) {1};
  \node (n22) at (6.5,0.85)  {2};
  \node (n32) at (7.5,0.85)  {3};
  \node (n42) at (8,0)  {4};
	\node (n52) at (7.5,-0.85)  {5};
  \node (n62) at (6.5,-0.85)  {6};
  \foreach \from/\to in {n12/n32,n12/n42,n22/n62,n32/n52,n42/n62}
    \draw[line width=0.5mm,blue]  (\from) -- (\to);	

\end{tikzpicture}
\caption{$P_1=(\Gamma_1^{(1)},\Gamma_2^{(1)},\Gamma_3^{(1)})$  of type $(I_6,I_6,I_6)$ } \label{fig11}
\end{figure}

We have that $$Stab_{S_6\times S_3}((\Gamma_1^{(1)},\Gamma_2^{(1)},\Gamma_3^{(1)}))=\{e_{S_6}\times e_{S_3}\}.$$

Indeed, if $(\sigma,\tau)P_1=P_1$, then $\sigma\Gamma_i^{(1)}=\Gamma_{\tau(i)}^{(1)}$, for all $1\leq i\leq 3$.
This implies that if $\tau(1)=1$, then $\sigma\in\{e_{S_6}, (1,6)(2,5)(3,4)\}$; if $\tau(2)=1$ we have
$\sigma\in\{(1,3)(2,6,4,5),(1,4)(3,5,6)\}$; and if $\tau(3)=1$, then $\sigma\in \{(1,2,6,5,3,4),(1,5,6,2,3)\}$. Direct computations show that out of the twelve possibilities, the only element in $ Stab_{S_6\times S_3}((\Gamma_1^{(1)},\Gamma_2^{(1)},\Gamma_3^{(1)}))$ is $e_{S_6}\times e_{S_3}$. We get that the orbit of $P_1$ has $4320$ elements.

Take $P_2=(\Gamma_1^{(2)},\Gamma_2^{(2)},\Gamma_3^{(2)})$ as in Figure \ref{III2}.

\begin{figure}[h]
\centering
\begin{tikzpicture}
  [scale=1.5,auto=left,every node/.style={shape = circle, draw, fill = white,minimum size = 1pt, inner sep=0.3pt}]
  \node (n1) at (0,0) {1};
  \node (n2) at (0.5,0.85)  {2};
  \node (n3) at (1.5,0.85)  {3};
  \node (n4) at (2,0)  {4};
	\node (n5) at (1.5,-0.85)  {5};
  \node (n6) at (0.5,-0.85)  {6};
  \foreach \from/\to in {n1/n2,n2/n3,n3/n4,n4/n5,n5/n6}
    \draw[line width=0.5mm,red]  (\from) -- (\to);	
    \node (n11) at (3,0) {1};
  \node (n21) at (3.5,0.85)  {2};
  \node (n31) at (4.5,0.85)  {3};
  \node (n41) at (5,0)  {4};
	\node (n51) at (4.5,-0.85)  {5};
  \node (n61) at (3.5,-0.85)  {6};
  \foreach \from/\to in {n11/n41,n11/n61,n21/n41,n31/n51,n31/n61}
  \draw[line width=0.5mm,orange]  (\from) -- (\to);	
	
	\node (n12) at (6,0) {1};
  \node (n22) at (6.5,0.85)  {2};
  \node (n32) at (7.5,0.85)  {3};
  \node (n42) at (8,0)  {4};
	\node (n52) at (7.5,-0.85)  {5};
  \node (n62) at (6.5,-0.85)  {6};
  \foreach \from/\to in {n12/n32,n12/n52,n22/n52,n22/n62,n42/n62}
    \draw[line width=0.5mm,blue]  (\from) -- (\to);	

\end{tikzpicture}
\caption{$P_2=(\Gamma_1^{(2)},\Gamma_2^{(2)},\Gamma_3^{(2)})$  of type $(I_6,I_6,I_6)$ } \label{III2}
\end{figure}
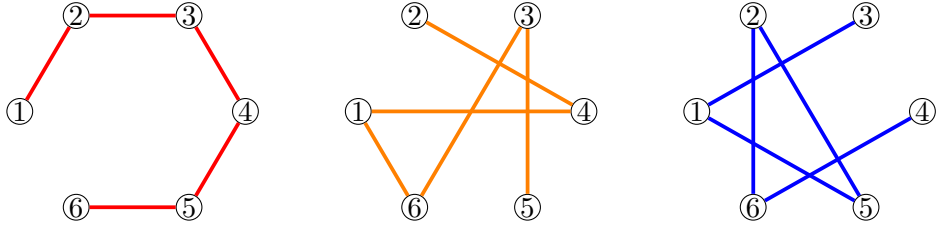

If $(\sigma,\tau)P_2=P_2$, then $\sigma\Gamma_i^{(2)}=\Gamma_{\tau(i)}^{(2)}$, for all $1\leq i\leq 3$. Again we have three possibilities, if $\tau(1)=1,$ then $\sigma\in\{e_{S_6}, (1,6)(2,5)(3,4)\}$; if $\tau(1)=2$, then $\sigma\in\{(1,2,4,6,5,3), (1,5,4)(2,3,6)\}$;  if $\tau(3)=1$, then $\sigma\in\{(1,3,5,6,4,2), (1,4,5)(2,6,3)\}$. In this case, we get
\begin{eqnarray*} &Stab_{S_6\times S_3}((\Gamma_1^{(2)},\Gamma_2^{(2)},\Gamma_3^{(2)}))=\{e_{S_6}\times e_{S_3}, (1,6)(2,5)(3,4)\times e_{S_3}, (1,2,4,6,5,3)\times(1,2,3),&\\
& (1,5,4)(2,3,6)\times (1,2,3), (1,3,5,6,4,2) \times (1,3,2), (1,4,5)(2,6,3)\times (1,3,2)\},&
\end{eqnarray*}
so the orbit of $P_2$ has $720$ elements.

Take  $P_3=(\Gamma_1^{(3)},\Gamma_2^{(3)},\Gamma_3^{(3)})$ as in Figure \ref{III3}.

\begin{figure}[h!]
\centering
\begin{tikzpicture}
  [scale=1.5,auto=left,every node/.style={shape = circle, draw, fill = white,minimum size = 1pt, inner sep=0.3pt}]
  \node (n1) at (0,0) {1};
  \node (n2) at (0.5,0.85)  {2};
  \node (n3) at (1.5,0.85)  {3};
  \node (n4) at (2,0)  {4};
	\node (n5) at (1.5,-0.85)  {5};
  \node (n6) at (0.5,-0.85)  {6};
  \foreach \from/\to in {n1/n2,n2/n3,n3/n4,n4/n5,n5/n6}
    \draw[line width=0.5mm,red]  (\from) -- (\to);	
    \node (n11) at (3,0) {1};
  \node (n21) at (3.5,0.85)  {2};
  \node (n31) at (4.5,0.85)  {3};
  \node (n41) at (5,0)  {4};
	\node (n51) at (4.5,-0.85)  {5};
  \node (n61) at (3.5,-0.85)  {6};
  \foreach \from/\to in {n11/n31,n11/n61,n21/n41,n31/n51,n41/n61}
  \draw[line width=0.5mm,orange]  (\from) -- (\to);	
	
	\node (n12) at (6,0) {1};
  \node (n22) at (6.5,0.85)  {2};
  \node (n32) at (7.5,0.85)  {3};
  \node (n42) at (8,0)  {4};
	\node (n52) at (7.5,-0.85)  {5};
  \node (n62) at (6.5,-0.85)  {6};
  \foreach \from/\to in {n12/n42,n12/n52,n22/n52,n22/n62,n32/n62}
    \draw[line width=0.5mm,blue]  (\from) -- (\to);	

\end{tikzpicture}
\caption{$P_3=(\Gamma_1^{(3)},\Gamma_2^{(3)},\Gamma_3^{(3)})$  of type $(I_6,I_6,I_6)$ } \label{III3}
\end{figure}
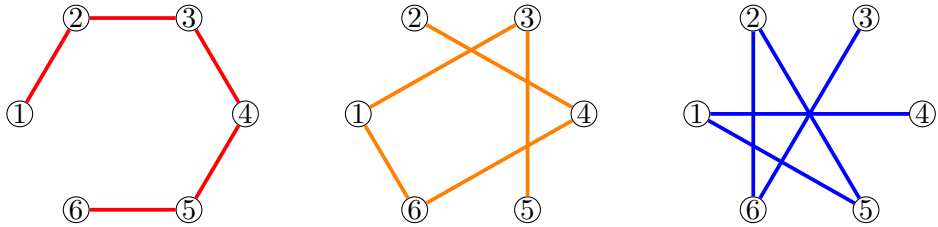

If $\tau(1)=1,$ then $\sigma\in\{e_{S_6}, (1,6)(2,5)(3,4)\}$; if $\tau(2)=1$, we have $\sigma\in\{(1,5,4,6,2,3), (1,2,4)(3,6,5)\}$;  if $\tau(3)=1$, then $\sigma\in\{(1,3,2,6,4,5), (1,4,2)(3,5,6)\}$. In this case, we get
\begin{eqnarray*} &Stab_{S_6\times S_3}((\Gamma_1^{(3)},\Gamma_2^{(3)},\Gamma_3^{(3)}))=\{e_{S_6}\times e_{S_3}, (1,6)(2,5)(3,4)\times e_{S_3}, (1,5,4,6,2,3)\times(1,2,3),& \\&(1,2,4)(3,6,5)\times (1,2,3), (1,3,2,6,4,5) \times (1,3,2), (1,4,2)(3,5,6)\times (1,3,2)\},&
\end{eqnarray*}
so the orbit of $P_3$ has $720$ elements. In summary, with the case $(I_6, I_6, I_6)$, we can generate 5760 homogeneous cycle-free partitions.

\subsubsection{Type $(I_6, I_6, E_6)$}

Take $P_4=(\Gamma_1^{(4)},\Gamma_2^{(4)},\Gamma_3^{(4)})$ as in Figure \ref{IIE1}.\\

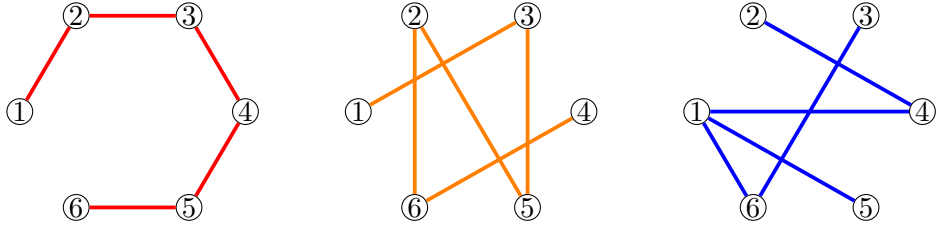
\begin{figure}[h]
\centering
\begin{tikzpicture}
  [scale=1.5,auto=left,every node/.style={shape = circle, draw, fill = white,minimum size = 1pt, inner sep=0.3pt}]
  \node (n1) at (0,0) {1};
  \node (n2) at (0.5,0.85)  {2};
  \node (n3) at (1.5,0.85)  {3};
  \node (n4) at (2,0)  {4};
	\node (n5) at (1.5,-0.85)  {5};
  \node (n6) at (0.5,-0.85)  {6};
  \foreach \from/\to in {n1/n2,n2/n3,n3/n4,n4/n5,n5/n6}
    \draw[line width=0.5mm,red]  (\from) -- (\to);	
    \node (n11) at (3,0) {1};
  \node (n21) at (3.5,0.85)  {2};
  \node (n31) at (4.5,0.85)  {3};
  \node (n41) at (5,0)  {4};
	\node (n51) at (4.5,-0.85)  {5};
  \node (n61) at (3.5,-0.85)  {6};
  \foreach \from/\to in {n11/n31,n31/n51,n51/n21,n21/n61,n41/n61}
  \draw[line width=0.5mm,orange]  (\from) -- (\to);	
	
	\node (n12) at (6,0) {1};
  \node (n22) at (6.5,0.85)  {2};
  \node (n32) at (7.5,0.85)  {3};
  \node (n42) at (8,0)  {4};
	\node (n52) at (7.5,-0.85)  {5};
  \node (n62) at (6.5,-0.85)  {6};
  \foreach \from/\to in {n12/n42,n12/n52,n12/n62,n22/n42,n32/n62}
    \draw[line width=0.5mm,blue]  (\from) -- (\to);	

\end{tikzpicture}
\caption{$P_4=(\Gamma_1^{(4)},\Gamma_2^{(4)},\Gamma_3^{(4)})$  of type $(I_6,I_6,E_6)$} \label{IIE1}
\end{figure}

If $\tau(1)=1$, then $\sigma\in\{e_{S_6}, (1,6)(2,5)(3,4)\}$; if $\tau(2)=1$, we have $\sigma\in\{(2,3,5,6), (1,5,3,4,6)\}$. Since $I$ and $E$ are different types, $\tau(3)=1$ is not possible. A direct computation shows that
$$Stab_{S_6\times S_3}((\Gamma_1^{(4)},\Gamma_2^{(4)},\Gamma_3^{(4)}))=\{e_{S_6}\times e_{S_3}\}.$$
This implies that the orbit of $P_4$ has $4320$ elements.

Take $P_5=(\Gamma_1^{(5)},\Gamma_2^{(5)},\Gamma_3^{(5)})$ as in Figure \ref{IIE2}.

\begin{figure}[h]
\centering
\begin{tikzpicture}
  [scale=1.5,auto=left,every node/.style={shape = circle, draw, fill = white,minimum size = 1pt, inner sep=0.3pt}]
  \node (n1) at (0,0) {1};
  \node (n2) at (0.5,0.85)  {2};
  \node (n3) at (1.5,0.85)  {3};
  \node (n4) at (2,0)  {4};
	\node (n5) at (1.5,-0.85)  {5};
  \node (n6) at (0.5,-0.85)  {6};
  \foreach \from/\to in {n1/n2,n2/n3,n3/n4,n4/n5,n5/n6}
    \draw[line width=0.5mm,red]  (\from) -- (\to);	
    \node (n11) at (3,0) {1};
  \node (n21) at (3.5,0.85)  {2};
  \node (n31) at (4.5,0.85)  {3};
  \node (n41) at (5,0)  {4};
	\node (n51) at (4.5,-0.85)  {5};
  \node (n61) at (3.5,-0.85)  {6};
  \foreach \from/\to in {n11/n31,n21/n41,n21/n51,n31/n61,n41/n61}
  \draw[line width=0.5mm,orange]  (\from) -- (\to);	
	
	\node (n12) at (6,0) {1};
  \node (n22) at (6.5,0.85)  {2};
  \node (n32) at (7.5,0.85)  {3};
  \node (n42) at (8,0)  {4};
	\node (n52) at (7.5,-0.85)  {5};
  \node (n62) at (6.5,-0.85)  {6};
  \foreach \from/\to in {n12/n42,n12/n52,n12/n62,n22/n62,n32/n52}
    \draw[line width=0.5mm,blue]  (\from) -- (\to);	

\end{tikzpicture}
\caption{$P_5=(\Gamma_1^{(5)},\Gamma_2^{(4)},\Gamma_3^{(5)})$  of type $(I_6,I_6,E_6)$} \label{IIE2}
\end{figure}

If $\tau(1)=1$, then $\sigma\in\{e_{S_6}, (1,6)(2,5)(3,4)\}$; if $\tau(2)=1$, we have $\sigma\in\{(1,4,5,3,2,6), (2,3,5,6,4)\}.$ It follows easily that
\begin{eqnarray*} Stab_{S_6\times S_3}((\Gamma_1^{(5)},\Gamma_2^{(5)},\Gamma_3^{(5)}))=\{e_{S_6}\times e_{S_3}\}.\end{eqnarray*}
This implies that the orbit of $P_5$ has $4320$ elements. In summary, with the case $(I_6, I_6, E_6)$, we generate $8640$ homogeneous cycle-free partitions.

\subsubsection{Type $(I_6, E_6, Y_6)$}
Take $P_6=(\Gamma_1^{(6)},\Gamma_2^{(6)},\Gamma_3^{(6)})$ as in Figure \ref{IEY1}.
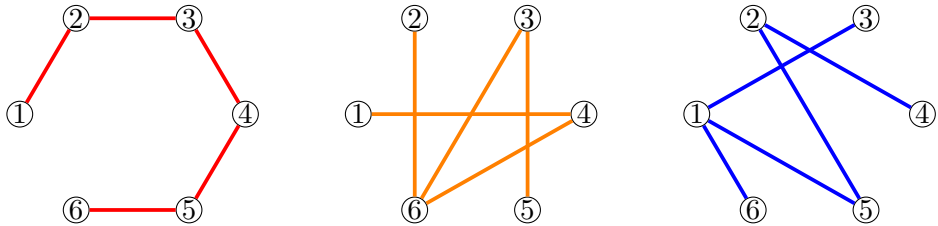
\begin{figure}[h!]
	\centering
	\begin{tikzpicture}
		[scale=1.5,auto=left,every node/.style={shape = circle, draw, fill = white,minimum size = 1pt, inner sep=0.3pt}]
		\node (n1) at (0,0) {1};
		\node (n2) at (0.5,0.85)  {2};
		\node (n3) at (1.5,0.85)  {3};
		\node (n4) at (2,0)  {4};
		\node (n5) at (1.5,-0.85)  {5};
		\node (n6) at (0.5,-0.85)  {6};
		\foreach \from/\to in {n1/n2,n2/n3,n3/n4,n4/n5,n5/n6}
		\draw[line width=0.5mm,red]  (\from) -- (\to);	
		\node (n11) at (3,0) {1};
		\node (n21) at (3.5,0.85)  {2};
		\node (n31) at (4.5,0.85)  {3};
		\node (n41) at (5,0)  {4};
		\node (n51) at (4.5,-0.85)  {5};
		\node (n61) at (3.5,-0.85)  {6};
		\foreach \from/\to in {n11/n41,n21/n61,n31/n51,n31/n61,n41/n61}
		\draw[line width=0.5mm,orange]  (\from) -- (\to);	
		
		\node (n12) at (6,0) {1};
		\node (n22) at (6.5,0.85)  {2};
		\node (n32) at (7.5,0.85)  {3};
		\node (n42) at (8,0)  {4};
		\node (n52) at (7.5,-0.85)  {5};
		\node (n62) at (6.5,-0.85)  {6};
		\foreach \from/\to in {n12/n32,n12/n52,n12/n62,n22/n42,n22/n52}
		\draw[line width=0.5mm,blue]  (\from) -- (\to);	
		
	\end{tikzpicture}
	\caption{$P_6=(\Gamma_1^{(6)},\Gamma_2^{(6)},\Gamma_3^{(6)})$  of type $(I_6,E_6,Y_6)$ } \label{IEY1}
\end{figure}
\\
In this case, all three graphs are of different type, so if $(\sigma,\tau)$ is in the stabilizer, then $\tau$ is the identity. So, for $\sigma$ to stabilize $\Gamma_1^{(6)}$, we necessarily have that $\sigma=e$ or $\sigma=(1,6)(2,5)(3,4)$. In the latter case, $\sigma$ sends the edge $(2,6)$ to $(1,5)$ and so it will not send $\Gamma_2^{(6)}$ to $\Gamma_2^{(6)}$.We get that 
\begin{eqnarray*}
Stab_{S_6\times S_3}((\Gamma_1^{(6)},\Gamma_2^{(6)},\Gamma_3^{(6)}))=\{e_{S_6}\times e_{S_3}\},
\end{eqnarray*}
so the orbit of $P_6$ has 4320 elements. With the case $(I_6,E_6,Y_6)$, we generate 4320 homogeneous cycle-free partitions.

\subsubsection{Type $(I_6, Y_6, E_6)$}
Take $P_{7}=(\Gamma_1^{(7)},\Gamma_2^{(7)},\Gamma_3^{(7)})$ as in Figure \ref{IYE1}.
\begin{figure}[h!]
	\centering
	\begin{tikzpicture}
		[scale=1.5,auto=left,every node/.style={shape = circle, draw, fill = white,minimum size = 1pt, inner sep=0.3pt}]
		\node (n1) at (0,0) {1};
		\node (n2) at (0.5,0.85)  {2};
		\node (n3) at (1.5,0.85)  {3};
		\node (n4) at (2,0)  {4};
		\node (n5) at (1.5,-0.85)  {5};
		\node (n6) at (0.5,-0.85)  {6};
		\foreach \from/\to in {n1/n2,n2/n3,n3/n4,n4/n5,n5/n6}
		\draw[line width=0.5mm,red]  (\from) -- (\to);	
		\node (n11) at (3,0) {1};
		\node (n21) at (3.5,0.85)  {2};
		\node (n31) at (4.5,0.85)  {3};
		\node (n41) at (5,0)  {4};
		\node (n51) at (4.5,-0.85)  {5};
		\node (n61) at (3.5,-0.85)  {6};
		\foreach \from/\to in {n11/n51,n21/n51,n21/n61,n31/n61,n41/n61}
		\draw[line width=0.5mm,orange]  (\from) -- (\to);	
		
		\node (n12) at (6,0) {1};
		\node (n22) at (6.5,0.85)  {2};
		\node (n32) at (7.5,0.85)  {3};
		\node (n42) at (8,0)  {4};
		\node (n52) at (7.5,-0.85)  {5};
		\node (n62) at (6.5,-0.85)  {6};
		\foreach \from/\to in {n12/n32,n12/n42,n12/n62,n22/n42,n32/n52}
		\draw[line width=0.5mm,blue]  (\from) -- (\to);	
		
	\end{tikzpicture}
	\caption{$P_{7}=(\Gamma_1^{(7)},\Gamma_2^{(7)},\Gamma_3^{(7)})$  of type $(I_6,Y_6,E_6)$ } \label{IYE1}
\end{figure}
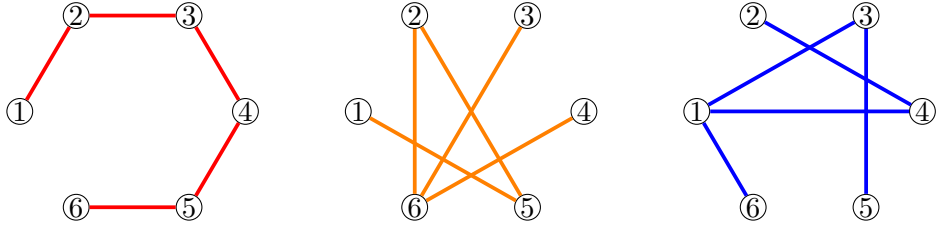
\\
We get that 
\begin{eqnarray*}
Stab_{S_6\times S_3}((\Gamma_1^{(7)},\Gamma_2^{(7)},\Gamma_3^{(7)}))=\{e_{S_6}\times e_{S_3}\},
\end{eqnarray*} 
so the orbit of $P_{7}$ has 4320 elements. With the case $(I_6,Y_6,E_6)$, we generate 4320 homogeneous cycle-free partitions.

\subsubsection{Type $(I_6, I_6, Y_6)$}
Take $P_8=(\Gamma_1^{(8)},\Gamma_2^{(8)},\Gamma_3^{(8)})$ as in Figure \ref{IIY1}.
\begin{figure}[h!]
	\centering
	\begin{tikzpicture}
		[scale=1.5,auto=left,every node/.style={shape = circle, draw, fill = white,minimum size = 1pt, inner sep=0.3pt}]
		\node (n1) at (0,0) {1};
		\node (n2) at (0.5,0.85)  {2};
		\node (n3) at (1.5,0.85)  {3};
		\node (n4) at (2,0)  {4};
		\node (n5) at (1.5,-0.85)  {5};
		\node (n6) at (0.5,-0.85)  {6};
		\foreach \from/\to in {n1/n2,n2/n3,n3/n4,n4/n5,n5/n6}
		\draw[line width=0.5mm,red]  (\from) -- (\to);	
		\node (n11) at (3,0) {1};
		\node (n21) at (3.5,0.85)  {2};
		\node (n31) at (4.5,0.85)  {3};
		\node (n41) at (5,0)  {4};
		\node (n51) at (4.5,-0.85)  {5};
		\node (n61) at (3.5,-0.85)  {6};
		\foreach \from/\to in {n11/n61,n21/n41,n21/n51,n31/n51,n41/n61}
		\draw[line width=0.5mm,orange]  (\from) -- (\to);	
		
		\node (n12) at (6,0) {1};
		\node (n22) at (6.5,0.85)  {2};
		\node (n32) at (7.5,0.85)  {3};
		\node (n42) at (8,0)  {4};
		\node (n52) at (7.5,-0.85)  {5};
		\node (n62) at (6.5,-0.85)  {6};
		\foreach \from/\to in {n12/n32,n12/n42,n12/n52,n22/n62,n32/n62}
		\draw[line width=0.5mm,blue]  (\from) -- (\to);	
		
	\end{tikzpicture}
	\caption{$P_8=(\Gamma_1^{(8)},\Gamma_2^{(8)},\Gamma_3^{(8)})$  of type $(I_6,I_6,Y_6)$  } \label{IIY1}
	\end{figure}
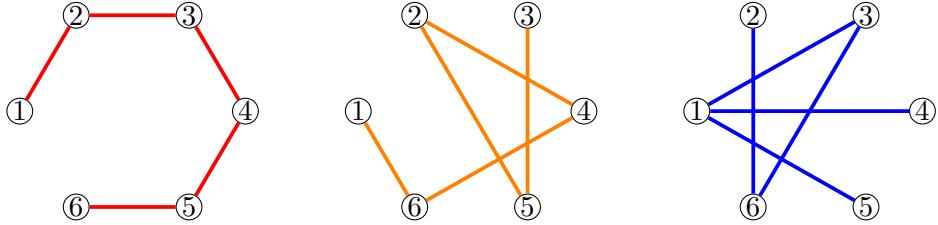

We have that \begin{eqnarray*}Stab_{S_6\times S_3}((\Gamma_1^{(8)},\Gamma_2^{(8)},\Gamma_3^{(8)}))=\{e_{S_6}\times e_{S_3}\},
\end{eqnarray*}
so the orbit of $P_8$  has 4320 elements.

Take $P_9=(\Gamma_1^{(9)},\Gamma_2^{(9)},\Gamma_3^{(9)})$ as in Figure \ref{IIY2}.
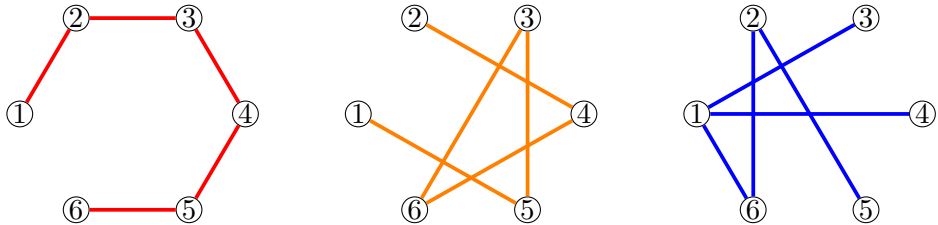
\begin{figure}[h!]
	\centering
	\begin{tikzpicture}
		[scale=1.5,auto=left,every node/.style={shape = circle, draw, fill = white,minimum size = 1pt, inner sep=0.3pt}]
		\node (n1) at (0,0) {1};
		\node (n2) at (0.5,0.85)  {2};
		\node (n3) at (1.5,0.85)  {3};
		\node (n4) at (2,0)  {4};
		\node (n5) at (1.5,-0.85)  {5};
		\node (n6) at (0.5,-0.85)  {6};
		\foreach \from/\to in {n1/n2,n2/n3,n3/n4,n4/n5,n5/n6}
		\draw[line width=0.5mm,red]  (\from) -- (\to);	
		\node (n11) at (3,0) {1};
		\node (n21) at (3.5,0.85)  {2};
		\node (n31) at (4.5,0.85)  {3};
		\node (n41) at (5,0)  {4};
		\node (n51) at (4.5,-0.85)  {5};
		\node (n61) at (3.5,-0.85)  {6};
		\foreach \from/\to in {n11/n51,n21/n41,n31/n51,n31/n61,n41/n61}
		\draw[line width=0.5mm,orange]  (\from) -- (\to);	
		
		\node (n12) at (6,0) {1};
		\node (n22) at (6.5,0.85)  {2};
		\node (n32) at (7.5,0.85)  {3};
		\node (n42) at (8,0)  {4};
		\node (n52) at (7.5,-0.85)  {5};
		\node (n62) at (6.5,-0.85)  {6};
		\foreach \from/\to in {n12/n32,n12/n42,n12/n62,n22/n62,n22/n52}
		\draw[line width=0.5mm,blue]  (\from) -- (\to);	
		
	\end{tikzpicture}
	\caption{$P_9=(\Gamma_1^{(9)},\Gamma_2^{(9)},\Gamma_3^{(9)})$  of type $(I_6,I_6,Y_6)$  } \label{IIY2}
\end{figure}
\\
We have that 
\begin{eqnarray*}
Stab_{S_6\times S_3}((\Gamma_1^{(9)},\Gamma_2^{(9)},\Gamma_3^{(9)}))=\{e_{S_6}\times e_{S_3}\},
\end{eqnarray*}
thus the orbit of $P_9$  has 4320 elements. In summary, with the case $(I_6,I_6,Y_6)$, we generate 8640 homogeneous cycle-free partitions.

\subsubsection{Type $(I_6, I_6, H_6)$}
Take $P_{10}=(\Gamma_1^{(10)},\Gamma_2^{(10)},\Gamma_3^{(10)})$ as in Figure \ref{IIH1}.
\begin{figure}[h!]
	\centering
	\begin{tikzpicture}
		[scale=1.5,auto=left,every node/.style={shape = circle, draw, fill = white,minimum size = 1pt, inner sep=0.3pt}]
		\node (n1) at (0,0) {1};
		\node (n2) at (0.5,0.85)  {2};
		\node (n3) at (1.5,0.85)  {3};
		\node (n4) at (2,0)  {4};
		\node (n5) at (1.5,-0.85)  {5};
		\node (n6) at (0.5,-0.85)  {6};
		\foreach \from/\to in {n1/n2,n2/n3,n3/n4,n4/n5,n5/n6}
		\draw[line width=0.5mm,red]  (\from) -- (\to);	
		\node (n11) at (3,0) {1};
		\node (n21) at (3.5,0.85)  {2};
		\node (n31) at (4.5,0.85)  {3};
		\node (n41) at (5,0)  {4};
		\node (n51) at (4.5,-0.85)  {5};
		\node (n61) at (3.5,-0.85)  {6};
		\foreach \from/\to in {n11/n41,n21/n41,n21/n51,n31/n51,n31/n61}
		\draw[line width=0.5mm,orange]  (\from) -- (\to);	
		
		\node (n12) at (6,0) {1};
		\node (n22) at (6.5,0.85)  {2};
		\node (n32) at (7.5,0.85)  {3};
		\node (n42) at (8,0)  {4};
		\node (n52) at (7.5,-0.85)  {5};
		\node (n62) at (6.5,-0.85)  {6};
		\foreach \from/\to in {n12/n32,n12/n52,n12/n62,n22/n62,n42/n62}
		\draw[line width=0.5mm,blue]  (\from) -- (\to);	
		
	\end{tikzpicture}
	\caption{$P_{10}=(\Gamma_1^{(10)},\Gamma_2^{(10)},\Gamma_3^{(10)})$  of type $(I_6,I_6,H_6)$  } \label{IIH1}
\end{figure}
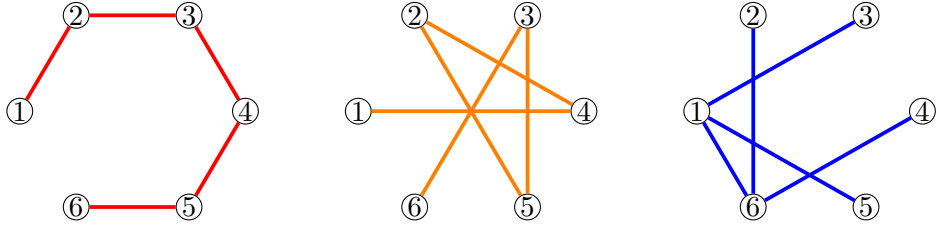
\\
We get that
\begin{eqnarray*}
Stab_{S_6\times S_3}((\Gamma_1^{(10)},\Gamma_2^{(10)},\Gamma_3^{(10)}))=\{e_{S_6}\times e_{S_3},(1,6)(2,5)(3,4)\times e_{S_3}\},
\end{eqnarray*}
so the orbit of $P_{10}$ has 2160 elements. With the case $(I_6,I_6,H_6)$, we generate 2160 homogeneous cycle-free partitions.
\\

\subsubsection{Type $(E_6, E_6, I_6)$}
Take $P_{11}=(\Gamma_1^{(11)},\Gamma_2^{(11)},\Gamma_3^{(11)})$ as in Figure \ref{EEI1}.
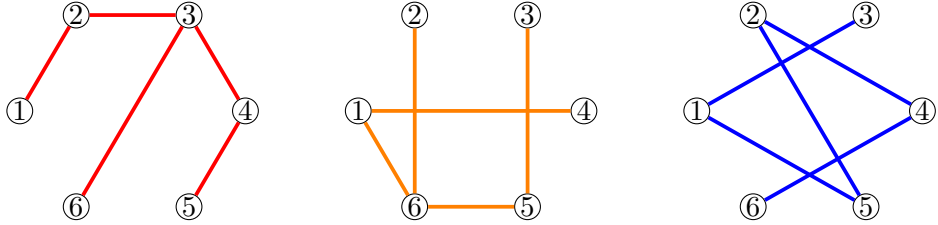
\begin{figure}[h!]
	\centering
	\begin{tikzpicture}
		[scale=1.5,auto=left,every node/.style={shape = circle, draw, fill = white,minimum size = 1pt, inner sep=0.3pt}]
		\node (n1) at (0,0) {1};
		\node (n2) at (0.5,0.85)  {2};
		\node (n3) at (1.5,0.85)  {3};
		\node (n4) at (2,0)  {4};
		\node (n5) at (1.5,-0.85)  {5};
		\node (n6) at (0.5,-0.85)  {6};
		\foreach \from/\to in {n1/n2,n2/n3,n3/n4,n3/n6,n4/n5}
		\draw[line width=0.5mm,red]  (\from) -- (\to);	
		\node (n11) at (3,0) {1};
		\node (n21) at (3.5,0.85)  {2};
		\node (n31) at (4.5,0.85)  {3};
		\node (n41) at (5,0)  {4};
		\node (n51) at (4.5,-0.85)  {5};
		\node (n61) at (3.5,-0.85)  {6};
		\foreach \from/\to in {n11/n41,n11/n61,n21/n61,n31/n51,n51/n61}
		\draw[line width=0.5mm,orange]  (\from) -- (\to);	
		
		\node (n12) at (6,0) {1};
		\node (n22) at (6.5,0.85)  {2};
		\node (n32) at (7.5,0.85)  {3};
		\node (n42) at (8,0)  {4};
		\node (n52) at (7.5,-0.85)  {5};
		\node (n62) at (6.5,-0.85)  {6};
		\foreach \from/\to in {n12/n32,n12/n52,n22/n42,n22/n52,n42/n62}
		\draw[line width=0.5mm,blue]  (\from) -- (\to);	
		
	\end{tikzpicture}
	\caption{$P_{11}=(\Gamma_1^{(11)},\Gamma_2^{(11)},\Gamma_3^{(11)})$  of type $(E_6,E_6,I_6)$ } \label{EEI1}
\end{figure}
\\
We get that 
\begin{eqnarray*}
Stab_{S_6\times S_3}((\Gamma_1^{(11)},\Gamma_2^{(11)},\Gamma_3^{(11)}))=\{e_{S_6}\times e_{S_3}\},
\end{eqnarray*}
 so the orbit of $P_{11}$ has 4320 elements. With the case $(E_6,E_6,I_6)$, we generate 4320 homogeneous cycle-free partitions.

\subsubsection{Type $(E_6, E_6, E_6)$}
Take  $P_{12}=(\Gamma_1^{(12)},\Gamma_2^{(12)},\Gamma_3^{(12)})$ as in Figure \ref{EEE1}.
\begin{figure}[h!]
	\centering
	\begin{tikzpicture}
		[scale=1.5,auto=left,every node/.style={shape = circle, draw, fill = white,minimum size = 1pt, inner sep=0.3pt}]
		\node (n1) at (0,0) {1};
		\node (n2) at (0.5,0.85)  {2};
		\node (n3) at (1.5,0.85)  {3};
		\node (n4) at (2,0)  {4};
		\node (n5) at (1.5,-0.85)  {5};
		\node (n6) at (0.5,-0.85)  {6};
		\foreach \from/\to in {n1/n2,n2/n3,n3/n4,n3/n6,n4/n5}
		\draw[line width=0.5mm,red]  (\from) -- (\to);	
		\node (n11) at (3,0) {1};
		\node (n21) at (3.5,0.85)  {2};
		\node (n31) at (4.5,0.85)  {3};
		\node (n41) at (5,0)  {4};
		\node (n51) at (4.5,-0.85)  {5};
		\node (n61) at (3.5,-0.85)  {6};
		\foreach \from/\to in {n11/n41,n11/n51,n21/n51,n21/n61,n31/n51}
		\draw[line width=0.5mm,orange]  (\from) -- (\to);	
		
		\node (n12) at (6,0) {1};
		\node (n22) at (6.5,0.85)  {2};
		\node (n32) at (7.5,0.85)  {3};
		\node (n42) at (8,0)  {4};
		\node (n52) at (7.5,-0.85)  {5};
		\node (n62) at (6.5,-0.85)  {6};
		\foreach \from/\to in {n12/n32,n12/n62,n22/n42,n42/n62,n52/n62}
		\draw[line width=0.5mm,blue]  (\from) -- (\to);	
		
	\end{tikzpicture}
	\caption{$P_{12}=(\Gamma_1^{(12)},\Gamma_2^{(12)},\Gamma_3^{(12)})$  of type $(E_6,E_6,E_6)$  } \label{EEE1}
\end{figure}
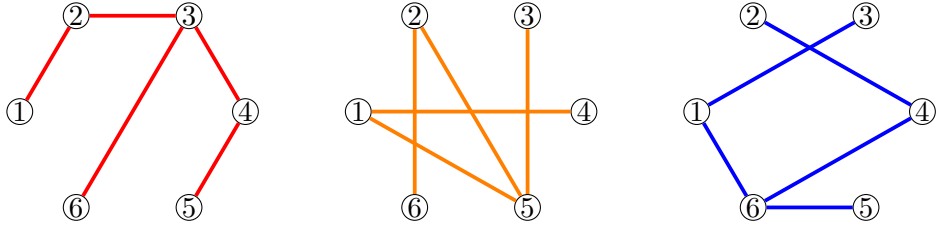
\\
We get that 
\begin{eqnarray*}
Stab_{S_6\times S_3}(P_{12})=\{e_{S_6}\times e_{S_3}, (1,4,2)(3,5,6)\times(1,2,3),(1,2,4)(3,6,5)\times(1,3,2)\},
\end{eqnarray*}
so the orbit of $P_{12}$ has 1440 elements. With the case $(E_6,E_6,E_6)$, we generate 1440 ho\-mo\-ge\-ne\-ous cycle-free partitions.

\subsubsection{Type $(E_6, E_6, Y_6)$}
Take $P_{13}=(\Gamma_1^{(13)},\Gamma_2^{(13)},\Gamma_3^{(13)})$ as in Figure \ref{EEY1}.
\begin{figure}[h!]
	\centering
	\begin{tikzpicture}
		[scale=1.5,auto=left,every node/.style={shape = circle, draw, fill = white,minimum size = 1pt, inner sep=0.3pt}]
		\node (n1) at (0,0) {1};
		\node (n2) at (0.5,0.85)  {2};
		\node (n3) at (1.5,0.85)  {3};
		\node (n4) at (2,0)  {4};
		\node (n5) at (1.5,-0.85)  {5};
		\node (n6) at (0.5,-0.85)  {6};
		\foreach \from/\to in {n1/n2,n2/n3,n3/n4,n3/n6,n4/n5}
		\draw[line width=0.5mm,red]  (\from) -- (\to);	
		\node (n11) at (3,0) {1};
		\node (n21) at (3.5,0.85)  {2};
		\node (n31) at (4.5,0.85)  {3};
		\node (n41) at (5,0)  {4};
		\node (n51) at (4.5,-0.85)  {5};
		\node (n61) at (3.5,-0.85)  {6};
		\foreach \from/\to in {n11/n61,n21/n41,n31/n51,n41/n61,n51/n61}
		\draw[line width=0.5mm,orange]  (\from) -- (\to);	
		
		\node (n12) at (6,0) {1};
		\node (n22) at (6.5,0.85)  {2};
		\node (n32) at (7.5,0.85)  {3};
		\node (n42) at (8,0)  {4};
		\node (n52) at (7.5,-0.85)  {5};
		\node (n62) at (6.5,-0.85)  {6};
		\foreach \from/\to in {n12/n32,n12/n42,n12/n52,n22/n52,n22/n62}
		\draw[line width=0.5mm,blue]  (\from) -- (\to);	
		
	\end{tikzpicture}
	\caption{$P_{13}=(\Gamma_1^{(13)},\Gamma_2^{(13)},\Gamma_3^{(13)})$  of type $(E_6,E_6,Y_6)$ } \label{EEY1}
\end{figure}
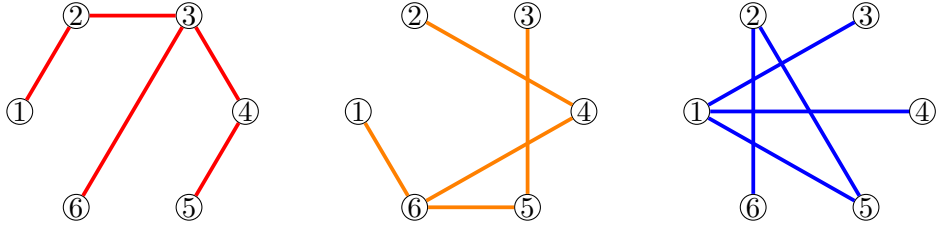
\\
We get that 
\begin{eqnarray*}
Stab_{S_6\times S_3}(P_{13})=\{e_{S_6}\times e_{S_3}\},
\end{eqnarray*}
so the orbit of $P_{13}$ has 4320 elements.

Take $P_{14}=(\Gamma_1^{(14)},\Gamma_2^{(14)},\Gamma_3^{(14)})$ as in Figure \ref{EEY2}.
\begin{figure}[h!]
	\centering
	\begin{tikzpicture}
		[scale=1.5,auto=left,every node/.style={shape = circle, draw, fill = white,minimum size = 1pt, inner sep=0.3pt}]
		\node (n1) at (0,0) {1};
		\node (n2) at (0.5,0.85)  {2};
		\node (n3) at (1.5,0.85)  {3};
		\node (n4) at (2,0)  {4};
		\node (n5) at (1.5,-0.85)  {5};
		\node (n6) at (0.5,-0.85)  {6};
		\foreach \from/\to in {n1/n2,n2/n3,n3/n4,n3/n6,n4/n5}
		\draw[line width=0.5mm,red]  (\from) -- (\to);	
		\node (n11) at (3,0) {1};
		\node (n21) at (3.5,0.85)  {2};
		\node (n31) at (4.5,0.85)  {3};
		\node (n41) at (5,0)  {4};
		\node (n51) at (4.5,-0.85)  {5};
		\node (n61) at (3.5,-0.85)  {6};
		\foreach \from/\to in {n11/n61,n21/n41,n31/n51,n21/n51,n51/n61}
		\draw[line width=0.5mm,orange]  (\from) -- (\to);	
		
		\node (n12) at (6,0) {1};
		\node (n22) at (6.5,0.85)  {2};
		\node (n32) at (7.5,0.85)  {3};
		\node (n42) at (8,0)  {4};
		\node (n52) at (7.5,-0.85)  {5};
		\node (n62) at (6.5,-0.85)  {6};
		\foreach \from/\to in {n12/n32,n12/n42,n12/n52,n42/n62,n22/n62}
		\draw[line width=0.5mm,blue]  (\from) -- (\to);	
		
	\end{tikzpicture}
	\caption{$P_{14}=(\Gamma_1^{(14)},\Gamma_2^{(14)},\Gamma_3^{(14)})$  of type $(E_6,E_6,Y_6)$ } \label{EEY2}
\end{figure}
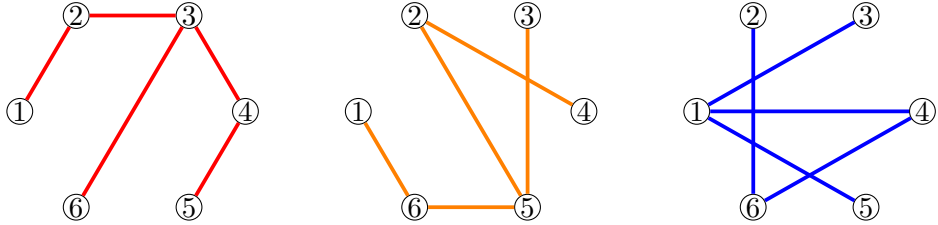
\\
We get that 
\begin{eqnarray*}
Stab_{S_6\times S_3}(P_{14})=\{e_{S_6}\times e_{S_3}\},
\end{eqnarray*}
so the orbit of $P_{14}$ has 4320 elements. In summary, with the case $(E_6,E_6,Y_6)$, we generate 8640 homogeneous cycle-free partitions.

\subsubsection{Type $(E_6, Y_6, H_6)$}
Take $P_{15}=(\Gamma_1^{(15)},\Gamma_2^{(15)},\Gamma_3^{(15)})$ as in Figure \ref{EYH1}.
\begin{figure}[h!]
	\centering
	\begin{tikzpicture}
		[scale=1.5,auto=left,every node/.style={shape = circle, draw, fill = white,minimum size = 1pt, inner sep=0.3pt}]
		\node (n1) at (0,0) {1};
		\node (n2) at (0.5,0.85)  {2};
		\node (n3) at (1.5,0.85)  {3};
		\node (n4) at (2,0)  {4};
		\node (n5) at (1.5,-0.85)  {5};
		\node (n6) at (0.5,-0.85)  {6};
		\foreach \from/\to in {n1/n2,n2/n3,n3/n4,n3/n6,n4/n5}
		\draw[line width=0.5mm,red]  (\from) -- (\to);	
		\node (n11) at (3,0) {1};
		\node (n21) at (3.5,0.85)  {2};
		\node (n31) at (4.5,0.85)  {3};
		\node (n41) at (5,0)  {4};
		\node (n51) at (4.5,-0.85)  {5};
		\node (n61) at (3.5,-0.85)  {6};
		\foreach \from/\to in {n11/n51,n21/n41,n21/n51,n31/n51,n41/n61}
		\draw[line width=0.5mm,orange]  (\from) -- (\to);	
		
		\node (n12) at (6,0) {1};
		\node (n22) at (6.5,0.85)  {2};
		\node (n32) at (7.5,0.85)  {3};
		\node (n42) at (8,0)  {4};
		\node (n52) at (7.5,-0.85)  {5};
		\node (n62) at (6.5,-0.85)  {6};
		\foreach \from/\to in {n12/n32,n12/n42,n12/n62,n22/n62,n52/n62}
		\draw[line width=0.5mm,blue]  (\from) -- (\to);	
		
	\end{tikzpicture}
	\caption{$P_{15}=(\Gamma_1^{(15)},\Gamma_2^{(15)},\Gamma_3^{(15)})$  of type $(E_6,Y_6,H_6)$ } \label{EYH1}
\end{figure}
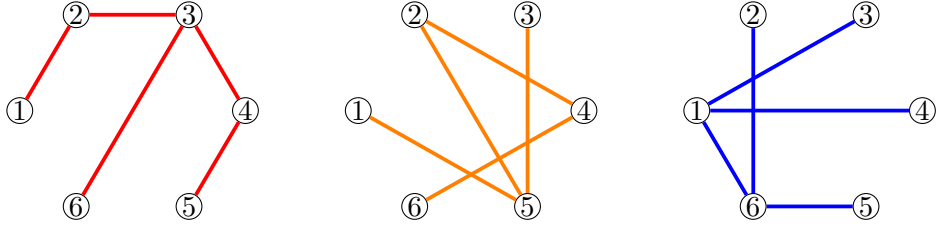
\\
We get that 
\begin{eqnarray*}
Stab_{S_6\times S_3}(P_{15})=\{e_{S_6}\times e_{S_3}\},
\end{eqnarray*}
so the orbit of $P_{15}$ has 4320 elements. With the case $(E_6,Y_6,H_6)$, we generate 4320 homogeneous cycle-free partitions.
\\

\subsubsection{Type $(E_6, H_6, Y_6)$}
Take $P_{16}=(\Gamma_1^{(16)},\Gamma_2^{(16)},\Gamma_3^{(16)})$ as in Figure \ref{EHY1}.
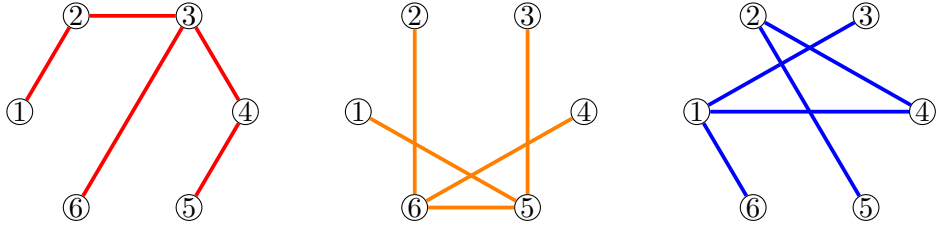
\begin{figure}[h!]
	\centering
	\begin{tikzpicture}
		[scale=1.5,auto=left,every node/.style={shape = circle, draw, fill = white,minimum size = 1pt, inner sep=0.3pt}]
		\node (n1) at (0,0) {1};
		\node (n2) at (0.5,0.85)  {2};
		\node (n3) at (1.5,0.85)  {3};
		\node (n4) at (2,0)  {4};
		\node (n5) at (1.5,-0.85)  {5};
		\node (n6) at (0.5,-0.85)  {6};
		\foreach \from/\to in {n1/n2,n2/n3,n3/n4,n3/n6,n4/n5}
		\draw[line width=0.5mm,red]  (\from) -- (\to);	
		\node (n11) at (3,0) {1};
		\node (n21) at (3.5,0.85)  {2};
		\node (n31) at (4.5,0.85)  {3};
		\node (n41) at (5,0)  {4};
		\node (n51) at (4.5,-0.85)  {5};
		\node (n61) at (3.5,-0.85)  {6};
		\foreach \from/\to in {n11/n51,n21/n61,n31/n51,n41/n61,n51/n61}
		\draw[line width=0.5mm,orange]  (\from) -- (\to);	
		
		\node (n12) at (6,0) {1};
		\node (n22) at (6.5,0.85)  {2};
		\node (n32) at (7.5,0.85)  {3};
		\node (n42) at (8,0)  {4};
		\node (n52) at (7.5,-0.85)  {5};
		\node (n62) at (6.5,-0.85)  {6};
		\foreach \from/\to in {n12/n32,n12/n42,n12/n62,n22/n42,n22/n52}
		\draw[line width=0.5mm,blue]  (\from) -- (\to);	
		
	\end{tikzpicture}
	\caption{$P_{16}=(\Gamma_1^{(16)},\Gamma_2^{(16)},\Gamma_3^{(16)})$  of type $(E_6,H_6,Y_6)$ } \label{EHY1}
\end{figure}
\\
We get that 
\begin{eqnarray*}
Stab_{S_6\times S_3}(P_{16})=\{e_{S_6}\times e_{S_3}\},
\end{eqnarray*}
so the orbit of $P_{16}$ has 4320 elements. With the case $(E_6,H_6,Y_6)$, we generate 4320 homogeneous cycle-free partitions.

\subsubsection{Type $(Y_6, Y_6, I_6)$}
Take $P_{17}=(\Gamma_1^{(17)},\Gamma_2^{(17)},\Gamma_3^{(17)})$ as in Figure \ref{YYI1}.
\begin{figure}[h!]
	\centering
	\begin{tikzpicture}
		[scale=1.5,auto=left,every node/.style={shape = circle, draw, fill = white,minimum size = 1pt, inner sep=0.3pt}]
		\node (n1) at (0,0) {1};
		\node (n2) at (0.5,0.85)  {2};
		\node (n3) at (1.5,0.85)  {3};
		\node (n4) at (2,0)  {4};
		\node (n5) at (1.5,-0.85)  {5};
		\node (n6) at (0.5,-0.85)  {6};
		\foreach \from/\to in {n1/n2,n2/n3,n3/n4,n4/n5,n4/n6}
		\draw[line width=0.5mm,red]  (\from) -- (\to);	
		\node (n11) at (3,0) {1};
		\node (n21) at (3.5,0.85)  {2};
		\node (n31) at (4.5,0.85)  {3};
		\node (n41) at (5,0)  {4};
		\node (n51) at (4.5,-0.85)  {5};
		\node (n61) at (3.5,-0.85)  {6};
		\foreach \from/\to in {n11/n41,n11/n61,n21/n51,n31/n51,n51/n61}
		\draw[line width=0.5mm,orange]  (\from) -- (\to);	
		
		\node (n12) at (6,0) {1};
		\node (n22) at (6.5,0.85)  {2};
		\node (n32) at (7.5,0.85)  {3};
		\node (n42) at (8,0)  {4};
		\node (n52) at (7.5,-0.85)  {5};
		\node (n62) at (6.5,-0.85)  {6};
		\foreach \from/\to in {n12/n32,n12/n52,n22/n42,n22/n62,n32/n62}
		\draw[line width=0.5mm,blue]  (\from) -- (\to);	
		
	\end{tikzpicture}
	\caption{$P_{17}=(\Gamma_1^{(17)},\Gamma_2^{(17)},\Gamma_3^{(17)})$  of type $(Y_6,Y_6,I_6)$ } \label{YYI1}
\end{figure}
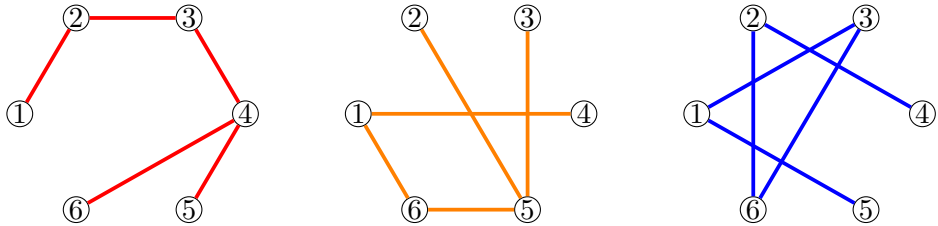
\\
We get that 
\begin{eqnarray*}
Stab_{S_6\times S_3}(P_{17})=\{e_{S_6}\times e_{S_3}\},
\end{eqnarray*}
so the orbit of $P_{17}$ has 4320 elements. With the case $(Y_6,Y_6,I_6)$, we generate 4320 homogeneous cycle-free partitions.
\\
\subsubsection{Type $(Y_6, Y_6, E_6)$} Take $P_{18}=(\Gamma_1^{(18)},\Gamma_2^{(18)},\Gamma_3^{(18)})$ as in Figure \ref{YYE1}.
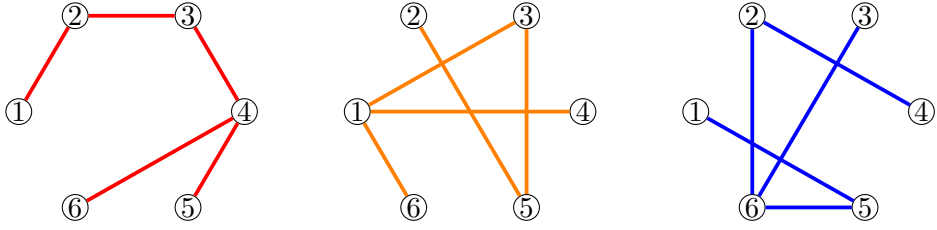
\begin{figure}[h!]
	\centering
	\begin{tikzpicture}
		[scale=1.5,auto=left,every node/.style={shape = circle, draw, fill = white,minimum size = 1pt, inner sep=0.3pt}]
		\node (n1) at (0,0) {1};
		\node (n2) at (0.5,0.85)  {2};
		\node (n3) at (1.5,0.85)  {3};
		\node (n4) at (2,0)  {4};
		\node (n5) at (1.5,-0.85)  {5};
		\node (n6) at (0.5,-0.85)  {6};
		\foreach \from/\to in {n1/n2,n2/n3,n3/n4,n4/n5,n4/n6}
		\draw[line width=0.5mm,red]  (\from) -- (\to);	
		\node (n11) at (3,0) {1};
		\node (n21) at (3.5,0.85)  {2};
		\node (n31) at (4.5,0.85)  {3};
		\node (n41) at (5,0)  {4};
		\node (n51) at (4.5,-0.85)  {5};
		\node (n61) at (3.5,-0.85)  {6};
		\foreach \from/\to in {n11/n31,n11/n41,n11/n61,n21/n51,n31/n51}
		\draw[line width=0.5mm,orange]  (\from) -- (\to);	
		
		\node (n12) at (6,0) {1};
		\node (n22) at (6.5,0.85)  {2};
		\node (n32) at (7.5,0.85)  {3};
		\node (n42) at (8,0)  {4};
		\node (n52) at (7.5,-0.85)  {5};
		\node (n62) at (6.5,-0.85)  {6};
		\foreach \from/\to in {n12/n52,n22/n42,n22/n62,n32/n62,n52/n62}
		\draw[line width=0.5mm,blue]  (\from) -- (\to);	
		
	\end{tikzpicture}
	\caption{$P_{18}=(\Gamma_1^{(18)},\Gamma_2^{(18)},\Gamma_3^{(18)})$  of type $(Y_6,Y_6,E_6)$ } \label{YYE1}
\end{figure}
\\
We get that 
\begin{eqnarray*}
Stab_{S_6\times S_3}(P_{18})=\{e_{S_6}\times e_{S_3}\},
\end{eqnarray*}
so the orbit of $P_{18}$ has 4320 elements. With the case $(Y_6,Y_6,E_6)$, we generate 4320 homogeneous cycle-free partitions.

\subsubsection{Type $(H_6, H_6, H_6)$} Take $P_{19}=(\Gamma_1^{(19)},\Gamma_2^{(19)},\Gamma_3^{(19)})$ as in Figure \ref{HHH1}.
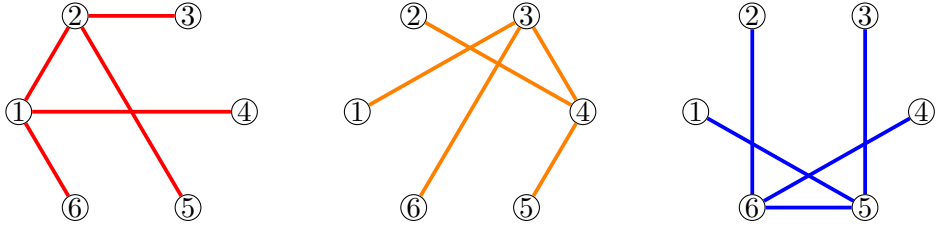
\begin{figure}[h!]
	\centering
	\begin{tikzpicture}
		[scale=1.5,auto=left,every node/.style={shape = circle, draw, fill = white,minimum size = 1pt, inner sep=0.3pt}]
		\node (n1) at (0,0) {1};
		\node (n2) at (0.5,0.85)  {2};
		\node (n3) at (1.5,0.85)  {3};
		\node (n4) at (2,0)  {4};
		\node (n5) at (1.5,-0.85)  {5};
		\node (n6) at (0.5,-0.85)  {6};
		\foreach \from/\to in {n1/n2,n1/n4,n1/n6,n2/n3,n2/n5}
		\draw[line width=0.5mm,red]  (\from) -- (\to);	
		\node (n11) at (3,0) {1};
		\node (n21) at (3.5,0.85)  {2};
		\node (n31) at (4.5,0.85)  {3};
		\node (n41) at (5,0)  {4};
		\node (n51) at (4.5,-0.85)  {5};
		\node (n61) at (3.5,-0.85)  {6};
		\foreach \from/\to in {n11/n31,n21/n41,n31/n41,n31/n61,n41/n51}
		\draw[line width=0.5mm,orange]  (\from) -- (\to);	
		
		\node (n12) at (6,0) {1};
		\node (n22) at (6.5,0.85)  {2};
		\node (n32) at (7.5,0.85)  {3};
		\node (n42) at (8,0)  {4};
		\node (n52) at (7.5,-0.85)  {5};
		\node (n62) at (6.5,-0.85)  {6};
		\foreach \from/\to in {n12/n52,n22/n62,n32/n52,n42/n62,n52/n62}
		\draw[line width=0.5mm,blue]  (\from) -- (\to);	
		
	\end{tikzpicture}
	\caption{$P_{19}=(\Gamma_1^{(19)},\Gamma_2^{(19)},\Gamma_3^{(19)})$  of type $(H_6,H_6,H_6)$ } \label{HHH1}
\end{figure}
\\
We get that
\begin{eqnarray*}
&Stab_{S_6\times S_3}(P_{19})=\{e_{S_6}\times e_{S_3},(1,2)(3,4)(5,6)\times e_{S_3},(1,3,5,2,4,6)\times(1,2,3),&\\
&(1,4,5)(2,3,6)\times(1,2,3),(1,5,4)(2,6,3)\times(1,3,2),(1,6,4,2,5,3)\times(1,3,2)\},&
\end{eqnarray*}
so the orbit of $P_{19}$ has 720 elements. With the case $(H_6,H_6,H_6)$, we generate 720 homogeneous cycle-free partitions.

\subsection{The map $\varepsilon_3^{S^2}$} If we add the above cases we accounted  for all the $\num{66240}$ homogeneous cycle-free $3$-partition of the garph $K_6$. With this classification we can give an explicit expression for $\varepsilon_3^{S^2}:\mathcal{P}_3^{h,cf}(K_{6})\to \{1,-1\}$.
\begin{remark}
For all $1\leq i\leq 19$, $\sigma\in S_6$ and $\tau\in S_3$ we have

\begin{eqnarray}
\varepsilon_3^{S^2}((\sigma \times \tau)*P_i)=
\begin{cases}
      +1& ~{\rm if} ~\tau\in A_3 \\
     -1& ~{\rm if} ~\tau\notin A_3.
   \end{cases}
\end{eqnarray}
Notice that the case $d=2$ and $d=3$ are a little bit different, in as such that $\sigma$ does not change the sign of a partition. This was also illustrated in Lemma \ref{actions2n}, where we proved that $\sigma  \rightharpoonup E_d=(sign(\sigma))^{d-1}E_d$. We expect that the same pattern will hold in higher dimensions.
\label{formulaeps}
\end{remark}

\begin{remark} The simplicity of the above formula above can be a little misleading. The hard part in defining $\varepsilon_3^{S^2}$ is choosing  the partitions $P_1,P_2,\dots,P_{19}$ such that the $\varepsilon_3^{S^2}$ is compatible with the $(x,y,z)$ involution for all $1\leq x<y<z\leq 6$ (see  Theorem \ref{th1A}).  This computation was checked using MATLAB.
\end{remark}

\begin{remark} One can see that for every $(\sigma, \tau)\in S_{2d}\times S_d$, $(\Gamma_1,\dots,\Gamma_d)\in \mathcal{P}_d^{h,cf}(K_{2d})$, and any $1\leq x<y<z\leq 2d$ we have 
$$ (\sigma, \tau)*((\Gamma_1,\dots,\Gamma_d)^{(x,y,z)})=((\sigma, \tau)*(\Gamma_1,\dots,\Gamma_d))^{(\sigma(x),\sigma(y),\sigma(z))}.$$
Combining this with the explicit description of the $19$ equivalence classes of the action of the group $S_6\times S_3$ on  $\mathcal{P}_3^{h,cf}(K_{6})$, one can check that the formula from Remark \ref{formulaeps} satisfies the statement in  Theorem \ref{th1A}. More precisely, one can show that for all $1\leq i\leq 19$ and $1\leq x<y<z\leq 6$, there exist $\sigma\in S_6$, $1\leq j\leq 19$, and an  {\bf odd} permutation   $\tau\in S_3$  such that 
$$(\Gamma_1^{(i)}, \Gamma_2^{(i)},\Gamma_3^{(i)})^{(x,y,z)}=(\sigma,\tau)*(\Gamma_1^{(j)}, \Gamma_2^{(j)},\Gamma_3^{(j)}).$$
This gives an alternative (MATLAB-free) proof for Theorem \ref{th1A}. 
\end{remark}

\bibliographystyle{amsalpha}

\begin{thebibliography}{A}







\bibitem
{cs}
S. Carolus,  and M. D. Staic,
\textit{$G$-Algebra Structure on the Higher Order Hochschild Cohomology $H^*_{S^2}(A,A)$}, to appear in Algebra Colloquium,   arXiv:1804.05096




\bibitem
{gkz}
I. M. Gelfand, M. M. Kapranov,  and A. V. Zelevinsky,    \textit{Discriminants, resultants, and multidimensional determinants}, Birkhäuser Boston,  (1994).


\bibitem
{h}
F. Harary,    \textit{Graph theory}, Addison-Wesley Publishing Co., Reading, Mass.-Menlo Park, Calif.-London,  (1969).



\bibitem
{la} J. Laubacher, \textit{Secondary Hochschild and Cyclic (co)homologies}, Ph.D. thesis,  (2017).




\bibitem
{p} T. Pirashvili, \textit{Hodge decomposition for higher order Hochschild homology}, Ann. Sci. Ecole Norm. Sup., (4) {\bf 33} (2000), 151--179.


\bibitem
{sta2} M. D. Staic, \textit{The Exterior Graded Swiss-Cheese Operad $\Lambda^{S^2}(V)$ (with an appendix by Ana Lorena Gherman and Mihai D. Staic)}, 	arXiv:2002.00520.

\bibitem
{sv} M. D. Staic, and J. Van Grinsven, \textit{A Geometric Application for the $det^{S^2}$ Map},  arXiv:2009.13641



\bibitem
{vo}
 A. A. Voronov, \textit{The Swiss-Cheese Operad},  Contemporary Mathematics, {\bf 239} (1999), 365--373.


\end{thebibliography}

\end{document}